\documentclass[12pt]{amsart}
\usepackage[margin=1in]{geometry}
\usepackage{amsmath,amssymb,mathtools}
\usepackage[T1]{fontenc}
\usepackage{lmodern}
\usepackage{xcolor}
\usepackage{microtype}
\usepackage{booktabs}
\usepackage{tikz}
\usetikzlibrary{calc,decorations.pathreplacing}
\usepackage[colorlinks,
linkcolor=blue!70!black,
citecolor=green!50!black,
urlcolor=blue!60!black
]{hyperref}
\hypersetup{pdftitle={The exact Euclidean window of Z[sqrt(14)]},
  pdfauthor={David A. Cardon}}

\theoremstyle{plain}
\newtheorem{theorem}{Theorem}[section]
\newtheorem{proposition}[theorem]{Proposition}
\newtheorem{corollary}[theorem]{Corollary}
\newtheorem{lemma}[theorem]{Lemma}
\theoremstyle{definition}

\newtheorem{question}[theorem]{Question}
\theoremstyle{remark}
\newtheorem{remark}[theorem]{Remark}

\numberwithin{equation}{section}

\newcommand{\Z}{\mathbb{Z}}
\newcommand{\Q}{\mathbb{Q}}
\newcommand{\R}{\mathbb{R}}
\newcommand{\ve}{v_{\pi}}
\title{The exact Euclidean window of $\mathbb{Z}[\sqrt{14}]$}
\author{David A. Cardon}
\address{Department of Mathematics, Brigham Young University, Provo, UT 84602}
\email{cardon@mathematics.byu.edu}
\date{October 16, 2026}

\begin{document}

\begin{abstract}
  Let $R = \mathbb{Z}[\sqrt{14}]$, let $\pi = 4 + \sqrt{14}$, and let
  $\mathfrak{p} = (2,\sqrt{14}) = \pi R$ be the prime ideal above $2$.
  Cavallar and Lemmermeyer studied the multiplicative weighted norms
  $\phi_{t}(x) = |N(x)|(t/2)^{v_{\pi}(x)}$ and asked for which weights $t$
  they are Euclidean functions on $R$. We show that the set of these
  weights, the Euclidean window of $\mathfrak{p}$, is $(\sqrt{5},T]$,
  where $T$ is the algebraic number
  \[
    T = \sqrt{\frac{8}{104051\sqrt{2} - 55617\sqrt{7}}}
    = 2.2468286690\ldots = \sqrt{5} + 0.0107606915\ldots.
  \]
  Thus the window is nonempty but not open, which answers three questions
  of Cavallar and Lemmermeyer. The proof codes the orbits of the classes
  under the fundamental unit by sequences of lattice translations, and a
  finite calculation shows that a single orbit, which joins two periodic
  orbits of points representing the class of $(3 + \sqrt{14})/2$,
  determines the endpoint $T$. A family of periodic classes, whose orbits
  are coded by the labels of closed walks in the same graph, shows that
  the endpoint is exact. We also obtain the integer-valued completely
  multiplicative Euclidean function $2^{v_{\pi}(x)}|N(x)|^{6}$ on $R$.
\end{abstract}

\maketitle \tableofcontents

\section{Introduction} \label{sec:intro}
For a squarefree integer $m > 1$, the ring of integers of $\Q(\sqrt{m})$
is Euclidean with respect to the absolute value of the norm if and only if
$m$ is one of $2$, $3$, $5$, $6$, $7$, $11$, $13$, $17$, $19$, $21$, $29$,
$33$, $37$, $41$, $57$, $73$ (see Lemmermeyer~\cite{Lemmermeyer2004}). In
this case we say that the ring is \emph{norm-Euclidean}. The list was
completed in $1952$ by Barnes and
Swinnerton-Dyer~\cite[p.~318]{BarnesSwinnertonDyer1952a}, who showed that
the ring of integers of $\Q(\sqrt{97})$ is not norm-Euclidean. Since
$\Z[\sqrt{10}]$ is not a unique factorization domain and $\Z[\sqrt{14}]$
has class number $1$, the number $14$ is the smallest positive $m$ for
which this ring is a unique factorization domain but is not
norm-Euclidean.

Thus the question for $\Z[\sqrt{14}]$ is whether some function other than
the norm is a Euclidean function on it. Weinberger~\cite{Weinberger1973}
proved, assuming the generalized Riemann hypothesis, that the ring of
integers of a number field with infinitely many units is a Euclidean
domain if and only if it is a principal ideal domain. Since the unit group
of $\Z[\sqrt{14}]$ is infinite, Weinberger's theorem predicts that
$\Z[\sqrt{14}]$ is Euclidean, although the norm is not a Euclidean
function on it. For this reason $\Z[\sqrt{14}]$ has long been the standard
test case for Euclidean functions other than the norm. Bedocchi showed
that a certain function on this ring is a Euclidean function, provided
that an unproved hypothesis holds~\cite[Hypoth\`ese~4.1]{Bedocchi1985},
and Nagata~\cite{Nagata1988,Nagata1994} also studied division in this
ring. Cavallar and Lemmermeyer describe the question whether
$\Z[\sqrt{14}]$ is Euclidean with respect to a weighted norm as one of the
most widely studied questions about weighted
norms~\cite{CavallarLemmermeyer2000}. The author showed that the larger
ring $\Z[\sqrt{14},1/2]$ is Euclidean with respect to the function that
sends a nonzero element $x$ to the odd part of the rational number
$|N(x)|$~\cite{Cardon1997}. Harper writes that $14$ was the least $m$ for
which it was unknown whether $\Q(\sqrt{m})$ is
Euclidean~\cite{Harper2004}, and he proved unconditionally that
$\Z[\sqrt{14}]$ is Euclidean. However, his proof does not show that any
weighted norm is a Euclidean function on $\Z[\sqrt{14}]$. In a note added
in proof, Cavallar and Lemmermeyer report Harper's result and state that
the question whether $\Z[\sqrt{14}]$ is Euclidean with respect to a
weighted norm remains open~\cite{CavallarLemmermeyer2000}.
Theorem~\ref{thm:main} below answers it, together with two further
questions of Cavallar and Lemmermeyer about the same weighted norms.

Let $K = \Q(\sqrt{14})$, let $R = \Z[\sqrt{14}]$, let $N$ be the norm from
$K$ to $\Q$, and let $\pi = 4 + \sqrt{14}$. For $x \in K^{\times}$ let
$\ve(x)$ be the exponent of the prime ideal
$\mathfrak{p} = \pi R = (2,\sqrt{14})$ above $2$ in the factorization of
$xR$. For $t > 0$, let $\phi_{t}(0) = 0$ and
$\phi_{t}(x) = |N(x)|(t/2)^{\ve(x)}$ for $x \in K^{\times}$. Thus
$\phi_{t}$ is the multiplicative function that agrees with $|N|$ except
that it gives the prime element $\pi \in R$ the weight $t$ instead of its
norm $2$. These are the weighted norms of Lenstra studied by Cavallar and
Lemmermeyer~\cite{CavallarLemmermeyer2000}. We determine which of them are
Euclidean (see \S\ref{sec:classes}).

\begin{theorem} \label{thm:main}
  Let
  \[
    T = \sqrt{\frac{8}{104051\sqrt{2} - 55617\sqrt{7}}}
    = 2.2468286690\ldots = \sqrt{5} + 0.0107606915\ldots.
  \]
  For $t > 0$, the multiplicative function $\phi_{t}$ is a Euclidean
  function on $R$ if and only if $\sqrt{5} < t \leq T$. Furthermore, for
  every $z \in K \setminus R$ with $\ve(z) \geq 0$ there is a $q \in R$
  such that $w = z - q$ satisfies $0 \leq \ve(w) \leq 4$ and
  $\phi_{T}(w) < 1$.
\end{theorem}

The set $w(\mathfrak{p})$ of all $t > 0$ for which $\phi_{t}$ is a
Euclidean function on $R$ is called the \emph{Euclidean window} of
$\mathfrak{p}$. Cavallar and Lemmermeyer~\cite{CavallarLemmermeyer2000}
showed that the Euclidean window of any prime ideal in a ring of integers
is either empty or an interval contained in $(1,\infty)$. In the specific
ring $R = \Z[\sqrt{14}]$ and for the prime ideal $\mathfrak{p} = \pi R$,
they proved that $w(\mathfrak{p}) \subseteq (\sqrt{5},\sqrt{7})$. They
then asked the following three questions.
\begin{enumerate}
  \item[(1)] Is $w(\mathfrak{p})$ nonempty? That is, is $\phi_{t}$ a
    Euclidean function on $R$ for some $t > 0$?
  \item[(2)] Is $w(\mathfrak{p}) = (\sqrt{5},\sqrt{7})$?
  \item[(3)] Is the Euclidean window of a prime ideal always open when the
    unit group of the ring of integers is infinite?
\end{enumerate}
Theorem~\ref{thm:main} says that $w(\mathfrak{p}) = (\sqrt{5},T]$, and so
it answers all three questions. The window is nonempty, so the answer to
question (1) is yes. Since $T < \sqrt{7}$, the window is a proper subset
of $(\sqrt{5},\sqrt{7})$, so the answer to question (2) is no. The window
contains its upper endpoint $T$, so it is not open, although the unit
group of $R$ is infinite. Thus the answer to question (3) is also no.
Since $\sqrt{5} < T < \sqrt{6}$, the theorem also shows that
$\phi_{\sqrt{6}}$ is not a Euclidean function.
Nagata~\cite{Nagata1988,Nagata1994} conjectured that it is, and he gave
computations that support the conjecture (see
Lemmermeyer~\cite{Lemmermeyer2004}).

The ring of integers of $\Q(\sqrt{69})$ is not norm-Euclidean either.
Clark~\cite{Clark1994} showed that it is Euclidean with respect to a
weighted norm. This was one of the first examples of a ring of integers
that is Euclidean with respect to a weighted norm but is not
norm-Euclidean (see Cavallar and
Lemmermeyer~\cite{CavallarLemmermeyer2000}). Cavallar and
Lemmermeyer~\cite{CavallarLemmermeyer2000} proved that the Euclidean
window of the prime ideal $(23,\sqrt{69})$ above $23$ is the open interval
$(25,\infty)$. Lemmermeyer~\cite{Lemmermeyer2004} reports that Clark's
method did not show that the Euclidean window of the prime above $2$ in
$\Z[\sqrt{14}]$ is nonempty, and that a modification of it due to
Schroeppel and Niklasch does not seem to work either.
Theorem~\ref{thm:main} shows that $\Z[\sqrt{14}]$ is also Euclidean with
respect to a weighted norm. Its window is bounded, unlike the window of
the prime above $23$ in $\Q(\sqrt{69})$, and it contains its upper
endpoint.

The proof follows the method of Cavallar and
Lemmermeyer~\cite{CavallarLemmermeyer2000} for $\Q(\sqrt{69})$, which goes
back to Barnes and Swinnerton-Dyer~\cite[Sections~4
and~8]{BarnesSwinnertonDyer1952a}, and it uses the symbolic coding of
orbits by which Ramsey studies inhomogeneous
minima~\cite[Sections~4--6]{Ramsey2025}. We work on a torus on which
multiplication by the fundamental unit $\varepsilon$ acts, and we describe
each orbit by its \emph{labels}, the lattice translations that return the
successive points of the orbit to a fundamental domain. The labels
determine the orbit through a convergent series in powers of
$\varepsilon^{-1}$ (Lemma~\ref{lem:coding}). A finite calculation on a
graph of small squares shows that every orbit which could obstruct
division at a weight near $T$ must pass through four specific squares, and
that its labels before and after this passage are forced for several steps
and are then bounded by the labels of two periodic orbits. The series then
bounds the two coordinates of the orbit at the passage by the coordinates
of one explicit orbit, which we call the critical orbit. The critical
orbit is an infinite orbit that joins the two periodic orbits, and its
coordinates are given by two explicit series. Their values determine $T$.
The bound is strict for elements of $K$ because their norms are rational,
and this proves division at the weight $T$. Finally, we construct periodic
classes whose orbits follow the two periodic orbits and the critical orbit
and return along a second orbit of the same graph, and we exclude all
their other possible reducing normalized remainders. This shows that the
upper endpoint is exact.

The expression for $T$ displays the integers $104051$ and $55617$ and the
square roots of $2$ and $7$, the two primes that ramify in $K$. We
emphasize that $T$ is not guessed in advance and then verified. The finite
calculation of \S\ref{sec:weighted} produces the labels of the critical
orbit, and the series of Lemma~\ref{lem:coding} produces $T$, and the
integers $104051$ and $55617$, from these labels. By
Proposition~\ref{prop:valuationfour}, the number $4$ in the last assertion
of Theorem~\ref{thm:main} cannot be replaced by $3$.

Taking the sixth power of a suitable weighted norm gives a simple
integer-valued Euclidean function.

\begin{corollary} \label{cor:integer}
  The function $\Phi \colon R \to \Z_{\geq 0}$ defined by
  \[
    \Phi(0) = 0,\qquad \Phi(x) = 2^{\ve(x)}|N(x)|^{6}\quad(x \neq 0)
  \]
  is a completely multiplicative Euclidean function on $R$.
\end{corollary}

Here $\Phi(x) = \phi_{t_{0}}(x)^{6}$ with $t_{0} = 2^{7/6}$, and
$\sqrt{5} < t_{0} < T$. We give the details in \S\ref{sec:division}.

\subsection{The plan of the proof} \label{sec:strategy}
For $z \in K$ we call the additive coset
$z + R = \{z + r \colon r \in R\}$ of $R$ in $K$ a \emph{class}. To
explain the term \emph{normalized remainder}, suppose that $z = a/b \in K$
with $a,b \in R$ and $b \neq 0$. The classical division problem is to find
$q,r \in R$ such that $a = bq + r$, where $q$ is the quotient and $r$ is
the remainder. Writing $r$ as $bw$, where $w \in K$, gives $a = bq + bw$.
The normalized remainder is $w = z - q$. As $q$ varies over $R$, the
normalized remainders are exactly the elements of $z + R$. Since
$\phi_{t}$ is multiplicative, finding a value of $t$ for which $\phi_{t}$
is a Euclidean function on $R$ amounts to finding some $w$ such that
$\phi_{t}(bw) = \phi_{t}(w)\phi_{t}(b) < \phi_{t}(b)$, which becomes
$\phi_{t}(w) < 1$. Thus $\phi_{t}$ is a Euclidean function if and only if
every class $z + R$ contains a normalized remainder $w$ with
$\phi_{t}(w) < 1$.

The lower endpoint comes from the class of $p = (3 + \sqrt{14})/2$. Every
normalized remainder of $p$ has absolute norm at least $5/4$ and
$\pi$-adic valuation $-2$, so its weighted norm is at least $5/t^{2}$,
with equality at $p$ itself. Raising the weight $t$ helps the classes of
negative valuation, such as $p + R$, but it hurts the classes of
nonnegative valuation, all of whose normalized remainders have nonnegative
valuation. In \S\ref{sec:prelim} we show that, for $t > \sqrt{5}$, the
function $\phi_{t}$ is a Euclidean function if and only if every
nonintegral class of nonnegative valuation contains a normalized remainder
$w$ with $\phi_{t}(w) < 1$, and that such a normalized remainder also
satisfies $\phi_{x}(w) < 1$ for every $x$ with $0 < x \leq t$. In
particular, if $\phi_{t}$ is a Euclidean function for some $t > \sqrt{5}$,
then the whole interval $(\sqrt{5},t]$ lies in the window. Therefore the
first assertion of Theorem~\ref{thm:main} is equivalent to the following
two statements.
\begin{enumerate}
  \item[(A)] For every $t > T$, some nonintegral class of nonnegative
    valuation contains no normalized remainder $w$ with $\phi_{t}(w) < 1$.
  \item[(B)] Every nonintegral class of nonnegative valuation contains a
    normalized remainder $w$ with $\phi_{T}(w) < 1$.
\end{enumerate}
Statement (A) shows that no weight larger than $T$ is in the window, and
statement (B) shows that $T$ is in the window. We prove (B), together with
the last assertion of Theorem~\ref{thm:main}, in \S\ref{sec:endpoint}, and
we prove (A) in \S\ref{sec:family}.

\subsection{Why a finite calculation suffices} \label{sec:finite}
There are infinitely many classes, and their denominators are unbounded,
so statement (B) cannot be verified one class at a time. Two observations
reduce it to a finite calculation.

First, multiplication by the fundamental unit
$\varepsilon = 15 + 4\sqrt{14}$ permutes the classes and preserves norms
and valuations. A class therefore has a normalized remainder of weighted
norm less than $1$ if and only if every class in its orbit has one. Also,
every class is periodic under this permutation, because multiplication by
$\varepsilon$ permutes the finite group $d^{-1}R/R$ for each positive
integer $d$.

Second, we can record these orbits by a finite partition, while retaining
control of the valuations. For a class of nonnegative valuation we choose
a representative $\xi$ with $\ve(\xi) \geq 5$. Then
$\ve(\xi - q) = \ve(q)$ whenever $\ve(q) \leq 4$, and the condition
$\ve(\xi) \geq 5$ is preserved by multiplication by $\varepsilon$ and by
translation by an element of $\pi^{5}R$ (Lemma~\ref{lem:periodic}). We
divide a fundamental domain of this lattice into small squares. We delete
a square when a single quotient of valuation at most $4$ gives a reducing
normalized remainder for every element of valuation at least $5$ whose
coefficient point lies in it, at a fixed rational trial weight. We record
the possible transitions between the remaining squares in a finite
directed graph, and we label each edge by the lattice translation that the
transition requires. The itinerary of a periodic orbit is then a closed
walk in the graph, much as Ramsey codes a periodic orbit by a periodic
sequence of members of his partition~\cite[Section~4]{Ramsey2025}. Our
graph may have additional walks, but every orbit that concerns us is
recorded by one of its closed walks, and the labels of the walk determine
the orbit. Thus one finite calculation applies to all the classes,
whatever their denominators.

The calculation shows that every closed walk of the graph, except for the
walks that record the classes $R$ and $p + R$ themselves, passes through
four small squares or through their images under negation, and that the
labels of a walk are forced for three steps before and for four steps
after this passage. The series of Lemma~\ref{lem:coding} converts the
labels into bounds for the two coordinates of the orbit at the passage,
and the product of the two bounds is the number $\nu = 0.6278\ldots$
of~\eqref{eq:nu}, which determines $T$ by $(T/2)^{4} = 1/\nu$. The number
$\nu$ is irrational, whereas the absolute norm of an element of $K$ is
rational. Therefore the absolute norm of the normalized remainder is
strictly smaller than $\nu$, and this gives division at the endpoint $T$,
not only at every smaller weight. The same calculation shows that every
class other than $p + R$ has a normalized remainder of absolute norm less
than $1$.

To show that no larger weight belongs to the window, we construct periodic
classes whose orbits are coded by closed walks in the same graph: many
turns around one periodic orbit, the critical orbit, many turns around the
other periodic orbit, and a return along a second orbit of the graph. The
series of Lemma~\ref{lem:coding} gives these classes in closed form, and
the orbits approach the critical orbit. We must also exclude all other
reducing normalized remainders of these classes. Multiplying a normalized
remainder by a suitable power of $\varepsilon$ reduces that question to a
fixed finite list of quotients at four positions of the orbit, for the
whole family (\S\ref{sec:smallrem}).

All of the finite calculations are exact computations in $K$. They are
carried out by a computer program, written in Python, which runs in about
half a minute and uses no floating-point arithmetic. The program is not
part of this paper. It is deposited, with its expected output, as an
ancillary file of the arXiv submission of this paper. The statements whose
proofs use the computer are Propositions~\ref{prop:finite}
and~\ref{prop:limits}, together with a few inequalities between explicit
elements of $K$, each of which is introduced by the words ``exact
comparison in $K$''. Every other step is proved in the text.

\subsection{Organization of the paper} \label{sec:organization}
In \S\ref{sec:prelim} we recall the arithmetic of $R$, reduce the division
problem to classes, and prove that the window is determined by the classes
of nonnegative valuation. In \S\ref{sec:graph} we divide the fundamental
domain into small squares, we describe the covering test and the finite
graph, and we prove that the labels of a walk determine the orbit that
follows it. In \S\ref{sec:weighted} we carry out the finite calculation,
we obtain the critical orbit and the number $T$ from it, and we prove that
every class can be divided at the weight $T$ and that every class other
than $p + R$ has a normalized remainder of absolute norm less than $1$. In
\S\ref{sec:family} we construct classes that cannot be divided at any
given weight larger than $T$, which shows that no larger weight is in the
window. In \S\ref{sec:division} we complete the proofs of
Theorem~\ref{thm:main} and Corollary~\ref{cor:integer}, and we show that
the Euclidean minimum of $R$ with respect to $\phi_{T}$ is $1$. Finally,
in \S\ref{sec:questions} we suggest questions for further study.
\section{Preliminaries} \label{sec:prelim}
In this section we recall the facts about $R$ that will be needed later,
and we reduce the division problem to classes.
Proposition~\ref{prop:shape} then shows that the window is determined by
the classes of nonnegative valuation. The section ends with the
periodicity of these classes under multiplication by the fundamental unit
and the choice of representatives of valuation at least $5$ that will
allow us to control valuations throughout a square.

\subsection{The ring \texorpdfstring{$R$}{R}} \label{sec:ring}
We identify $K$ with a subfield of $\R$, so that $\sqrt{14}$ is the positive
square root and inequalities between elements of $K$ are inequalities
between real numbers. For $x = a + b\sqrt{14} \in K$ with $a,b \in \Q$, the
\emph{conjugate} of $x$ is $\bar{x} = a - b\sqrt{14}$, and
$N(x) = x\bar{x} = a^{2} - 14b^{2}$. The element
$\varepsilon = 15 + 4\sqrt{14}$ is the fundamental unit of $R$.
The element $\pi = 4 + \sqrt{14}$ has norm $2$, and its conjugate is
$\bar{\pi} = 4 - \sqrt{14}$. A direct calculation gives
\[
  \pi^{2} = 2\varepsilon,\qquad \pi = \varepsilon\bar{\pi},\qquad
  \pi\bar{\pi} = 2.
\]
The ideal $\mathfrak{p} = (2,\sqrt{14}) = \pi R$ is the unique prime ideal
of $R$ lying above $2$. The ring $R$ has class number $1$. Since
$\mathfrak{p} = \pi R$ is principal, the exponent $\ve(x)$ of
$\mathfrak{p}$ in the factorization of $xR$, for $0 \neq x \in R$, is the
largest integer $j \geq 0$ for which $\pi^{j}$ divides $x$ in $R$. We set
$\ve(0) = \infty$. For $r \in \Q^{\times}$ let $\operatorname{ord}_{2}(r)$
be the exponent of $2$ in $r$, and set
$\operatorname{ord}_{2}(0) = \infty$. We use the usual conventions that
$\infty$ is larger than every integer and that
$a + \infty = \infty + a = \infty$ and $2 \cdot \infty = \infty$ for
$a \in \Z \cup \{\infty\}$. The next lemma collects the properties of
$\ve$ that we will use.

\begin{lemma} \label{lem:valuation}
  \mbox{}
  \begin{enumerate}
    \item[(a)] For $0 \neq x = a + b\sqrt{14} \in K$ with $a,b \in \Q$,
      \begin{equation} \label{eq:valuation}
        \ve(x) = \operatorname{ord}_{2} N(x)
        = \min\{2\operatorname{ord}_{2}(a),\,1 + 2\operatorname{ord}_{2}(b)\}.
      \end{equation}
    \item[(b)] For $x,y \in K$ we have $\ve(xy) = \ve(x) + \ve(y)$ and
      $\ve(x + y) \geq \min\{\ve(x),\ve(y)\}$, with equality if
      $\ve(x) \neq \ve(y)$.
    \item[(c)] For every integer $j \geq 0$,
      \[
        \pi^{2j}R = 2^{j}R,\qquad
        \pi^{2j + 1}R = 2^{j + 1}\Z + 2^{j}\Z\sqrt{14}.
      \]
      In particular, $\pi R = 2\Z + \Z\sqrt{14}$ and
      $\pi^{5}R = 8\Z + 4\Z\sqrt{14}$.
    \item[(d)] An element $x = a + b\sqrt{14}$ of $K$ satisfies
      $\ve(x) \geq 0$ if and only if the rational numbers $a$ and $b$ have
      odd denominators.
  \end{enumerate}
\end{lemma}
\begin{proof}
  The additivity in (b) is clear if $x = 0$ or $y = 0$. Otherwise it holds
  because $(xy)R = (xR)(yR)$ and fractional ideals factor uniquely into
  prime ideals. For $y = c + d\sqrt{14} \in R$ we have
  \[
    \frac{y}{\pi} = \frac{y\bar{\pi}}{2}
    = (2c - 7d) + \left(2d - \frac{c}{2}\right)\sqrt{14},
  \]
  so $y \in \pi R$ if and only if $c$ is even. This proves
  $\pi R = 2\Z + \Z\sqrt{14}$, and (c) follows from
  $\pi^{2} = 2\varepsilon$. It also shows that if $y \in R$ is not
  divisible by $\pi$, then $c$ is odd and $N(y) = c^{2} - 14d^{2}$ is odd.

  Now let $0 \neq x \in R$, and write $x = \pi^{j}y$ with $j = \ve(x)$ and
  $y \in R$ not divisible by $\pi$. Then $N(x) = 2^{j}N(y)$ with $N(y)$
  odd, so $\operatorname{ord}_{2}N(x) = j = \ve(x)$. Both $\ve$ and
  $\operatorname{ord}_{2} \circ N$ are additive on products, and every
  element of $K^{\times}$ is a quotient of two nonzero elements of $R$, so
  $\ve(x) = \operatorname{ord}_{2}N(x)$ for every $x \in K^{\times}$. If
  $a$ and $b$ are both nonzero, then $\operatorname{ord}_{2}(a^{2})$ is
  even and $\operatorname{ord}_{2}(14b^{2})$ is odd, so the order of their
  difference $N(x)$ is the smaller of the two. If $a = 0$ or $b = 0$, the
  second equality in~\eqref{eq:valuation} is immediate. This proves (a).

  For the inequality in (b) we may assume that $x$, $y$, and $x + y$ are
  nonzero. Let $k = \min\{\ve(x),\ve(y)\}$. By~\eqref{eq:valuation}, the
  inequality $\ve(x) \geq k$ means that
  $2\operatorname{ord}_{2}(a) \geq k$ and
  $1 + 2\operatorname{ord}_{2}(b) \geq k$, where $x = a + b\sqrt{14}$. The
  same holds for $y$, and since $\operatorname{ord}_{2}(r_{1} + r_{2}) \geq
  \min\{\operatorname{ord}_{2}(r_{1}),\operatorname{ord}_{2}(r_{2})\}$ for rational
  numbers $r_{1}$ and $r_{2}$, it holds for $x + y$. Hence
  $\ve(x + y) \geq k$. If $\ve(x) < \ve(y)$, then the inequality applied
  to $x = (x + y) - y$ gives $\ve(x) \geq \min\{\ve(x + y),\ve(y)\}$,
  which forces $\ve(x + y) \leq \ve(x)$, and so $\ve(x + y) = \ve(x)$.
  This proves (b). Finally,~\eqref{eq:valuation} shows that
  $\ve(x) \geq 0$ if and only if $\operatorname{ord}_{2}(a) \geq 0$ and
  $\operatorname{ord}_{2}(b) \geq 0$, which is (d). This proves the lemma.
\end{proof}

\subsection{Euclidean functions and classes} \label{sec:classes}
We use the definition of Cavallar and
Lemmermeyer~\cite[p.~336]{CavallarLemmermeyer2000}. A function
$f \colon R \to \R_{\geq 0}$ is a \emph{Euclidean function} on $R$ if it
satisfies the following conditions with $\kappa = 1$:
\begin{enumerate}
  \item[(i)] $f(R) \cap [0,c]$ is finite for every $c \geq 0$,
  \item[(ii)] $f(r) = 0$ if and only if $r = 0$, and
  \item[(iii)] for all $a,b \in R$ with $b \neq 0$, there is a $q \in R$
    such that $f(a - bq) < \kappa f(b)$.
\end{enumerate}
Condition (iii) is equivalent to saying that there exist $q,r \in R$ such
that $a = bq + r$ where $f(r) < \kappa f(b)$. We say that $q$ is the
\emph{quotient} and that $r = a - bq$ is the remainder of $a$ upon
division by $b$ with the quotient $q$. If $f$ satisfies (i) and (ii), the
infimum of the numbers $\kappa$ for which (iii) holds is the
\emph{Euclidean minimum} of $R$ with respect to $f$. We return to it in
Proposition~\ref{prop:minimumT} and in \S\ref{sec:questions}. Condition
(i) replaces the usual requirement that a Euclidean function take values
in the nonnegative integers. It guarantees that every nonempty set $S$ of
values of $f$ has a least element, since for any $c \in S$ the finite set
$S \cap [0,c]$ has one. Consequently repeated division terminates, because
the values of $f$ at the successive remainders strictly decrease. A
function $f$ on $K$ or on $R$ is \emph{multiplicative} if
$f(xy) =
f(x)f(y)$ for all $x$ and $y$ in its domain, without a
coprimality restriction. When we say that $\phi_{t}$ is a Euclidean
function on $R$, we mean its restriction to $R$.

For $z \in K$, recall that $z + R = \{z + r \colon r \in R\}$ is the
\emph{class} of $z$. Its elements can also be written $w = z - q$ with
$q \in R$, and we call these elements the \emph{normalized remainders} of
$z$, or of the class $z + R$, and $q$ the corresponding quotient. If
$z = a/b$ with $a,b \in R$ and $b \neq 0$, then $w = r/b$, where
$r = a - bq$ is the remainder of $a$ upon division by $b$ with the
quotient $q$. Thus a normalized remainder is an element of $K$, whereas a
remainder is an element of $R$, and the word remainder, without
qualification, always has the latter meaning. A normalized remainder $w$
is \emph{reducing at weight $t$} if $\phi_{t}(w) < 1$, and a class
\emph{obstructs division at weight $t$} if none of its normalized
remainders is reducing at weight $t$. We also call $\phi_{t}(w)$ the
\emph{weighted norm} of $w$.

\begin{lemma} \label{lem:classes}
  Let $t > 0$.
  \begin{enumerate}
    \item[(a)] The function $\phi_{t}$ is multiplicative on $K$, and it
      satisfies condition (ii).
    \item[(b)] The function $\phi_{t}$ satisfies condition (iii) with
      $\kappa = 1$ if and only if no class obstructs division at weight
      $t$.
    \item[(c)] If $t > 1$, then $\phi_{t}$ satisfies condition (i).
  \end{enumerate}
  Consequently, if $t > 1$, then $\phi_{t}$ is a Euclidean function on $R$
  if and only if no class obstructs division at weight $t$.
\end{lemma}
\begin{proof}
  (a) The norm is multiplicative, and $\ve$ is additive by
  Lemma~\ref{lem:valuation}(b). Condition (ii) holds because $N(x) \neq 0$
  for $x \neq 0$.

  (b) Suppose that no class obstructs division at weight $t$, and let
  $a,b \in R$ with $b \neq 0$. The class of $z = a/b$ contains a reducing
  normalized remainder $z - q$, and by (a),
  \[
    \phi_{t}(a - bq) = \phi_{t}(b)\,\phi_{t}(z - q) < \phi_{t}(b).
  \]
  Conversely, suppose that $\phi_{t}$ satisfies (iii) with $\kappa = 1$,
  and let $z \in K$. Write $z = a/b$ with $a,b \in R$ and $b \neq 0$.
  Condition (iii) gives $q \in R$ with $\phi_{t}(a - bq) < \phi_{t}(b)$,
  and dividing by $\phi_{t}(b) > 0$ shows that $z - q$ is a reducing
  normalized remainder.

  (c) Let $0 \neq x \in R$ and $j = \ve(x) \geq 0$.
  By~\eqref{eq:valuation}, $|N(x)| = 2^{j}k$ with $k$ a positive odd
  integer, so $\phi_{t}(x) = t^{j}k$. For $t > 1$ and $c \geq 0$, only
  finitely many numbers $t^{j}k$ with $j \geq 0$ and $k \geq 1$ lie in
  $[0,c]$. This proves the lemma.
\end{proof}

In other words, to divide $a$ by $b$ we look for an element of the class
$a/b + R$ whose weighted norm is less than $1$. We say that the class
$z + R$ has \emph{nonnegative valuation} if $\ve(z) \geq 0$. This
condition does not depend on the representative $z$, because every element
of $R$ has nonnegative valuation, so Lemma~\ref{lem:valuation}(b) gives
$\ve(z - q) \geq 0$ for every $q \in R$. The valuations of the normalized
remainders of such a class need not be equal. In contrast, if
$\ve(z) < 0$, then Lemma~\ref{lem:valuation}(b) gives
$\ve(z - q) = \ve(z)$ for every $q \in R$. Thus every class is the class
$R$, a nonintegral class of nonnegative valuation, or a class all of whose
normalized remainders have the same negative valuation.

\subsection{The class of \texorpdfstring{$p$}{p}} \label{sec:pclass}
Let $p = (3 + \sqrt{14})/2$, not to be confused with the prime ideal
$\mathfrak{p}$. The class $p + R$ obstructs division by the
ordinary absolute norm, and it determines the lower endpoint of the window.

\begin{lemma} \label{lem:pclass}
  We have $(\varepsilon - 1)p = 49 + 13\sqrt{14} \in R$, and therefore
  $(\varepsilon^{k} - 1)p \in R$ for every $k \in \Z$. Every element of
  $p + R$ has the form $(A + B\sqrt{14})/2$ with $A$ and $B$ odd integers,
  it has valuation $-2$, and its absolute norm is at least $5/4$, with
  equality at $p$. Consequently
  \begin{equation} \label{eq:lower}
    \min_{q \in R}\phi_{t}(p - q) = \frac{5}{t^{2}}\qquad(t > 0),
  \end{equation}
  and the class $p + R$ obstructs division at weight $t$ if and only if
  $t \leq \sqrt{5}$.
\end{lemma}
\begin{proof}
  The value of $(\varepsilon - 1)p$ is a direct calculation. Since
  $\varepsilon$ is a unit, $(\varepsilon^{k} - 1)p$ is a multiple of
  $(\varepsilon - 1)p$ in $R$ for every $k \in \Z$. An element $p - q$
  with $q = a + b\sqrt{14} \in R$ equals
  $((3 - 2a) + (1 - 2b)\sqrt{14})/2$, so it has the form
  $(A + B\sqrt{14})/2$ with $A$ and $B$ odd, and every such element is of
  this form. By~\eqref{eq:valuation}, its valuation is
  $\min\{-2,-1\} = -2$. Also
  \[
    4\,N(p - q) = A^{2} - 14B^{2}
  \]
  is an odd integer, and it is congruent to $1 - 14 \equiv 3 \pmod{8}$,
  because odd squares are congruent to $1$ modulo $8$. Its absolute value
  is therefore not $1$. It is not $3$ either, because the value $-3$ is
  congruent to $5$ modulo $8$, and the value $3$ would give
  $A^{2} \equiv 3 \pmod{7}$, which is impossible because the squares
  modulo $7$ are $0$, $1$, $2$, and $4$. Hence $|A^{2} - 14B^{2}| \geq 5$
  and $|N(p - q)| \geq 5/4$. Equality holds for $q = 0$, since
  $N(p) = (9 - 14)/4 = -5/4$. Finally,
  $\phi_{t}(p - q) = |N(p - q)|(t/2)^{-2} = 4|N(p - q)|/t^{2}$, which
  gives~\eqref{eq:lower}. This proves the lemma.
\end{proof}

Every class other than $p + R$ can be divided by the ordinary absolute
norm.

\begin{proposition} \label{prop:ordinary}
  Every $z \in K$ with $z \notin p + R$ has a normalized remainder
  $w \in z + R$ with $|N(w)| < 1$.
\end{proposition}

This is a classical fact. The Euclidean minimum of $R$ with respect to
$|N|$, which is the classical Euclidean minimum
$M(K) = \sup_{z \in K}\inf_{q \in R}|N(z - q)|$ of the field, is $5/4$. It
is attained only on the class $p + R$, and the next largest value of
$\inf_{q \in R}|N(z - q)|$, the second Euclidean minimum of Cavallar and
Lemmermeyer, is $31/32$ (see Cavallar and Lemmermeyer~\cite[Section~3,
p.~339]{CavallarLemmermeyer2000}). We give a self-contained proof in
\S\ref{sec:endpoint}, which does not use Proposition~\ref{prop:shape}.

\subsection{The shape of the window} \label{sec:critical}
Let
\[
  \mathcal{Z} = \{z \in K \setminus R \colon \ve(z) \geq 0\}
\]
be the set of all elements of the nonintegral classes of nonnegative
valuation. Thus $z \in \mathcal{Z}$ if and only if the class $z + R$ is
nonintegral and has nonnegative valuation. Raising the weight $t$ makes
these classes harder to divide. The next proposition shows that they alone
decide which weights larger than $\sqrt{5}$ are in the window, and that
the window is an interval.

\begin{proposition} \label{prop:shape}
  \mbox{}
  \begin{enumerate}
    \item[(a)] If $0 < t \leq \sqrt{5}$, then $\phi_{t}$ is not a
      Euclidean function on $R$.
    \item[(b)] If $t > \sqrt{5}$, then $\phi_{t}$ is a Euclidean function
      on $R$ if and only if every $z \in \mathcal{Z}$ has a normalized
      remainder that is reducing at weight $t$.
    \item[(c)] Let $z \in \mathcal{Z}$, and let $w$ be a normalized
      remainder of $z$ that is reducing at weight $t$ for some $t > 0$.
      Then $w$ is reducing at every weight $x$ with $0 < x \leq t$.
    \item[(d)] If $\phi_{t}$ is a Euclidean function on $R$ for some
      $t > 0$, then $(\sqrt{5},t] \subseteq w(\mathfrak{p})$.
  \end{enumerate}
\end{proposition}
\begin{proof}
  (a) By Lemma~\ref{lem:pclass}, the class $p + R$ obstructs division at
  weight $t$. By Lemma~\ref{lem:classes}(b), $\phi_{t}$ does not satisfy
  condition (iii) with $\kappa = 1$.

  (b) Let $t > \sqrt{5}$. Then $t > 2$, and by Lemma~\ref{lem:classes},
  $\phi_{t}$ is a Euclidean function if and only if no class obstructs
  division at weight $t$. The class $R$ has the reducing normalized
  remainder $0$. The class $p + R$ has the reducing normalized remainder
  $p$, because $\phi_{t}(p) = 5/t^{2} < 1$. Let $z \notin p + R$ with
  $\ve(z) < 0$. By Proposition~\ref{prop:ordinary}, $z$ has a normalized
  remainder $w$ with $|N(w)| < 1$, and $\ve(w) = \ve(z) < 0$. Since
  $t/2 > 1$,
  \[
    \phi_{t}(w) = |N(w)|\left(\frac{t}{2}\right)^{\ve(w)} < |N(w)| < 1.
  \]
  The remaining classes are the classes of the elements of $\mathcal{Z}$.
  This proves (b).

  (c) Since the class $z + R$ has nonnegative valuation, $\ve(w) \geq 0$.
  Hence, for $0 < x \leq t$,
  \[
    \phi_{x}(w) = |N(w)|\left(\frac{x}{2}\right)^{\ve(w)}
    \leq |N(w)|\left(\frac{t}{2}\right)^{\ve(w)} = \phi_{t}(w) < 1.
  \]

  (d) Suppose that $\phi_{t}$ is a Euclidean function. By (a),
  $t > \sqrt{5}$, and by (b), every $z \in \mathcal{Z}$ has a normalized
  remainder that is reducing at weight $t$. Let $\sqrt{5} < x \leq t$. By
  (c), the same normalized remainders are reducing at weight $x$, so
  $\phi_{x}$ is a Euclidean function by (b). This proves the proposition.
\end{proof}

Thus the first assertion of Theorem~\ref{thm:main} is equivalent to the
statements (A) and (B) of \S\ref{sec:strategy}, which are
Propositions~\ref{prop:upper} and~\ref{prop:division}.

\subsection{Unit orbits} \label{sec:orbits}
Multiplication by $\varepsilon$ permutes the classes. The \emph{orbit} of
the class $z + R$ is the set of classes $\varepsilon^{k}z + R$ with
$k \in \Z$. The class is \emph{periodic} if $(\varepsilon^{P} - 1)z \in R$
for some integer $P \geq 1$, and the least such $P$ is its \emph{period}.
The next lemma shows that every class is periodic, and it gives a
convenient representative of a class of nonnegative valuation.

\begin{lemma} \label{lem:periodic}
  \mbox{}
  \begin{enumerate}
    \item[(a)] For every $z \in K$ there is an integer $P \geq 1$ such
      that $(\varepsilon^{P} - 1)z \in \pi^{5}R$. In particular, every
      class is periodic.
    \item[(b)] If $z \in K$ and $\ve(z) \geq 0$, then $z + R$ contains an
      element $\xi$ with $\ve(\xi) \geq 5$. For every such $\xi$ and every
      $q \in R$ with $\ve(q) \leq 4$,
      \begin{equation} \label{eq:fixedval}
        \ve(\xi - q) = \ve(q).
      \end{equation}
    \item[(c)] If $\ve(\xi) \geq 5$, then $\ve(\varepsilon\xi) \geq 5$,
      and $\ve(\xi + \gamma) \geq 5$ for every $\gamma \in \pi^{5}R$.
  \end{enumerate}
\end{lemma}
\begin{proof}
  (a) Choose a positive integer $d$ with $dz \in R$. The group
  $G = d^{-1}R/\pi^{5}R$ is finite, since its order is
  $[d^{-1}R:R][R:\pi^{5}R] = 32d^{2}$ by Lemma~\ref{lem:valuation}(c).
  Multiplication by the unit $\varepsilon$ maps $d^{-1}R$ and $\pi^{5}R$
  onto themselves, so it induces a permutation of $G$. Hence some power
  $\varepsilon^{P}$ with $P \geq 1$ fixes the element $z + \pi^{5}R$ of
  $G$.

  (b) By Lemma~\ref{lem:valuation}(d), we may write
  $z = (X + Y\sqrt{14})/d$ with $X,Y \in \Z$ and $d$ a positive odd
  integer. Choose $a,b \in \Z$ with $da \equiv X \pmod{8}$ and
  $db \equiv Y \pmod{4}$. Then
  \[
    d(z - a - b\sqrt{14}) = (X - da) + (Y - db)\sqrt{14} \in 8\Z + 4\Z\sqrt{14}
    = \pi^{5}R.
  \]
  Since $\ve(d) = \operatorname{ord}_{2}(d^{2}) = 0$, the element
  $\xi = z - a - b\sqrt{14}$ has valuation at least $5$. If $q \in R$ and
  $\ve(q) \leq 4$, then $\ve(q) < \ve(\xi)$, and
  Lemma~\ref{lem:valuation}(b) gives~\eqref{eq:fixedval}.

  (c) This follows from Lemma~\ref{lem:valuation}(b), because
  $\ve(\varepsilon) = 0$ and every element of $\pi^{5}R$ has valuation at
  least $5$. This proves the lemma.
\end{proof}
\section{Squares, labels, and the coding of orbits} \label{sec:graph}
In this section we describe the squares, the covering test, and the finite
directed graph, and we prove the lemmas that connect the orbits of classes
to closed walks in the graph. The calculation itself is carried out in
\S\ref{sec:weighted}. Ramsey codes the orbit of a point by the sequence of
members of a partition that the orbit visits, and his partitions have the
property that every sequence allowed by the transitions between their
members is the code of the orbit of some
point~\cite[Section~4]{Ramsey2025}. In the language of this paper, every
bi-infinite walk in the corresponding graph is followed by the orbit of
some point. Our partition into squares does not have this property, and we
use only the fact that the orbit of every point follows a walk in the
graph (Lemma~\ref{lem:edges}). What we do use is the sequence of labels of
the walk, which records the lattice translations that return the orbit to
the fundamental domain. Lemma~\ref{lem:coding} shows that the labels
determine the orbit through a series in powers of $\varepsilon^{-1}$, and
in \S\ref{sec:weighted} this series converts the labels found by the
calculation into the number $T$. Directed graphs of the regions left after
a norm covering also occur in Lezowski's algorithm for computing Euclidean
minima~\cite[Section~2.4]{Lezowski2014}. His simplification of the graph
uses merging of regions and strongly connected
components~\cite[Section~3.2.2]{Lezowski2014}.

\subsection{Coordinates and squares} \label{sec:squares}
Let $\Lambda = 8\Z \times 4\Z$. By Lemma~\ref{lem:valuation}(c), this is
the set of points $(a,b)$ for which $a + b\sqrt{14}$ lies in $\pi^{5}R$.
We write a point of $\R^{2}$ as $\rho = (x,y)$, and when $x$ and $y$ are
rational we call $\rho$ the \emph{coefficient point} of
$x + y\sqrt{14} \in K$. Following Ramsey~\cite[Section~3]{Ramsey2025}, we
call
\[
  \rho_{s} = x - y\sqrt{14},\qquad \rho_{u} = x + y\sqrt{14}
\]
the \emph{stable} and \emph{unstable coordinates} of $\rho$. If $\rho$ is
the coefficient point of $z \in K$, then $\rho_{u} = z$ and
$\rho_{s} = \bar{z}$. We write
$N(\rho) = \rho_{s}\rho_{u} = x^{2} - 14y^{2}$, which is $N(z)$ in this
case, and for $q = a + b\sqrt{14} \in R$ we write $\rho - q$ for the point
$(x - a,y - b)$, so that
$N(\rho - q) = (\rho_{s} - \bar{q})(\rho_{u} - q)$.

For $\rho = (x,y)$ let $\varepsilon\rho = (15x + 56y,\,4x + 15y)$. When
$x$ and $y$ are rational, this is the coefficient point of
$\varepsilon(x + y\sqrt{14})$, and in general
\[
  (\varepsilon\rho)_{u} = \varepsilon\rho_{u},\qquad
  (\varepsilon\rho)_{s} = \varepsilon^{-1}\rho_{s}.
\]
Since $\varepsilon = 29.966\ldots$, repeated multiplication by
$\varepsilon$ expands the unstable coordinate and contracts the stable
coordinate, which explains Ramsey's names. The map
$\rho \mapsto \varepsilon\rho$ is linear and invertible, and it maps
$\Lambda$ onto $\Lambda$, because $\varepsilon\pi^{5}R = \pi^{5}R$. We
write $\varepsilon^{k}\rho$ for its $k$th iterate, $k \in \Z$. For
$\rho = (x,y) \in \R^{2}$ let $\|\rho\| = \max\{|x|,|y|\}$, and define the
\emph{distance} between two subsets of $\R^{2}$ as the infimum of
$\|\rho - \rho'\|$ over points $\rho$ of the first and $\rho'$ of the
second. For every $\rho = (x,y) \in \R^{2}$ we have
\begin{equation} \label{eq:coordbound}
  |\rho_{u}|,|\rho_{s}| \leq |x| + \sqrt{14}\,|y| \leq (1 + \sqrt{14})\|\rho\|.
\end{equation}

We work on the torus $\R^{2}/\Lambda$. If $\rho$ is the coefficient point
of $z \in K$, we call $\rho + \Lambda$ the \emph{point of $z$} in
$\R^{2}/\Lambda$. The \emph{orbit} of a point $\rho + \Lambda$ is the set
of points $\varepsilon^{k}\rho + \Lambda$ with $k \in \Z$, which we index
by $k$, and the point is \emph{periodic} if
$\varepsilon^{P}\rho - \rho \in \Lambda$ for some $P \geq 1$. The least
such $P$ is its \emph{period}. If $\rho$ is the coefficient point of an
element of $K$, then $\rho + \Lambda$ is periodic by
Lemma~\ref{lem:periodic}(a). The \emph{fundamental domain} is
$D = [0,8) \times [0,4)$. A \emph{representative} of a point of
$\R^{2}/\Lambda$ is a point of $\R^{2}$ congruent to it modulo $\Lambda$.
Every point has exactly one representative in $D$, which we call its
\emph{representative in $D$}, and at most four representatives in the
closure of $D$. We use three symmetries of the torus. Negation
$\rho \mapsto -\rho$, conjugation $(x,y) \mapsto (x,-y)$, and their
composition, the reflection $(x,y) \mapsto (8 - x,y)$ modulo $\Lambda$,
map $\Lambda$ onto itself. If $\rho$ is the coefficient point of
$\xi \in K$, these maps send it to the coefficient points of $-\xi$,
$\bar{\xi}$, and $8 - \bar{\xi}$, and the unstable coordinate of the
reflected point is $8 - \rho_{s}$. Negation commutes with multiplication
by $\varepsilon$. Conjugation and the reflection reverse the direction of
time, since the conjugate of $\varepsilon\rho$ is $\varepsilon^{-1}$ times
the conjugate of $\rho$, so they carry the orbit of a point onto the orbit
of its image, traversed in the opposite order. In all three cases the
image of a periodic orbit is a periodic orbit.

For a positive integer $L$, which we call the \emph{resolution}, we cover
the closure of $D$ by the closed squares
\[
  C_{L}(i,j) = [i/L,(i + 1)/L] \times [j/L,(j + 1)/L],
\]
where $(i,j)$ runs over the integer points of $LD$, that is,
$0 \leq i < 8L$ and $0 \leq j < 4L$. The corresponding half-open squares
partition $D$, while their closures are used for the covering tests. A
point of $\R^{2}/\Lambda$ whose representative in $D$ is $(x,y)$ is
\emph{assigned} to the square
$C_{L}(\lfloor Lx\rfloor,\lfloor Ly\rfloor)$. The \emph{itinerary} of an
orbit at resolution $L$ is the sequence of squares assigned to its points.
A closed square is \emph{incident} to a point of $\R^{2}/\Lambda$ if it
contains a representative of that point, and it is incident to a set of
points if it is incident to one of them. At a vertex of the grid, four
squares are incident, including squares on the other side of the boundary
of $D$.

On each square we test finitely many quotients. The tested quotients of
the square $C_{L}(i,j)$ are the following $32$ elements:
\begin{equation} \label{eq:qbox}
  q = a + b\sqrt{14},\qquad
  \lfloor i/L\rfloor - 3 \leq a \leq \lfloor i/L\rfloor + 4,\qquad
  \lfloor j/L\rfloor - 1 \leq b \leq \lfloor j/L\rfloor + 2.
\end{equation}
The reason for this box is as follows. We work modulo
$\pi^{5}R = 4\mathfrak{p} = 8\Z + 4\sqrt{14}\,\Z$, whose coefficient
points form the lattice $\Lambda$. A box of $8 \times 4$ consecutive
integer points therefore contains exactly one element of each of the $32$
classes of $R$ modulo $\pi^{5}R$. The box~\eqref{eq:qbox} is the one that
is centered on the unit square $[\lfloor i/L\rfloor,\lfloor i/L\rfloor + 1]
\times [\lfloor j/L\rfloor,\lfloor j/L\rfloor + 1]$ containing
$C_{L}(i,j)$. So, for each class modulo $\pi^{5}R$, we test one quotient
of that class near the square. For every $q = a + b\sqrt{14} \in R$, let
$I_{1} = [i - aL,\,i + 1 - aL]$ and $I_{2} = [j - bL,\,j + 1 - bL]$, and
let $B_{L}(i,j;q)$ be the integer
\begin{equation} \label{eq:cellbound}
  B_{L}(i,j;q) = \max\Bigl\{\Bigl|\min_{X \in I_{1}}X^{2}
  - 14\max_{Y \in I_{2}}Y^{2}\Bigr|,\
  \Bigl|\max_{X \in I_{1}}X^{2} - 14\min_{Y \in I_{2}}Y^{2}\Bigr|\Bigr\}.
\end{equation}
As $\rho = (x,y)$ runs over $C_{L}(i,j)$, the numbers $L(x - a)$ and
$L(y - b)$ run independently over $I_{1}$ and $I_{2}$, and
$X^{2} - 14Y^{2}$ increases with $X^{2}$ and decreases with $Y^{2}$. Hence
\begin{equation} \label{eq:cellmax}
  \max_{\rho \in C_{L}(i,j)}|N(\rho - q)| = \frac{B_{L}(i,j;q)}{L^{2}}.
\end{equation}

Covered squares, which we define in the next subsection, have a simple
geometric meaning. For a quotient $q$ and a number $c > 0$, the set of
points $\rho$ with $|N(\rho - q)| < c$ is the region between the four
branches of the two hyperbolas $N(\rho - q) = \pm c$. Its asymptotes are
the two lines through the coefficient point of $q$ on which one of the
coordinates $\rho_{s} - \bar{q}$ and $\rho_{u} - q$ vanishes.
By~\eqref{eq:cellmax}, an inequality of the form $B_{L}(i,j;q) < cL^{2}$
says that the square lies inside this region. Then every class containing
an element whose coefficient point lies in the square has a normalized
remainder, with the quotient $q$, of absolute norm less than $c$. Ramsey
calls a member of his partition \emph{trapped} when its closure lies in
the union of finitely many such regions~\cite[Section~5]{Ramsey2025}, and
our covered squares play the same role. A square may fail the test for one
quotient and pass it for another (Figure~\ref{fig:bands}), and an
uncovered square has merely failed finitely many tests. This does not mean
that it contains a point without a small normalized remainder.

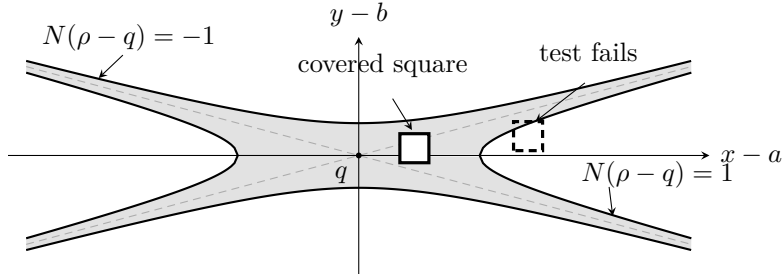
\begin{figure}[tbp]
  \centering
  \begin{tikzpicture}[x=1.6cm,y=1.6cm,
  every node/.style={font=\footnotesize}]
    % Exact level curves x^2-14y^2=+-1, in coefficient coordinates
    % centered at the quotient q. The gray shading is |N(rho-q)|<1.
    \begin{scope}
      \clip (-2.75,-0.9) rectangle (2.75,0.9); \fill[gray!23]
      plot[domain=-2.75:2.75,samples=120] (\x,{sqrt((\x*\x+1)/14)}) --
      plot[domain=2.75:1,samples=60] (\x,{sqrt((\x*\x-1)/14)}) -- (-1,0)
      -- plot[domain=-1:-2.75,samples=60] (\x,{sqrt((\x*\x-1)/14)}) --
      cycle; \fill[gray!23] plot[domain=-2.75:2.75,samples=120]
      (\x,{-sqrt((\x*\x+1)/14)}) -- plot[domain=2.75:1,samples=60]
      (\x,{-sqrt((\x*\x-1)/14)}) -- (-1,0) --
      plot[domain=-1:-2.75,samples=60] (\x,{-sqrt((\x*\x-1)/14)}) --
      cycle; \draw[gray!65,densely dashed] (-2.75,{-2.75/sqrt(14)}) --
      (2.75,{2.75/sqrt(14)}); \draw[gray!65,densely dashed]
      (-2.75,{2.75/sqrt(14)}) -- (2.75,{-2.75/sqrt(14)}); \draw[thick]
      plot[domain=-2.75:2.75,samples=120] (\x,{sqrt((\x*\x+1)/14)});
      \draw[thick] plot[domain=-2.75:2.75,samples=120]
      (\x,{-sqrt((\x*\x+1)/14)}); \foreach \sgn in {1,-1} { \draw[thick]
      plot[domain=1:2.75,samples=65] (\x,{\sgn*sqrt((\x*\x-1)/14)});
      \draw[thick] plot[domain=-2.75:-1,samples=65]
      (\x,{\sgn*sqrt((\x*\x-1)/14)}); }
    \end{scope}
    \draw[-stealth] (-2.9,0) -- (2.9,0) node[right] {$x-a$};
    \draw[-stealth] (0,-0.98) -- (0,0.98) node[above] {$y-b$}; \fill (0,0)
    circle (1.1pt); \node[below left] at (0,0) {$q$}; \draw[very
    thick,fill=white] (0.34,-0.06) rectangle (0.58,0.18); \node[above] at
    (0.2,0.56) {covered square}; \draw[-stealth,thin] (0.28,0.47) --
    (0.44,0.21); \draw[very thick,densely dashed] (1.28,0.04) rectangle
    (1.52,0.28); \node[above] at (1.90,0.69) {test fails};
    \draw[-stealth,thin] (1.87,0.60) -- (1.46,0.28); \node[anchor=south]
    at (-1.90,0.78) {$N(\rho-q)=-1$}; \draw[-stealth,thin] (-1.95,0.83) --
    (-2.15,0.64); \node[anchor=west] at (1.75,-0.16) {$N(\rho-q)=1$};
    \draw[-stealth,thin] (2.07,-0.26) -- (2.10,-0.50);
  \end{tikzpicture}
  \caption{The hyperbolic region of \S\ref{sec:squares}, drawn in
  coefficient coordinates relative to a quotient $q = a + b\sqrt{14}$ for
  the illustrative threshold $c = 1$. The gray region is
  $|N(\rho - q)| < 1$, bounded by the level curves $N(\rho - q) = \pm 1$
  with their dashed asymptotes. A closed square wholly inside the gray
  region satisfies $|N(\rho - q)| < 1$ at every point. The covering test
  of \S\ref{sec:test} is this containment with the threshold $(8/9)^{h}$,
  $h = \ve(q)$, in place of $1$, and formula~\eqref{eq:cellmax} decides it
  exactly. The dashed square meets the boundary, so this quotient does not
  cover it at threshold $1$, whether or not another quotient does.}
  \label{fig:bands}
\end{figure}

\subsection{The test at the trial weight \texorpdfstring{$9/4$}{9/4}} \label{sec:test}
Two choices must be made before the construction begins. First, to obtain
finitely many tests, we consider only normalized remainders of valuation at most some
bound. We take the bound $4$, the least bound for which the calculation of
\S\ref{sec:weighted} succeeds. Proposition~\ref{prop:valuationfour} shows that
no smaller bound can succeed, because some classes have no reducing
normalized remainder of valuation at most $3$ at the weight $T$. With this bound we
work modulo $\pi^{5}R$, which is why $\Lambda = 8\Z \times 4\Z$. For a
class of nonnegative valuation we choose a representative $\xi$ with
$\ve(\xi) \geq 5$, as in Lemma~\ref{lem:periodic}(b).
By~\eqref{eq:fixedval}, the valuation of a normalized remainder $\xi - q$ with
$\ve(q) \leq 4$ then depends only on $q$, so the same test can be used at
every point of a square. Second, we choose a trial weight. A covered
square gives reducing normalized remainders at every weight up to the trial weight
(Lemma~\ref{lem:covered} below), so the trial weight must be at least the
endpoint that the calculation will find, and a smaller trial weight covers
more squares. Since the lower endpoint $\sqrt{5}$ is known, a simple
rational number just above $\sqrt{5}$ is a natural choice, and we take
$9/4$, the first multiple of $1/4$ above $\sqrt{5}$. The inequality
$T < 9/4$ that this choice requires is verified in~\eqref{eq:Tbounds}. If
the calculation had produced a number $T \geq 9/4$, we would have repeated
it with a larger trial weight. Since $2/(9/4) = 8/9$, the corresponding
norm bounds are the rational numbers $(8/9)^{h}$, so the tests require
only integer arithmetic.

A tested quotient $q$ of~\eqref{eq:qbox} is used only if $q \neq 0$ and
$h = \ve(q) \leq 4$, and then the test is
\begin{equation} \label{eq:rationaltest}
  9^{h}
  B_{L}(i,j;q) < 8^{h}L^{2}.
\end{equation}
By~\eqref{eq:cellmax}, this means that $|N(\rho - q)| < (8/9)^{h}$ at
every point $\rho$ of the square. A square is \emph{covered} if one of its
tested quotients satisfies the test, and \emph{uncovered} otherwise.

\begin{lemma} \label{lem:covered}
  Suppose that $C_{L}(i,j)$ is covered by the tested quotient $q$ with
  $h = \ve(q) \leq 4$. If $\xi \in K$ has valuation at least $5$ and its
  coefficient point lies in $C_{L}(i,j)$, then $\ve(\xi - q) = h$ and
  $\phi_{t}(\xi - q) < 1$ for every $t$ with $0 < t \leq 9/4$.
\end{lemma}
\begin{proof}
  By~\eqref{eq:cellmax} and~\eqref{eq:rationaltest},
  $|N(\xi - q)| < (8/9)^{h}$, and~\eqref{eq:fixedval} gives
  $\ve(\xi - q) = h$. Hence, for $0 < t \leq 9/4$,
  \[
    \phi_{t}(\xi - q) = |N(\xi - q)|\left(\frac{t}{2}\right)^{h}
    < \left(\frac{8}{9}\right)^{h}\left(\frac{9}{8}\right)^{h} = 1.
  \]
  This proves the lemma.
\end{proof}

It is convenient to describe the points that the tests can never remove by
one set. Let
\begin{equation} \label{eq:X}
  \mathcal{X} = \Bigl\{\rho + \Lambda \in \R^{2}/\Lambda \colon
  |N(\rho - q)| \geq \Bigl(\frac{8}{9}\Bigr)^{\ve(q)}
  \text{ for every } q \in R \setminus \pi^{5}R\Bigr\}.
\end{equation}
Since $\pi^{5}R$ is the set of elements of $R$ of valuation at least $5$,
the condition is imposed exactly for the quotients $q$ with
$\ve(q) \leq 4$. The set is well defined, because replacing $\rho$ by
$\rho + \lambda$ with $\lambda \in \Lambda$, the coefficient point of
$\gamma \in \pi^{5}R$, and $q$ by $q + \gamma$ changes neither
$N(\rho - q)$ nor $\ve(q)$. The set $\mathcal{X}$ is of interest for two
reasons.

\begin{lemma} \label{lem:X}
  \mbox{}
  \begin{enumerate}
    \item[(a)] Multiplication by $\varepsilon$, negation, conjugation, and
      the reflection $(x,y) \mapsto (8 - x,y)$ map $\mathcal{X}$ onto
      itself.
    \item[(b)] At every resolution, no point of $\mathcal{X}$ lies in a
      covered square.
    \item[(c)] If $z \in K$ has no normalized remainder of absolute norm
      less than $1$, then the point of $z$ in $\R^{2}/\Lambda$ lies in
      $\mathcal{X}$.
    \item[(d)] Let $0 < t \leq 9/4$, let $\xi \in K$ satisfy
      $\ve(\xi) \geq 5$, and suppose that no normalized remainder $w$ of
      $\xi$ with $\ve(w) \leq 4$ satisfies $\phi_{t}(w) < 1$. Then the
      point of $\xi$ in $\R^{2}/\Lambda$ lies in $\mathcal{X}$.
  \end{enumerate}
\end{lemma}
\begin{proof}
  (a) For $\rho \in \R^{2}$ and $q \in R$ we have
  $N(\varepsilon\rho - \varepsilon q) = N(\rho - q)$,
  $N(-\rho - (-q)) = N(\rho - q)$, and
  $N(\bar{\rho} - \bar{q}) = N(\rho - q)$, where $\bar{\rho} = (x,-y)$,
  and $\ve(\varepsilon q) = \ve(-q)
  = \ve(\bar{q}) = \ve(q)$. As $q$ runs over $R \setminus \pi^{5}R$,
  so do $\varepsilon q$, $-q$, and $\bar{q}$. Hence the first three maps
  preserve $\mathcal{X}$, and so does their composition.

  (b) If the square $C_{L}(i,j)$ is covered by the tested quotient $q$
  with $h = \ve(q) \leq 4$, then~\eqref{eq:cellmax}
  and~\eqref{eq:rationaltest} give $|N(\rho - q)| < (8/9)^{h}$ for every
  point $\rho$ of the square, so no such point lies in $\mathcal{X}$.

  (c) Let $\rho$ be the coefficient point of $z$. For $q \in R$, the
  element $z - q$ is a normalized remainder of $z$ and
  $N(\rho - q) = N(z - q)$, so $|N(\rho - q)| \geq 1 \geq (8/9)^{\ve(q)}$.

  (d) Let $\rho$ be the coefficient point of $\xi$, and let $q \in R$ with
  $h = \ve(q) \leq 4$. By~\eqref{eq:fixedval}, $\ve(\xi - q) = h$, and by
  hypothesis $\phi_{t}(\xi - q) = |N(\xi - q)|(t/2)^{h} \geq 1$. Since
  $t/2 \leq 9/8$, this gives
  $|N(\rho - q)| = |N(\xi - q)| \geq (8/9)^{h}$. This proves the lemma.
\end{proof}

By (a), the orbit of a point of $\mathcal{X}$ lies in $\mathcal{X}$, and
by (b) it never visits a covered square. Part (c) will give
Proposition~\ref{prop:ordinary}, and part (d) will give division at the
weight $T$, from the same finite calculation.

\subsection{The points of the classes \texorpdfstring{$R$}{R} and \texorpdfstring{$p + R$}{p+R}} \label{sec:Fsec}
Some points can never be removed by a covering test. Let
\begin{equation} \label{eq:centers}
  \mathcal{F} = \{0 + \Lambda\} \cup
  \{(a + 1/2,\,b + 1/2) + \Lambda \colon a,b \in \Z,\ 0 \leq a < 8,\
  0 \leq b < 4\}.
\end{equation}
The point $0$ is the coefficient point of the element $0$ of $\pi^{5}R$,
and the point $(a + 1/2,b + 1/2)$ is the coefficient point of the element
$(A + B\sqrt{14})/2$ of $p + R$ with $A = 2a + 1$ and $B = 2b + 1$. Thus
the point $0$ represents the class $R$, and the $32$ points with
half-integral coordinates represent the class $p + R$. Multiplication by
$\varepsilon$ fixes $0$, and by Lemma~\ref{lem:pclass} it maps the class
$p + R$ to itself, so it permutes the $32$ points with half-integral
coordinates. Since $\varepsilon^{2} - 1 = 448 + 120\sqrt{14}$ has
valuation $7$ and every element of $p + R$ has valuation $-2$,
multiplication by $\varepsilon^{2}$ fixes each of these points.
Multiplication by $\varepsilon$ itself fixes none of them, because
$\varepsilon - 1 = 14 + 4\sqrt{14}$ has valuation $2$, so that
$(\varepsilon - 1)c$ has valuation $0$ for every $c \in p + R$ and does
not lie in $\pi^{5}R$. Thus they form sixteen orbits of period $2$. In the
closure of $D$, the point $0$ has the four representatives $(0,0)$,
$(8,0)$, $(0,4)$, and $(8,4)$, and every other point of $\mathcal{F}$ has
exactly one representative. In this paper a representative of a point of
$\mathcal{F}$ always means one of these representatives in the closure of
$D$.

\begin{lemma} \label{lem:exceptional}
  At every resolution, no square incident to $\mathcal{F}$ is covered. If
  $L \geq 3$, then every closed square of side $1/L$ contains at most one
  representative of a point of $\mathcal{F}$.
\end{lemma}
\begin{proof}
  Suppose, by way of contradiction, that a square incident to
  $\mathcal{F}$ is covered by a tested quotient $q$ with
  $h = \ve(q) \leq 4$. The closed square contains a representative $c$ of
  a point of $\mathcal{F}$, and by~\eqref{eq:cellmax}
  and~\eqref{eq:rationaltest}, $|N(c - q)| < (8/9)^{h} \leq 1$. If
  $c \in \Lambda$, then $c$ is the coefficient point of an element
  $\gamma$ of $\pi^{5}R$, and $\gamma - q$ is a nonzero element of $R$,
  because $\ve(q) \leq 4 < \ve(\gamma)$. So $|N(\gamma - q)| \geq 1$, a
  contradiction. Otherwise $c$ is the coefficient point of an element of
  $p + R$, and so is $c - q$. Then Lemma~\ref{lem:pclass} gives
  $|N(c - q)| \geq 5/4$, again a contradiction. Finally, two distinct
  representatives of points of $\mathcal{F}$ differ by at least $1/2$ in
  one of their coordinates, so a closed square of side $1/L < 1/2$
  contains at most one of them. This proves the lemma.
\end{proof}

\subsection{The graph} \label{sec:thegraph}
The covered squares alone do not suffice, and we also record the possible
transitions between uncovered squares. We fix a resolution $L$ and a set
$\mathcal{S}$ of uncovered squares $C_{L}(i,j)$ at this resolution. In
\S\ref{sec:closedwalks}, $\mathcal{S}$ will be the set of all uncovered
squares at the resolution $2^{10}$, and later the set of those among them
that lie on closed walks. The squares of $\mathcal{S}$ are the vertices of
a directed graph whose edges are defined as follows. Let
\begin{equation} \label{eq:offsets}
  \begin{split}
    \mathcal{O} = \{(e_{1},e_{2}) \in \Z^{2} \colon
     & -1 \leq e_{1} \leq 71,\ -1 \leq e_{2} \leq 19, \\
     & -15 \leq 15e_{1} - 56e_{2} \leq 57,\ -15 \leq -4e_{1} + 15e_{2} \leq 5\}.
  \end{split}
\end{equation}
This set has $97$ elements. Its bounds come from the matrices
\[
  \begin{pmatrix}
    15 & 56 \\ 4 & 15\end{pmatrix}
  \quad\text{and}\quad
  \begin{pmatrix}
    15 & -56 \\ -4 & 15\end{pmatrix}
\]
of multiplication by $\varepsilon$ and by $\varepsilon^{-1}$ on
coefficient points. The proof of Lemma~\ref{lem:edges} below shows that
the offsets that occur satisfy $0 \leq e_{1} \leq 70 = 15 + 56 - 1$,
$0 \leq e_{2} \leq 18 = 4 + 15 - 1$,
$-15 < 15e_{1} - 56e_{2} < 57 = 1 + 56$, and
$-15 < -4e_{1} + 15e_{2} < 5 = 1 + 4$. In~\eqref{eq:offsets} we have
enlarged each of these ranges by one at each end, which can only add
edges. For a vertex $(i_{0},j_{0})$ and $(e_{1},e_{2}) \in \mathcal{O}$,
let $(X,Y) = \varepsilon(i_{0},j_{0}) + (e_{1},e_{2})$, and let
$(i_{1},j_{1})$ be the unique integer point of $LD$ that is congruent to
$(X,Y)$ modulo $L\Lambda$. If $(i_{1},j_{1})$ is a vertex, we include an
edge from $(i_{0},j_{0})$ to $(i_{1},j_{1})$, and we label it by the
element $\ell$ of $K$ whose coefficient point is
$((X,Y) - (i_{1},j_{1}))/L$. This coefficient point lies in $\Lambda$, so
$\ell \in \pi^{5}R$. Two vertices may be joined by several edges with
different labels.

A \emph{directed walk} is a finite sequence of vertices in which each
vertex is joined to the next one by an edge, together with a choice of
such an edge for each pair of consecutive vertices. Its \emph{length} is
the number of these edges, and its \emph{labels} are their labels, in
order. It is \emph{closed} if it has positive length and its last vertex
is its first. Vertices and edges may be repeated. A \emph{directed cycle}
is a closed walk in which no vertex is repeated except for the first
vertex at the end. A \emph{strongly connected component} of the graph is a
maximal set of vertices any two of which can be reached from each other by
directed walks. For brevity we call these sets \emph{components}. A vertex
lies on a closed walk if and only if its component has more than one
vertex or the vertex has a \emph{loop}, an edge from the vertex to itself.
We \emph{retain} the vertices that lie on closed walks.

The next lemma shows that the graph records every transition that the
orbit of a point can make (see Figure~\ref{fig:edges}).

\begin{lemma} \label{lem:edges}
  Let $\rho_{0}$ and $\rho_{1}$ be the representatives in $D$ of a point
  of $\R^{2}/\Lambda$ and of its product with $\varepsilon$, and let
  $C_{L}(i_{0},j_{0})$ and $C_{L}(i_{1},j_{1})$ be the squares to which
  they are assigned. Then there is $(e_{1},e_{2}) \in \mathcal{O}$ such
  that $(X,Y) = \varepsilon(i_{0},j_{0}) + (e_{1},e_{2})$ is congruent to
  $(i_{1},j_{1})$ modulo $L\Lambda$, and
  \[
    \rho_{1} = \varepsilon\rho_{0} - \frac{(X,Y) - (i_{1},j_{1})}{L}.
  \]
  In particular, if both squares are vertices of the graph at resolution
  $L$, then they are joined by an edge whose label $\ell$ satisfies
  \[
    (\rho_{1})_{u} = \varepsilon(\rho_{0})_{u} - \ell,\qquad
    (\rho_{1})_{s} = \varepsilon^{-1}(\rho_{0})_{s} - \bar{\ell}.
  \]
\end{lemma}
\begin{proof}
  Write $L\rho_{0} = (i_{0},j_{0}) + \theta$ with
  $\theta = (\theta_{1},\theta_{2})$ in $[0,1)^{2}$, and write
  $\varepsilon\theta = (e_{1} + f_{1},e_{2} + f_{2})$ with
  $e_{1},e_{2} \in \Z$ and $f_{1},f_{2} \in [0,1)$. Since
  $\varepsilon\theta = (15\theta_{1} + 56\theta_{2},\,4\theta_{1} + 15\theta_{2})$,
  we have $0 \leq e_{1} \leq 70$ and $0 \leq e_{2} \leq 18$. The inverse
  of multiplication by $\varepsilon$ is
  $(x,y) \mapsto (15x - 56y,\,-4x + 15y)$, so
  \[
    15e_{1} - 56e_{2} = \theta_{1} - 15f_{1} + 56f_{2},\qquad
    -4e_{1} + 15e_{2} = \theta_{2} + 4f_{1} - 15f_{2}.
  \]
  The first of these numbers lies in the open interval $(-15,57)$, and the
  second lies in $(-15,5)$. Hence $(e_{1},e_{2}) \in \mathcal{O}$. Since
  $\varepsilon(i_{0},j_{0})$ is an integer point, the floors of the
  coordinates of
  $L\varepsilon\rho_{0} = \varepsilon(i_{0},j_{0}) + \varepsilon\theta$
  are $(X,Y) = \varepsilon(i_{0},j_{0}) + (e_{1},e_{2})$. Now
  $\rho_{1} = \varepsilon\rho_{0} - \lambda$ for the unique
  $\lambda \in \Lambda$ with $\varepsilon\rho_{0} - \lambda \in D$. The
  floors of the coordinates of
  $L\rho_{1} = L\varepsilon\rho_{0} - L\lambda$ are $(X,Y) - L\lambda$,
  because $L\lambda$ is an integer point. Hence
  $(i_{1},j_{1}) = (X,Y) - L\lambda$ and
  $\lambda = ((X,Y) - (i_{1},j_{1}))/L$. The last assertion follows,
  because $\lambda$ is the coefficient point of $\ell$, whose stable and
  unstable coordinates are $\bar{\ell}$ and $\ell$. This proves the lemma.
\end{proof}

\begin{figure}[tbp]
  \centering
  \begin{tikzpicture}[x=0.3cm,y=0.3cm]
    % Exact coefficient-coordinate example (the axes are to scale):
    % rho_0=(3/2,1/10), epsilon*rho_0=(281/10,15/2),
    % lambda=(24,4), rho_1=(41/10,7/2).
    \foreach \a in {0,8,16,24} { \foreach \b in {0,4,8} {
    \draw[gray!60,densely dashed] (\a,\b) rectangle (\a+8,\b+4); } }
    \fill[gray!20] (0,0) rectangle (8,4); \draw[thick] (0,0) rectangle
    (8,4); \node at (7.2,2.7) {\small $D$}; \fill (1.5,0.1) circle
    (2.5pt); \node[above] at (1.5,0.2) {\small $\rho_{0}$}; \fill
    (28.1,7.5) circle (2.5pt); \node[right] at (28.1,7.5) {\small
    $\varepsilon\rho_{0}$}; \fill (4.1,3.5) circle (2.5pt); \node[below]
    at (4.1,3.45) {\small $\rho_{1}$}; \draw[-stealth,thin] (1.75,0.25)
    to[bend left = 12] (27.8,7.55); \node[above,fill = white,inner sep =
    1pt] at (12.0,8.8) {\small multiplication by $\varepsilon$};
    \draw[-stealth,thick,densely dashed] (27.8,7.45) -- (4.45,3.55);
    \node[below right,fill = white,inner sep = 1pt] at (17.0,4.6) {\small
    subtraction of $\lambda = (24,4)$};
  \end{tikzpicture}
  \caption{An example of Lemma~\ref{lem:edges} in coefficient coordinates.
  For $\rho_{0} = (3/2,1/10)$ we have
  $\varepsilon\rho_{0} = (281/10,15/2)$. Subtracting
  $\lambda = (24,4) \in \Lambda$ gives the representative
  $\rho_{1} = (41/10,7/2) \in D$. The edge label is the element
  $24 + 4\sqrt{14}$ of $\pi^{5}R$ with coefficient point $\lambda$.}
  \label{fig:edges}
\end{figure}
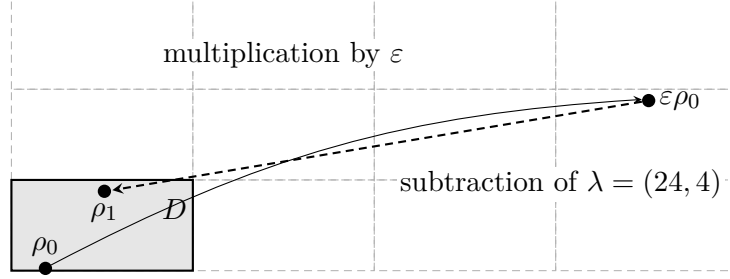

\subsection{The coding of an orbit by its labels} \label{sec:coding}
Let $\rho_{k}$ be the representatives in $D$ of the successive points of
an orbit. The points $\lambda_{k} = \varepsilon\rho_{k} - \rho_{k + 1}$
lie in $\Lambda$, and we call the elements $\ell_{k}$ of $\pi^{5}R$ with
coefficient points $\lambda_{k}$ the \emph{labels of the orbit}. By
Lemma~\ref{lem:edges}, when the squares of $\rho_{k}$ and $\rho_{k + 1}$
are both vertices of the graph, $\ell_{k}$ is the label of the edge that
the orbit follows from the first square to the second. We identify a label
$\ell \in \pi^{5}R$ with its coefficient point $\lambda \in \Lambda$
whenever it is convenient, so that $(176,44)$ and $176 + 44\sqrt{14}$
denote the same label. The next lemma shows that the labels determine the
orbit, and it gives the two coordinates of every point of the orbit as
series in the labels. In Ramsey's coding the same series express the
coordinates of a point in terms of its symbolic
coordinates~\cite[Section~4, equations~(3) and~(4),
pp.~127--128]{Ramsey2025}. The lemma will be used in two ways. In
\S\ref{sec:weighted} the finite calculation shows that the labels of
certain orbits are forced for several steps, and the series then bound the
coordinates of these orbits. Also, when the labels of a walk are the
differences $\varepsilon c_{k} - c_{k + 1}$ of a bounded sequence of
points $c_{k}$, the series telescope, and the orbit that follows the walk
is the sequence $c_{k}$. We apply this to the points of $\mathcal{F}$.

For a point $\lambda \in \Lambda$ with coordinates
$(\lambda_{1},\lambda_{2})$ we write
$\ell = \lambda_{1} + \lambda_{2}\sqrt{14}$ for the element of $\pi^{5}R$
with coefficient point $\lambda$, so that $\lambda_{u} = \ell$ and
$\lambda_{s} = \bar{\ell}$. More generally, for every point
$c = (c_{1},c_{2}) \in \R^{2}$ we have $c_{u} = c_{1} + c_{2}\sqrt{14}$
and $c_{s} = c_{1} - c_{2}\sqrt{14}$, and
$(\varepsilon c)_{u} = \varepsilon c_{u}$ and
$(\varepsilon c)_{s} = \varepsilon^{-1}c_{s}$.

\begin{lemma} \label{lem:coding}
  \mbox{}
  \begin{enumerate}
    \item[(a)] Let $(\rho_{k})_{k \in \Z}$ be a bounded sequence of points
      of $\R^{2}$, and let $(\lambda_{k})_{k \in \Z}$ be a bounded
      sequence of points of $\R^{2}$ such that
      $\rho_{k + 1} = \varepsilon\rho_{k} - \lambda_{k}$ for every
      $k \in \Z$. Then, for every $k \in \Z$,
      \begin{equation} \label{eq:series}
        (\rho_{k})_{u}
        = \sum_{j = 0}^{\infty}\varepsilon^{-j - 1}(\lambda_{k + j})_{u},
        \qquad
        (\rho_{k})_{s} = -\sum_{j = 1}^{\infty}\varepsilon^{1 - j}(\lambda_{k - j})_{s}.
      \end{equation}
      In particular the sequence $(\lambda_{k})$ determines the sequence
      $(\rho_{k})$.
    \item[(b)] Conversely, let $(\lambda_{k})_{k \in \Z}$ be a bounded
      sequence of points of $\R^{2}$. Then the series in~\eqref{eq:series}
      converge, and the points $\rho_{k}$ that they define form a bounded
      sequence with $\rho_{k + 1} = \varepsilon\rho_{k} - \lambda_{k}$ for
      every $k \in \Z$.
    \item[(c)] Let $(\rho_{k})$ and $(\lambda_{k})$ be as in (a), and
      suppose that there is a bounded sequence $(c_{k})_{k \in \Z}$ of
      points of $\R^{2}$ with
      $\lambda_{k} = \varepsilon c_{k} - c_{k + 1}$ for every $k \in \Z$.
      Then $\rho_{k} = c_{k}$ for every $k \in \Z$.
  \end{enumerate}
\end{lemma}
\begin{proof}
  (a) Taking unstable coordinates in
  $\rho_{k + 1} = \varepsilon\rho_{k} - \lambda_{k}$ gives
  $(\rho_{k})_{u} = \varepsilon^{-1}((\rho_{k + 1})_{u}
  + (\lambda_{k})_{u})$, and $n$ applications of this identity give
  \[
    (\rho_{k})_{u} = \sum_{j = 0}^{n - 1}\varepsilon^{-j - 1}(\lambda_{k + j})_{u}
    + \varepsilon^{-n}(\rho_{k + n})_{u}.
  \]
  By~\eqref{eq:coordbound}, the coordinates $(\rho_{k + n})_{u}$ are
  bounded, and $\varepsilon > 1$, so the last term tends to $0$ as
  $n \to \infty$. The terms of the series are bounded by a constant times
  $\varepsilon^{-j - 1}$, so the series converges, and this proves the
  first formula. Taking stable coordinates gives
  $(\rho_{k})_{s} = \varepsilon^{-1}(\rho_{k - 1})_{s}
  - (\lambda_{k - 1})_{s}$, and $n$ applications give
  \[
    (\rho_{k})_{s} = -\sum_{j = 1}^{n}\varepsilon^{1 - j}(\lambda_{k - j})_{s}
    + \varepsilon^{-n}(\rho_{k - n})_{s}.
  \]
  Letting $n \to \infty$ proves the second formula. A point of $\R^{2}$ is
  determined by its two coordinates, because $x = (\rho_{u} + \rho_{s})/2$
  and $y = (\rho_{u} - \rho_{s})/(2\sqrt{14})$ for $\rho = (x,y)$. Hence
  the labels determine the sequence $(\rho_{k})$.

  (b) The series converge because their terms are bounded by a constant
  times $\varepsilon^{-j - 1}$ and $\varepsilon^{1 - j}$, and the same
  bound shows that the points $\rho_{k}$ are bounded. Multiplying the
  first series by $\varepsilon$ and removing the term with $j = 0$ gives
  $\varepsilon(\rho_{k})_{u} - (\lambda_{k})_{u} = (\rho_{k + 1})_{u}$,
  and multiplying the second series by $\varepsilon^{-1}$ and adding the
  term $-(\lambda_{k})_{s}$ gives $\varepsilon^{-1}(\rho_{k})_{s} - (\lambda_{k})_{s}
  = (\rho_{k + 1})_{s}$. These are the
  two coordinates of the identity
  $\rho_{k + 1} = \varepsilon\rho_{k} - \lambda_{k}$.

  (c) The sequence $(c_{k})$ is bounded and satisfies
  $c_{k + 1} = \varepsilon c_{k} - \lambda_{k}$ for every $k$. By (a), the
  coordinates of $c_{k}$ are given by the series~\eqref{eq:series}, and so
  are the coordinates of $\rho_{k}$. Hence $\rho_{k} = c_{k}$. One can
  also see this directly, because the series telescope:
  $\sum_{j \geq 0}\varepsilon^{-j - 1}(\varepsilon(c_{k + j})_{u}
  - (c_{k + j + 1})_{u}) = (c_{k})_{u}$. This proves the lemma.
\end{proof}

In the situation of Lemma~\ref{lem:edges}, the points $\lambda_{k}$ lie in
$\Lambda$ and $(\lambda_{k})_{u} = \ell_{k}$ and
$(\lambda_{k})_{s} = \bar{\ell}_{k}$, where $\ell_{k}$ is the label of the
$k$th edge. The labels of the edges of a finite graph form a finite set,
so the sequence $(\lambda_{k})$ is bounded, and the representatives
$\rho_{k} \in D$ are bounded. Thus the series~\eqref{eq:series} become
\begin{equation} \label{eq:labelseries}
  (\rho_{k})_{u}
  = \sum_{j = 0}^{\infty}\varepsilon^{-j - 1}\ell_{k + j},
  \qquad
  (\rho_{k})_{s} = -\sum_{j = 1}^{\infty}\varepsilon^{1 - j}\bar{\ell}_{k - j}.
\end{equation}
The unstable coordinate of a point is determined by the labels of the walk
after it, and the stable coordinate by the labels before it. When the
labels $\ell_{k},\ldots,\ell_{k + n - 1}$ are known and
$(\rho_{k + n})_{u}$ is bounded below, the identity
$(\rho_{k})_{u} = \sum_{j = 0}^{n - 1}\varepsilon^{-j - 1}\ell_{k + j}
+ \varepsilon^{-n}(\rho_{k + n})_{u}$ from the proof of the lemma bounds
$(\rho_{k})_{u}$ from below, and similarly the labels
$\ell_{k - n},\ldots,\ell_{k - 1}$ and an upper bound for
$(\rho_{k - n})_{s}$ bound $(\rho_{k})_{s}$ from above. This is how the
labels found by the finite calculation will give the number $T$
(Figure~\ref{fig:series}).

\begin{figure}[!htbp]
  \centering
  \begin{tikzpicture}[x=1cm,y=1cm,
  every node/.style={font=\footnotesize}]
    \draw[-stealth,gray!75] (-6.8,0) -- (6.8,0) node[right] {time $k$};
    \foreach \x/\lab in {-5.6/{k-3},-3.9/{k-2},-2.2/{k-1},
    0/{k},2.2/{k+1},3.9/{k+2},5.6/{k+3}} { \fill (\x,0) circle (1.6pt);
    \node[below=3pt] at (\x,0) {$\rho_{\lab}$}; } \draw[very thick] (0,0)
    circle (4pt); \foreach \a/\b/\t in {-5.6/-3.9/{k-3},-3.9/-2.2/{k-2},
    -2.2/0/{k-1},0/2.2/{k}, 2.2/3.9/{k+1},3.9/5.6/{k+2}} { \draw[-stealth]
    (\a+0.13,0.08) -- (\b-0.15,0.08); \node[above] at ({(\a+\b)/2},0.18)
    {$\ell_{\t}$}; }
    \draw[decorate,decoration={brace,amplitude=5pt},gray!75] (-6,1.04) --
    (-0.13,1.04); \draw[decorate,decoration={brace,amplitude=5pt},gray!75]
    (0.13,1.04) -- (6,1.04); \node at (-3.15,1.55) {past labels determine
    $(\rho_k)_s$}; \node at (3.15,1.55) {future labels determine
    $(\rho_k)_u$}; \node[align=center,text width=6.4cm] at (-3.2,-1.55)
    {$\displaystyle (\rho_k)_s=
        -\sum_{j=1}^{\infty}\varepsilon^{1-j}\bar\ell_{k-j}$}; \node[align=center,text width=6.4cm] at
    (3.2,-1.55) {$\displaystyle (\rho_k)_u=
        \sum_{j=0}^{\infty}\varepsilon^{-j-1}\ell_{k+j}$};
  \end{tikzpicture}
  \caption{The two-sided reconstruction of an orbit from its labels, the
  lattice translations of \S\ref{sec:coding} (Lemma~\ref{lem:coding}). The
  labels preceding $\rho_{k}$ determine its stable coordinate, and the
  labels following $\rho_{k}$ determine its unstable coordinate. Both
  series converge because $\varepsilon > 1$. Thus forced finite strings of
  labels, together with bounds for their remaining tails, give the two
  inequalities in Lemma~\ref{lem:cornerbound}.} \label{fig:series}
\end{figure}
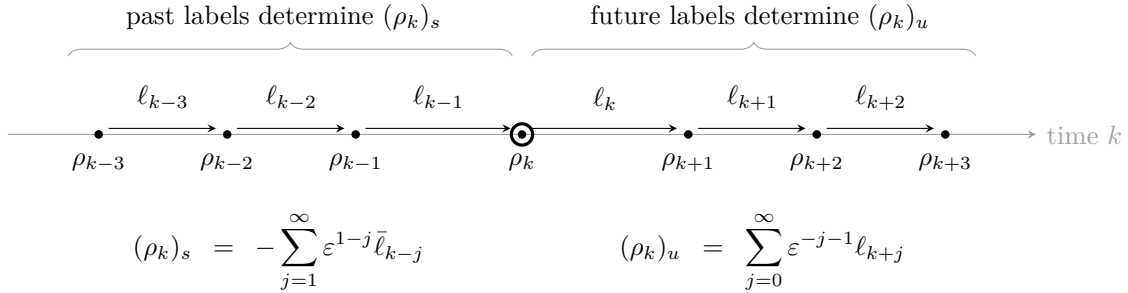

\subsection{Periodic orbits and closed walks} \label{sec:closedwalks}
We now fix the resolution and the vertex set that we use. Let
$L = 2^{10}$, and test every square of $D$. There are
$32 \cdot 2^{20} = 2^{25}$ of them, about $3.4 \cdot 10^{7}$, and a
computer tests them in a few seconds. Let $\mathcal{S}$ be the set of
uncovered squares, and let $\mathcal{R}$ be the set of retained vertices
of the graph on $\mathcal{S}$. The \emph{final graph} is the graph on
$\mathcal{R}$, with all the edges of \S\ref{sec:thegraph} between its
vertices. Only the retained squares can matter for periodic orbits, as the
next lemma shows. The resolution $2^{10}$ is the least power of $2$ at
which the final graph has the shape described in
Proposition~\ref{prop:finite}. At the resolutions $2^{6}$, $2^{7}$, and
$2^{8}$, all retained squares other than the four corner squares form a
single component. At $2^{9}$ the final graph has ten components, one of
four vertices, eight of eight vertices, and one of $424$ vertices that
contains squares on both sides of the line $y = 2$, and at $2^{10}$ the
last of these splits into the two components of
Proposition~\ref{prop:finite}(1).

\begin{lemma} \label{lem:walks}
  Let $\rho + \Lambda$ be a periodic point, and suppose that no point of
  its orbit lies in a covered square. Then the squares assigned to the
  points of the orbit belong to $\mathcal{R}$, the itinerary of the orbit
  is periodic, and the squares of one period of the itinerary, with the
  edges of Lemma~\ref{lem:edges} between consecutive squares, form a
  closed walk in the final graph, all of whose vertices lie in one
  component of the final graph. In particular, this holds for every
  periodic point of $\mathcal{X}$.
\end{lemma}
\begin{proof}
  The squares assigned to the points of the orbit are uncovered, because
  each of them contains a point of the orbit, so they belong to
  $\mathcal{S}$. By Lemma~\ref{lem:edges}, consecutive squares of the
  itinerary are joined by edges of the graph on $\mathcal{S}$. Since the
  point is periodic, the itinerary is periodic, and the squares of one
  period, with these edges, form a closed walk, so its vertices lie on a
  closed walk, and they belong to $\mathcal{R}$. The final graph contains
  all the edges between the vertices of $\mathcal{R}$, so this is a closed
  walk in the final graph. Its vertices can be reached from each other, so
  they lie in one component. The last assertion follows from
  Lemma~\ref{lem:X}(a) and (b). This proves the lemma.
\end{proof}

We say that the itinerary of such an orbit \emph{runs along} this closed
walk. The labels of the edges of the closed walk, repeated periodically,
are the labels of the orbit, and by Lemma~\ref{lem:coding} they determine
the representatives $\rho_{k}$ through the series~\eqref{eq:labelseries}.
In the next section we use the final graph to determine these labels, in
part, for the orbits that concern us.

\section{The critical orbit and the number \texorpdfstring{$T$}{T}} \label{sec:weighted}
In this section we carry out the finite calculation and examine the
final graph (Proposition~\ref{prop:finite}). We then show that every
periodic orbit of $\mathcal{X}$ that is not contained in $\mathcal{F}$
passes, after negation if necessary, through one of four squares, and that
the labels of such an orbit are forced for several steps on both sides
of the passage (Lemma~\ref{lem:passage}). The series of
Lemma~\ref{lem:coding} then bound the two coordinates of the orbit at the
passage by the coordinates of a single explicit orbit, which we call the
critical orbit (Lemmas~\ref{lem:minimum} and~\ref{lem:cornerbound}). The
product of the two bounds is the number $\nu$ of~\eqref{eq:nu}, and
$\nu$ determines $T$. Finally, in \S\ref{sec:endpoint} we prove division
at the weight $T$.

\subsection{The critical orbit} \label{sec:etadef}
The point $(5/2,5/2)$ of $\mathcal{F}$ is the coefficient point of
$(5 + 5\sqrt{14})/2 = p + 1 + 2\sqrt{14}$, and the point $(11/2,5/2)$ is
the coefficient point of $(11 + 5\sqrt{14})/2 = p + 4 + 2\sqrt{14}$, both
elements of the class $p + R$. A direct calculation gives
\begin{equation} \label{eq:twocyclepoints}
  \begin{aligned}
    \varepsilon\left(\frac{5}{2},\frac{5}{2}\right) - (176,44)
     & = \left(\frac{3}{2},\frac{7}{2}\right),  &
    \varepsilon\left(\frac{3}{2},\frac{7}{2}\right) - (216,56)
     & = \left(\frac{5}{2},\frac{5}{2}\right), \\
    \varepsilon\left(\frac{11}{2},\frac{5}{2}\right) - (216,56)
     & = \left(\frac{13}{2},\frac{7}{2}\right), &
    \varepsilon\left(\frac{13}{2},\frac{7}{2}\right) - (288,76)
     & = \left(\frac{11}{2},\frac{5}{2}\right),
  \end{aligned}
\end{equation}
so the orbits of $(5/2,5/2) + \Lambda$ and of $(11/2,5/2) + \Lambda$ have
period $2$. The labels of the two edges of the first orbit are
$176 + 44\sqrt{14}$ from $(5/2,5/2)$ and $216 + 56\sqrt{14}$ from
$(3/2,7/2)$, and the labels of the two edges of the second orbit are
$216 + 56\sqrt{14}$ from $(11/2,5/2)$ and $288 + 76\sqrt{14}$ from
$(13/2,7/2)$. All four labels lie in $\pi^{5}R$. By
Lemma~\ref{lem:coding}(a), applied to these two periodic sequences of
labels,
\begin{equation} \label{eq:twocycle}
  \begin{aligned}
    \frac{5 + 5\sqrt{14}}{2}
     & = \frac{\varepsilon(176 + 44\sqrt{14}) + (216 + 56\sqrt{14})}
         {\varepsilon^{2} - 1}, \\
    \frac{11 - 5\sqrt{14}}{2}
     & = -\frac{(288 - 76\sqrt{14}) + \varepsilon^{-1}(216 - 56\sqrt{14})}
          {1 - \varepsilon^{-2}}.
  \end{aligned}
\end{equation}
The number on the left of the first formula is the unstable coordinate of
$(5/2,5/2)$, and the number on the left of the second formula is the
stable coordinate of $(11/2,5/2)$. The formulas say how these coordinates
arise from the labels. The first formula is the sum of the
series~\eqref{eq:labelseries} for the labels
$176 + 44\sqrt{14},216 + 56\sqrt{14},176 + 44\sqrt{14},\ldots$ that follow
$(5/2,5/2)$, and the second formula is the sum of the series for the
labels $\ldots,216 + 56\sqrt{14},288 + 76\sqrt{14}$ that precede
$(11/2,5/2)$.

The finite calculation of \S\ref{sec:finitecalc} produces seven further
labels. We list them now, so that the critical orbit can be defined before
the calculation is described, and we explain after
Proposition~\ref{prop:finite} how the calculation finds them. Let
\begin{equation} \label{eq:labels}
  \begin{aligned}
    \ell_{-3} & = 200 + 52\sqrt{14}, &
    \ell_{-2} & = 304 + 80\sqrt{14}, &
    \ell_{-1} & = 216 + 56\sqrt{14}, &
    \ell_{0}  & = 280 + 72\sqrt{14}, \\
    \ell_{1}  & = 216 + 56\sqrt{14}, &
    \ell_{2}  & = 192 + 48\sqrt{14}, &
    \ell_{3}  & = 200 + 52\sqrt{14},
  \end{aligned}
\end{equation}
and for the remaining integers $k$ let
\[
  \ell_{k} =
  \begin{cases}
    176 + 44\sqrt{14} & \text{if $k \geq 4$ is even}, \\
    216 + 56\sqrt{14} & \text{if $k \geq 5$ is odd}, \\
    288 + 76\sqrt{14} & \text{if $k \leq -4$ is even}, \\
    216 + 56\sqrt{14} & \text{if $k \leq -5$ is odd}.
  \end{cases}
\]
Thus the sequence $(\ell_{k})_{k \in \Z}$ consists of the labels of the
period-$2$ orbit of $(11/2,5/2) + \Lambda$ for $k \leq -4$, the seven
labels~\eqref{eq:labels}, and the labels of the period-$2$ orbit of
$(5/2,5/2) + \Lambda$ for $k \geq 4$. All of these labels lie in
$\pi^{5}R$, and we write $\lambda_{k}$ for the coefficient point of
$\ell_{k}$. By Lemma~\ref{lem:coding}(b), there is a unique bounded
sequence $(\eta_{k})_{k \in \Z}$ of points of $\R^{2}$ with
$\eta_{k + 1} = \varepsilon\eta_{k} - \lambda_{k}$ for all $k$, and its
coordinates are given by the series~\eqref{eq:labelseries}. We call the
sequence $(\eta_{k})$ the \emph{critical orbit}. Its points
$\eta_{k} + \Lambda$ form the orbit of $\eta_{0} + \Lambda$ in
$\R^{2}/\Lambda$, and we use the same name for that orbit. We show below
that every $\eta_{k}$ lies in $D$, so that $\eta_{k}$ is the
representative in $D$ of $\eta_{k} + \Lambda$, and then the identity
$\eta_{k + 1} = \varepsilon\eta_{k} - \lambda_{k}$ says that the labels of
the orbit of $\eta_{0} + \Lambda$, in the sense of \S\ref{sec:coding}, are
the $\ell_{k}$. Because the labels $\ell_{k}$ agree with those of the
period-$2$ orbit of $(5/2,5/2) + \Lambda$ from $k = 4$ on, we have
$(\eta_{k})_{u} = (5 + 5\sqrt{14})/2$ for even $k \geq 4$ and
$(\eta_{k})_{u} = (3 + 7\sqrt{14})/2$ for odd $k \geq 5$, by the first
series in~\eqref{eq:labelseries}. Similarly
$(\eta_{k})_{s} = (11 - 5\sqrt{14})/2$ for odd $k \leq -3$ and
$(\eta_{k})_{s} = (13 - 7\sqrt{14})/2$ for even $k \leq -4$. For
$k \geq 4$ the labels of $\eta_{k}$ and of the period-$2$ orbit agree, so
the recursion for the stable coordinates gives
$(\eta_{k + 1})_{s} - c'_{s} = \varepsilon^{-1}((\eta_{k})_{s} - c_{s})$,
where $c$ is $(5/2,5/2)$ for even $k$ and $(3/2,7/2)$ for odd $k$, and
$c'$ is the other of these two points. Exact computation in $K$ gives
$(\eta_{4})_{s} - c_{s} = 0.9989\ldots$ for $c = (5/2,5/2)$ and
$(\eta_{5})_{s} - c_{s} = 0.0333\ldots$ for $c = (3/2,7/2)$, and the
second number is $\varepsilon^{-1}$ times the first. Hence
\begin{equation} \label{eq:etabound}
  \lvert(\eta_{k})_{s} - c_{s}\rvert
  = \varepsilon^{-(k - 4)}\,\lvert(\eta_{4})_{s} - (5/2,5/2)_{s}\rvert
  < \varepsilon^{-(k - 4)}
  \qquad\text{for $k \geq 4$},
\end{equation}
where $c$ is $(5/2,5/2)$ for even $k$ and $(3/2,7/2)$ for odd $k$. Thus
the stable coordinates of $\eta_{k}$ and of the period-$2$ orbit differ by
an amount that tends to $0$, their unstable coordinates agree, and the
distance from $\eta_{k} + \Lambda$ to the period-$2$ orbit of
$(5/2,5/2) + \Lambda$ tends to $0$ as $k \to \infty$. In the same way,
with the unstable coordinates in place of the stable ones and the backward
recursion
$(\eta_{k - 1})_{u} - c'_{u} = \varepsilon^{-1}((\eta_{k})_{u} - c_{u})$
for $k \leq -3$, we have
$\lvert(\eta_{k})_{u} - c_{u}\rvert < \varepsilon^{-(-k - 3)}$ for
$k \leq -3$, where $c$ is $(11/2,5/2)$ for odd $k$ and $(13/2,7/2)$ for
even $k$, because $(\eta_{-3})_{u} - (11/2,5/2)_{u} = -0.9989\ldots$, and
the distance from $\eta_{k} + \Lambda$ to the period-$2$ orbit of
$(11/2,5/2) + \Lambda$ tends to $0$ as $k \to -\infty$. In the language of
the theory of dynamical systems, the critical orbit is a
\emph{heteroclinic orbit} in $\R^{2}/\Lambda$, an orbit that approaches
one periodic orbit in forward time and another periodic orbit in backward
time.

We now show that every $\eta_{k}$ lies in $D$. The coefficient point of a
point $\rho$ is
$((\rho_{u} + \rho_{s})/2,\,(\rho_{u} - \rho_{s})/(2\sqrt{14}))$, so a
change of one coordinate of $\rho$ by an amount $d$ changes each
coefficient of $\rho$ by less than $d$. For $k \geq 20$ and for
$k \leq -19$, the bounds above show that the coefficient point of
$\eta_{k}$ lies within $\varepsilon^{-16} < 10^{-23}$ of one of the four
points $(5/2,5/2)$, $(3/2,7/2)$, $(11/2,5/2)$, and $(13/2,7/2)$, each of
which has distance $1/2$ from the boundary of $D$, so $\eta_{k}$ lies in
the interior of $D$. For $-20 \leq k \leq 20$, the coefficient point of
$\eta_{k}$ lies in $D$ by exact comparison in $K$. The points
$\eta_{-5},\ldots,\eta_{5}$ are shown in Figure~\ref{fig:clusters}.

The series~\eqref{eq:labelseries} for $\eta_{0}$ can be summed. The labels
$\ell_{k}$ with $k \geq 4$ give $\varepsilon^{-4}(5 + 5\sqrt{14})/2$, and
the labels $\ell_{-k}$ with $k \geq 4$ give
$\varepsilon^{-3}(11 - 5\sqrt{14})/2$, so
\begin{equation} \label{eq:etacoords}
  \begin{aligned}
    (\eta_{0})_{u}
                   & = \varepsilon^{-1}\ell_{0} + \varepsilon^{-2}\ell_{1}
    + \varepsilon^{-3}\ell_{2} + \varepsilon^{-4}\ell_{3}
    + \varepsilon^{-4}\,\frac{5 + 5\sqrt{14}}{2}
    = \frac{-805515 + 215293\sqrt{14}}{2}, \\
    (\eta_{0})_{s} & = -\bar{\ell}_{-1} - \varepsilon^{-1}\bar{\ell}_{-2}
    - \varepsilon^{-2}\bar{\ell}_{-3} + \varepsilon^{-3}\,\frac{11 - 5\sqrt{14}}{2}
    = \frac{26893 - 7191\sqrt{14}}{2}.
  \end{aligned}
\end{equation}
The coefficient point of $\eta_{0}$ is $(6.0963\ldots,\,3.4010\ldots)$.
Let $q = 4 + 4\sqrt{14}$, and define
\begin{equation} \label{eq:UV}
  \begin{aligned}
    U & = q - (\eta_{0})_{u} = \frac{805523 - 215285\sqrt{14}}{2} = 0.1447\ldots, \\
    V & = (\eta_{0})_{s} - \bar{q} = \frac{26885 - 7183\sqrt{14}}{2} = 4.3375\ldots.
  \end{aligned}
\end{equation}
Thus $(V,-U)$ is the pair of stable and unstable coordinates of the
translated point $\eta_{0} - q$. The quotient $q$ has valuation $4$
by~\eqref{eq:valuation}, and it is the fixed point of the affine map
$x \mapsto \varepsilon x - \ell_{0}$, because
$\ell_{0} = 280 + 72\sqrt{14} = (\varepsilon - 1)q$. Exact comparison in
$K$ gives
\begin{equation} \label{eq:symmetry}
  (\eta_{0})_{s}
  = 8 - (\eta_{1})_{u},\qquad
  (\eta_{1})_{s} = 8 - (\eta_{0})_{u},\qquad V = \varepsilon U.
\end{equation}
The third identity follows from the first, because $q + \bar{q} = 8$ and
$(\eta_{1})_{u} = \varepsilon(\eta_{0})_{u}
- \ell_{0}$ give
\[
  V = 8 - (\eta_{1})_{u} - \bar{q} = q - \varepsilon(\eta_{0})_{u} + \ell_{0}
  = \varepsilon q - \varepsilon(\eta_{0})_{u} = \varepsilon U.
\]
Remark~\ref{rem:symmetry} explains the first two identities by a symmetry
of the critical orbit. We will see in Lemma~\ref{lem:cornerbound} that the
point $\eta_{0} - q$ is the corner of a rectangle that contains the
translated points $\rho - q$ of all the orbits that concern us.

\subsection{The finite calculation} \label{sec:finitecalc}
We carry out the construction of \S\ref{sec:closedwalks} at the resolution
$L = 2^{10}$. For a square $C_{L}(i,j)$ we write $(i,j)$ for the vertex,
and for a vertex $\sigma$ we write $-\sigma$ for the square obtained from
it by negation, which acts on the grid by
$(i,j) \mapsto (8L - 1 - i,\,4L - 1 - j)$.

\begin{proposition} \label{prop:finite}
  The construction at the resolution $L = 2^{10}$ has the following
  properties.
  \begin{enumerate}
    \item[(1)] The final graph has $432$ vertices and $2106$ edges. It has
      eleven components, each of which contains a closed walk. One of them
      consists of the four squares at the corners of $D$, which are
      incident to the point $0$ of $\mathcal{F}$, and eight of them
      consist of the eight squares incident to the two points of an orbit
      of period $2$ in $\mathcal{F}$, one orbit for each component. Thus
      every vertex of these nine components is a square incident to
      $\mathcal{F}$. The other two components have $182$ vertices each.
      Negation maps the final graph onto itself and exchanges these two
      components. We call the one whose vertices lie in $y \geq 2 + 1/L$
      the \emph{upper component}. The vertices of the other lie in
      $y \leq 2 - 1/L$.
    \item[(2)] The vertices of the upper component at distance less than
      $1/100$ from $\eta_{0}$ are the four squares $(6241 + i,\,3482)$
      with $0 \leq i \leq 3$, which we call the \emph{critical squares}.
      Every other vertex of the upper component is at distance greater
      than $1/5$ from $\eta_{0}$. Every point $\rho$ of a closed critical
      square satisfies $\rho_{u} < q$ and $\rho_{s} > \bar{q}$, where
      $q = 4 + 4\sqrt{14}$.
    \item[(3)] Every walk of length $4$ in the final graph that starts at
      a critical square has the labels
      $\ell_{0},\ell_{1},\ell_{2},\ell_{3}$ of~\eqref{eq:labels}, in this
      order, and every walk of length $3$ that ends at a critical square
      has the labels $\ell_{-3},\ell_{-2},\ell_{-1}$.
    \item[(4)] Every vertex of a closed walk in the final graph that
      passes through no critical square $\sigma$ and no square $-\sigma$
      is a square incident to $\mathcal{F}$, and every edge of such a
      walk, from a square that contains the representative $c$ of a point
      of $\mathcal{F}$ to a square that contains the representative $c'$,
      has the label whose coefficient point is $\varepsilon c - c'$.
    \item[(5)] If a square of the upper component contains a point $\rho$
      with $\rho_{u} < (5 + 5\sqrt{14})/2$, then every edge from this
      square has the label $176 + 44\sqrt{14}$, and every edge from the
      end of such an edge has the label $216 + 56\sqrt{14}$. If a square
      of the upper component contains a point $\rho$ with
      $\rho_{s} > (11 - 5\sqrt{14})/2$, then every edge into this square
      has the label $288 + 76\sqrt{14}$, and every edge into the start of
      such an edge has the label $216 + 56\sqrt{14}$.
  \end{enumerate}
\end{proposition}

\begin{proof}
  This is a finite calculation, which we carried out on a computer in
  exact integer and rational arithmetic. The calculation tests all
  $2^{25}$ squares of $D$ with~\eqref{eq:rationaltest} and finds $527616$
  uncovered squares. The test is carried out column by column. Fix a unit
  square $[a_{0},a_{0} + 1] \times [b_{0},b_{0} + 1]$ with integers
  $0 \leq a_{0} < 8$ and $0 \leq b_{0} < 4$, a column of the squares
  $C_{L}(i,j)$ in it, with $i$ fixed, and a tested quotient $q$. All
  squares in this unit square have the same tested
  quotients~\eqref{eq:qbox}, and $Y^{2}$ is monotone on every interval
  $I_{2}$ of~\eqref{eq:cellbound} that occurs in the column, so the
  squares of the column that $q$ covers have consecutive indices $j$,
  which are found from the first and the last of them. As a check, each
  retained square and each square in a neighborhood of the critical
  squares was also tested one square at a time, directly
  from~\eqref{eq:rationaltest}, with the same result. The calculation then
  constructs the edges between the uncovered squares, retains the $432$
  squares that lie on closed walks, computes the components of the final
  graph, and verifies (1).

  For (2), we compute the distance from each closed square of the upper
  component to $\eta_{0}$, whose coordinates are the elements of $K$
  in~\eqref{eq:etacoords}, by exact comparison in $K$, and we compare the
  coordinates of the corners of the critical squares with $q$ and
  $\bar{q}$. Since $\rho_{u} = x + y\sqrt{14}$ increases with both
  coordinates of $\rho = (x,y)$, and $\rho_{s} = x - y\sqrt{14}$ increases
  with $x$ and decreases with $y$, it is enough to compare $\rho_{u}$ at
  the upper right corner of each critical square with $q$ and $\rho_{s}$
  at the upper left corner with $\bar{q}$. For (3), we list the labels of
  all walks of length $4$ from the critical squares and of all walks of
  length $3$ into them. For (4), we remove the four critical squares and
  their negatives, compute the components of the remaining graph, and
  examine the components that contain a closed walk. Every vertex of these
  components is incident to $\mathcal{F}$, and by
  Lemma~\ref{lem:exceptional} it contains exactly one representative $c$
  of a point of $\mathcal{F}$. We verify that every edge between two
  vertices of one of these components has the label with coefficient point
  $\varepsilon c - c'$. Since a closed walk that avoids the critical
  squares and their negatives lies in one of these components, this proves
  (4). For (5), we find the squares of the upper component whose lower
  left corner has unstable coordinate less than $(5 + 5\sqrt{14})/2$ and
  the squares whose lower right corner has stable coordinate greater than
  $(11 - 5\sqrt{14})/2$, and we examine the labels of the edges that leave
  or enter them. There are eleven squares of each kind. This proves the
  proposition.
\end{proof}

The two numbers in Proposition~\ref{prop:finite}(2) are chosen as follows.
The number $1/100$ isolates the four squares nearest to $\eta_{0}$, since
the next nearest vertex of the upper component is at distance
$0.2395\ldots$ from $\eta_{0}$. The number $1/5$ is a convenient bound
below this distance and below the distance $0.4036\ldots$ from $\eta_{0}$
to $(13/2,7/2)$, the nearest representative of a point of $\mathcal{F}$,
which is what the proof of Lemma~\ref{lem:passage} needs.

The labels~\eqref{eq:labels} were found as follows. The thirty-two points
of $\mathcal{F}$ with half-integral coordinates form sixteen orbits of
period $2$. The nine small components of the final graph consist of the
squares incident to the point $0$ and to eight of these sixteen orbits.
The remaining eight orbits are those of $(5/2,5/2) + \Lambda$,
$(7/2,5/2) + \Lambda$, $(9/2,5/2) + \Lambda$, $(11/2,5/2) + \Lambda$, and
their negatives, and the squares incident to them belong to the two large
components. In the upper component, the squares incident to $(11/2,5/2)$
or $(13/2,7/2)$ can be joined to the squares incident to $(5/2,5/2)$ or
$(3/2,7/2)$ by a walk. The shortest such walk has eleven edges, and its
labels are $\ell_{-5},\ldots,\ell_{5}$. That is, after two labels of the
period-$2$ orbit of $(11/2,5/2) + \Lambda$ and before two labels of the
period-$2$ orbit of $(5/2,5/2) + \Lambda$, the walk has the seven
labels~\eqref{eq:labels}. Thus the critical orbit is the orbit that leaves
the period-$2$ orbit of $(11/2,5/2) + \Lambda$ and arrives at the
period-$2$ orbit of $(5/2,5/2) + \Lambda$ along this walk. We do not claim
that it is the only orbit with this property, and the proof below does not
need this.

Figure~\ref{fig:uncovered} shows the squares that remain after the two
steps of the calculation. The covering test leaves $527616$ of the
$2^{25}$ squares of $D$ uncovered, and only $432$ of these are retained.
Apart from the four squares at the corners of $D$, which are incident to
the point $0$ of $\mathcal{F}$, the uncovered squares form thin regions
near the lines $y = b + 1/2$ with $0 \leq b \leq 3$, and every other point
of $\mathcal{F}$ lies in one of these regions. By
Lemma~\ref{lem:exceptional}, the squares incident to $\mathcal{F}$ remain
uncovered at every resolution. So the covering test alone can never finish
the calculation, and this is why we need the graph of \S\ref{sec:thegraph}
and the coding of Lemma~\ref{lem:coding}. Part (a) of the figure is
symmetric under negation, which exchanges the two large components. Part
(b) shows the strip $2.2 \leq y \leq 3.8$, which contains the upper
component. Its vertices form the small clusters of
Figure~\ref{fig:clusters}.

\begin{figure}[tbp]
  \centering
  \begin{tikzpicture}[x=1.85cm,y=1.85cm,
  every node/.style={font=\footnotesize}]
    % The fundamental domain D = [0,8) x [0,4) at the resolution 2^10. The
    % image has one pixel for each square of D. Light gray: the uncovered
    % squares. Black: the vertices of the two large components of the final
    % graph, drawn enlarged. Small circles: the points of F (the point 0 at
    % the four corners). The images window_uncovered_a.png and
    % window_uncovered_b.png are drawn by a separate program from
    % the output of the calculation.
    \node[anchor=south west,inner sep=0pt] at (0,0)
    {\includegraphics[width=14.8cm,height=7.4cm]{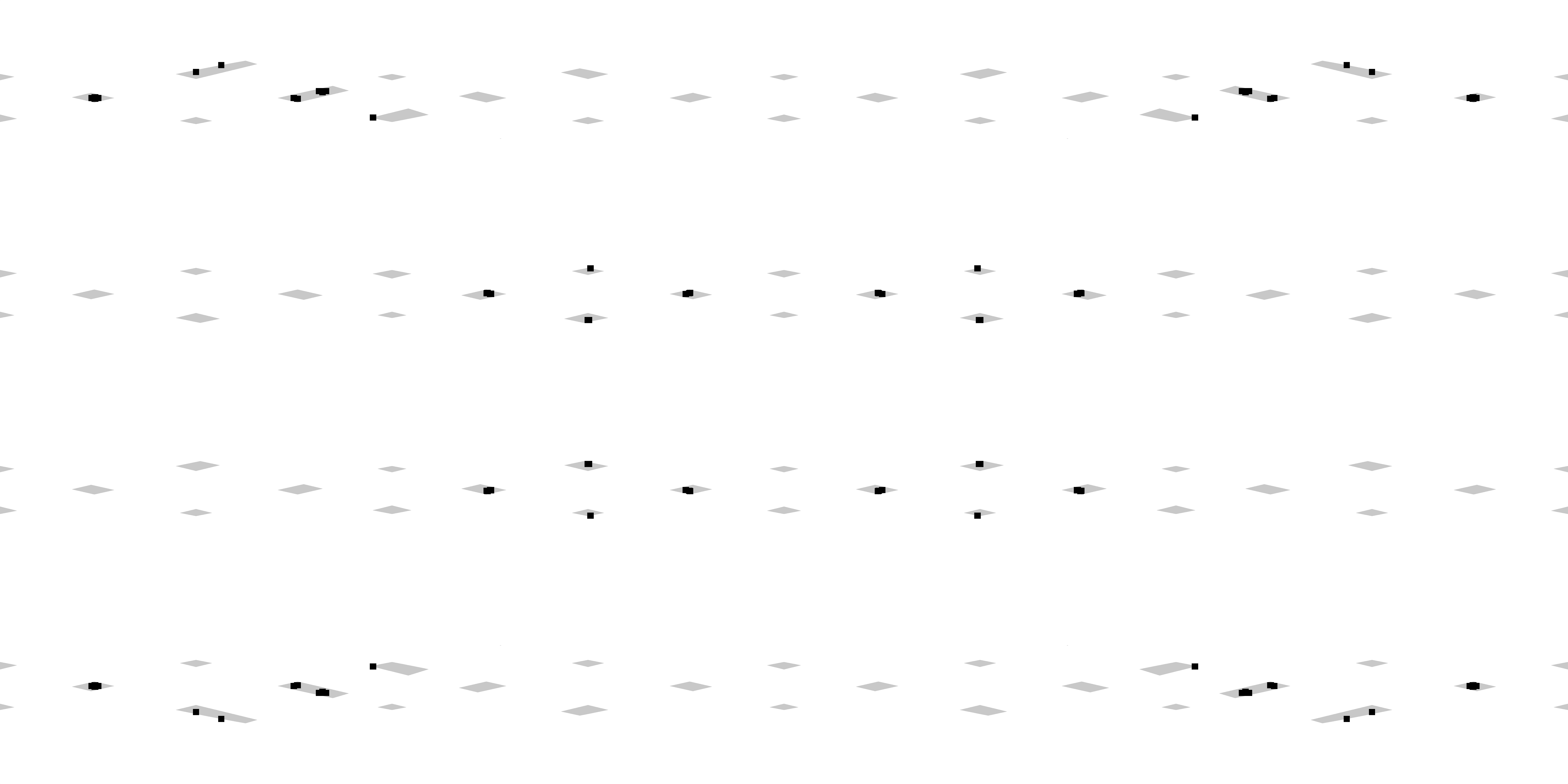}};
    \draw[thin] (0,0) rectangle (8,4); \draw[densely dotted] (0,2) --
    (8,2); \foreach \a in {0,...,7} { \foreach \b in {0,...,3} {
    \draw[thin] (\a+0.5,\b+0.5) circle (1.3pt); } } \foreach \a in {0,8} {
    \foreach \b in {0,4} { \draw[thin] (\a,\b) circle (1.3pt); } }
    \node[below] at (0,0) {$0$}; \node[below] at (4,0) {$4$}; \node[below]
    at (8,0) {$8$}; \node[left] at (0,2) {$2$}; \node[left] at (0,4)
    {$4$};
  \end{tikzpicture}

  (a)

  \vspace{2ex}
  \begin{tikzpicture}[x=1.85cm,y=4.4cm,
  every node/.style={font=\footnotesize}]
    % The strip 2.2 <= y <= 3.8 of D, with the vertical direction stretched
    % as in Figure (fig:clusters). Light gray: the uncovered squares. Black:
    % the vertices of the upper component of the final graph, drawn
    % enlarged. Small circles: the points of F.
    \begin{scope}
      \clip (0,2.2) rectangle (8,3.8); \node[anchor=south west,inner
      sep=0pt] at (0,2.1875)
      {\includegraphics[width=14.8cm,height=7.15cm]{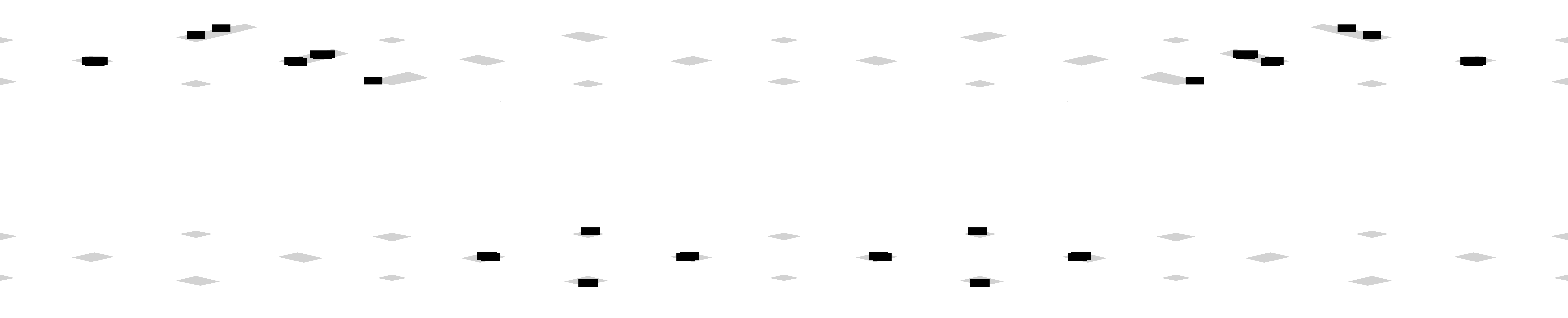}};
    \end{scope}
    \draw[thin] (0,2.2) rectangle (8,3.8); \draw[densely dotted] (4,2.2)
    -- (4,3.8); \foreach \a in {0,...,7} { \foreach \b in {2,3} {
    \draw[thin] (\a+0.5,\b+0.5) circle (1.3pt); } }
    % Enlarge the locations of the critical squares and their reflections.
    \draw[thick] (1.904,3.401) circle (0.064); \draw[thick] (6.096,3.401)
    circle (0.064); \node[below left] at (1.904,3.401) {$\eta_{1}$};
    \node[below right] at (6.096,3.401) {$\eta_{0}$};
    % Key in the empty space between the two bands of uncovered squares.
    \fill[white] (2.53,2.92) rectangle (5.47,3.08); \draw[gray!60,thin]
    (2.53,2.92) rectangle (6.2,3.08); \fill[gray!25] (2.64,2.969)
    rectangle (2.73,3.031); \node[anchor=west,font=\scriptsize] at
    (2.78,3) {uncovered}; \fill[black] (3.92,2.969) rectangle
    (4.01,3.031); \node[anchor=west,font=\scriptsize] at (4.06,3)
    {vertices of the final graph}; \node[left] at (0,2.2) {$2.2$};
    \node[left] at (0,3.8) {$3.8$}; \node[below] at (0,2.2) {$0$};
    \node[below] at (4,2.2) {$4$}; \node[below] at (8,2.2) {$8$};
  \end{tikzpicture}

  (b) \caption{(a) The fundamental domain $D$ at the resolution $2^{10}$.
  The $527616$ uncovered squares are light gray, the vertices of the two
  large components of the final graph are black, and the points of
  $\mathcal{F}$ are small circles (the point $0$ appears at the four
  corners). The vertices of the other nine components are squares incident
  to $\mathcal{F}$. (b) The strip $2.2 \leq y \leq 3.8$ of $D$, with the
  vertical direction stretched as in Figure~\ref{fig:clusters}. The
  uncovered squares are light gray, and the vertices of the upper
  component of the final graph are black. The circles surround the
  critical squares, near $\eta_{0}$, and their reflections, near
  $\eta_{1}$. In both parts the vertices are drawn enlarged so that they
  are visible.} \label{fig:uncovered}
\end{figure}

Figure~\ref{fig:clusters} shows the vertices of the upper component of the
final graph. They lie in twenty small clusters. Six clusters surround the
points $\eta_{-3}$, $\eta_{-1}$, $\eta_{0}$, $\eta_{1}$, $\eta_{2}$, and
$\eta_{4}$ of the critical orbit. Eight clusters contain the squares
incident to the eight points of $\mathcal{F}$ in the upper component, and
four of these, the clusters at $(1/2,7/2)$, $(3/2,7/2)$, $(13/2,7/2)$, and
$(15/2,7/2)$, also contain the squares around the nearby points
$\eta_{3}$, $\eta_{5}$, $\eta_{-4}$, and $\eta_{-2}$. Two clusters, of
thirteen squares each, lie near a second orbit, which leaves the
period-$2$ orbit of $(5/2,5/2) + \Lambda$ and arrives at the period-$2$
orbit of $(11/2,5/2) + \Lambda$. This second orbit is the orbit along
which the classes of \S\ref{sec:family} return from the first of these
period-$2$ orbits to the second, and we will meet it there as the orbit
through the homoclinic point of displacement $\bar{\pi}$. The remaining
four clusters, of one or two squares each, lie near the points $(1,3.63)$,
$(1.13,3.67)$, $(6.87,3.67)$, and $(7,3.63)$, far from $\mathcal{F}$. By
Proposition~\ref{prop:finite}(4), every closed walk through them passes
through the critical squares or their negatives. The proof of
Theorem~\ref{thm:main} uses only the properties listed in
Proposition~\ref{prop:finite}, and not this description of the clusters.

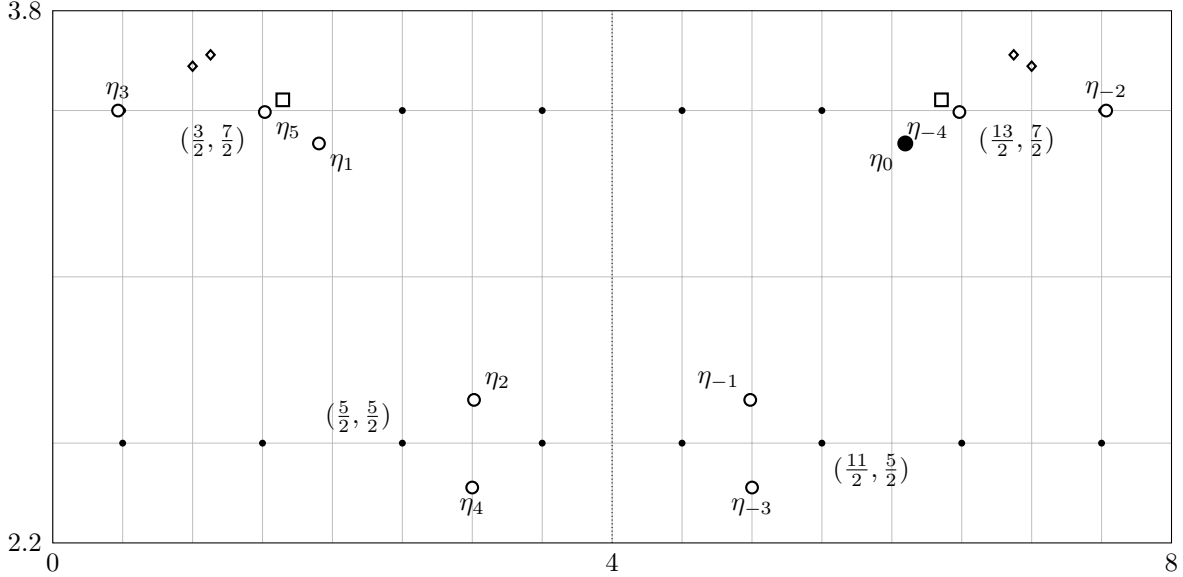
\begin{figure}[tbp]
  \centering
  \begin{tikzpicture}[x=1.85cm,y=4.4cm,
  every node/.style={font=\footnotesize}]
    % The part 2.2 <= y <= 3.8 of the fundamental domain D = [0,8) x [0,4),
    % which contains all vertices of the upper component. Black dots: the
    % points of F. Circles: the points eta_k of the critical orbit, at the
    % positions computed in Section (sec:etadef). Rectangles: the clusters
    % near the second orbit. Diamonds: the four small clusters far from F.
    \begin{scope}
      \clip (-0.05,2.2) rectangle (8.05,3.8); \draw[gray!50,very thin]
      (0,2) grid[step=0.5] (8,4);
    \end{scope}
    \draw[thin] (0,2.2) rectangle (8,3.8); \foreach \a in {0,...,7} {
    \foreach \b in {2,3} { \fill (\a+0.5,\b+0.5) circle (1.3pt); } }
    \draw[densely dotted] (4,2.2) -- (4,3.8);
    % The points eta_k, k = -4, ..., 5, of the critical orbit, and the
    % second orbit.
    \coordinate (gm4) at (6.4833,3.4955); \coordinate (gm3) at
    (5.0005,2.3665); \coordinate (gm2) at (7.5329,3.4999); \coordinate
    (gm1) at (4.9876,2.6300); \coordinate (g0) at (6.0964,3.4010);
    \coordinate (g1) at (1.9036,3.4010); \coordinate (g2) at
    (3.0124,2.6300); \coordinate (g3) at (0.4671,3.4999); \coordinate (g4)
    at (2.9995,2.3665); \coordinate (g5) at (1.5167,3.4955); \coordinate
    (h1) at (1.6450,3.5322); \coordinate (h2) at (6.3550,3.5322); \foreach
    \k in {gm4,gm3,gm2,gm1,g1,g2,g3,g4,g5} \draw[fill=white,thick] (\k)
    circle (2.2pt); \draw[fill=black,thick] (g0) circle (2.6pt); \foreach
    \k in {h1,h2} \draw[thick] ($(\k)+(-0.045,-0.019)$) rectangle
    ($(\k)+(0.045,0.019)$);
    % The four small clusters far from F.
    \foreach \p in
    {(1.0000,3.6333),(1.1284,3.6675),(6.8716,3.6675),(7.0000,3.6333)}
    \draw[thick] \p ++(0,0.014) -- ++(0.030,-0.014) -- ++(-0.030,-0.014)
    -- ++(-0.030,0.014) -- cycle; \node[below left] at (gm4)
    {$\eta_{-4}$}; \node[below] at (gm3) {$\eta_{-3}$}; \node[above] at
    (gm2) {$\eta_{-2}$}; \node[above left] at (gm1) {$\eta_{-1}$};
    \node[below left] at (g0) {$\eta_{0}$}; \node[below right] at (g1)
    {$\eta_{1}$}; \node[above right] at (g2) {$\eta_{2}$}; \node[above] at
    (g3) {$\eta_{3}$}; \node[below] at (g4) {$\eta_{4}$}; \node[below
    right] at (g5) {$\eta_{5}$}; \node[above left] at (2.5,2.5)
    {$(\frac{5}{2},\frac{5}{2})$}; \node[below left] at (1.46,3.49)
    {$(\frac{3}{2},\frac{7}{2})$}; \node[below right] at (5.5,2.5)
    {$(\frac{11}{2},\frac{5}{2})$}; \node[below right] at (6.54,3.49)
    {$(\frac{13}{2},\frac{7}{2})$}; \node[left] at (0,2.2) {$2.2$};
    \node[left] at (0,3.8) {$3.8$}; \node[below] at (0,2.2) {$0$};
    \node[below] at (4,2.2) {$4$}; \node[below] at (8,2.2) {$8$};
  \end{tikzpicture}
  \caption{The strip $2.2 \leq y \leq 3.8$ of $D$, with the vertical
  direction stretched, which contains all vertices of the upper component
  of the final graph. The dots are the points of $\mathcal{F}$. The
  circles mark the points $\eta_{-4},\ldots,\eta_{5}$ of the critical
  orbit, around which the vertices cluster. The black circle marks the
  critical squares, which surround $\eta_{0}$. The points $\eta_{k}$ for
  $k \geq 6$ and $k \leq -5$ are too close to the period-$2$ orbits of
  $(5/2,5/2) + \Lambda$ and $(11/2,5/2) + \Lambda$ to be drawn separately.
  The two rectangles mark the clusters near the second orbit of
  \S\ref{sec:family}, and the four diamonds mark the four small clusters
  far from $\mathcal{F}$.} \label{fig:clusters}
\end{figure}

\begin{remark} \label{rem:resolution}
  The resolution $L = 2^{10}$ was found by doubling $L$, starting from
  small values, until the final graph had the shape described in
  Proposition~\ref{prop:finite}, as explained in \S\ref{sec:closedwalks}.
  A larger resolution gives a final graph whose vertices lie in smaller
  clusters around the same points, and the properties of
  Proposition~\ref{prop:finite} persist, with the critical squares
  replaced by the squares around $\eta_{0}$ at the larger resolution. We
  verified this at the resolution $2^{11}$, where the final graph again
  has $432$ vertices, with $2014$ edges, and the two large components
  again have $182$ vertices each. A larger resolution costs only computing
  time, and it is not needed for the proof.
\end{remark}

\subsection{The necessary passage} \label{sec:passage}
\begin{lemma} \label{lem:passage}
  Let $\rho + \Lambda$ be a periodic point of $\mathcal{X}$ whose orbit is
  not contained in $\mathcal{F}$. Then the orbit of $\rho + \Lambda$ or
  the orbit of $-\rho + \Lambda$ has a point whose representative in $D$
  is assigned to a critical square.
\end{lemma}
\begin{proof}
  Let $\rho_{k}$ be the representatives in $D$ of the points of the orbit
  of $\rho + \Lambda$, and let $\lambda_{k} \in \Lambda$ be the
  coefficient points of its labels, so that
  $\rho_{k + 1} = \varepsilon\rho_{k} - \lambda_{k}$. By
  Lemma~\ref{lem:walks}, the itinerary of the orbit at the resolution
  $L = 2^{10}$ runs along a closed walk in the final graph, and by
  Lemma~\ref{lem:edges} the $\lambda_{k}$ are the coefficient points of
  the labels of its edges.

  Suppose first that the itinerary passes through no critical square
  $\sigma$ and no square $-\sigma$. By Proposition~\ref{prop:finite}(4),
  the square of $\rho_{k}$ contains the representative $c_{k}$ of a point
  of $\mathcal{F}$, and $\lambda_{k} = \varepsilon c_{k} - c_{k + 1}$ for
  every $k$. The points $c_{k}$ lie in the closure of $D$, so they are
  bounded, and the points $\rho_{k}$ lie in $D$. By
  Lemma~\ref{lem:coding}(c), $\rho_{k} = c_{k}$ for every $k$. Then the
  orbit of $\rho + \Lambda$ is contained in $\mathcal{F}$, contrary to the
  hypothesis. Hence the itinerary passes through a critical square
  $\sigma$ or through its negative $-\sigma$.

  If it passes through a critical square, we are done. Suppose that some
  $\rho_{k} = (x,y)$ is assigned to the square $-\sigma$, where
  $\sigma = (6241 + i,\,3482)$ is critical. Then $x < (8L - 6241 - i)/L$
  and $y < (4L - 3482)/L$, so the point $(8 - x,\,4 - y)$ lies in $D$, and
  it is a representative of $-\rho_{k} + \Lambda$. It lies in the closed
  box $[6241/L,\,6245/L] \times [3482/L,\,3483/L]$, which is the union of
  the four closed critical squares. By Proposition~\ref{prop:finite}(2),
  every point of this box is at distance less than $1/100 + 4/L$ from
  $\eta_{0}$. By Lemma~\ref{lem:X}(a), $-\rho + \Lambda$ is a periodic
  point of $\mathcal{X}$, so by Lemma~\ref{lem:walks} the square assigned
  to $(8 - x,4 - y)$ is a vertex of the final graph. It contains a point
  at distance less than $1/100 + 4/L < 1/5$ from $\eta_{0}$. The
  coefficient point $(6.0963\ldots,3.4010\ldots)$ of $\eta_{0}$ is at
  distance greater than $1/5 + 1/L$ from every representative of a point
  of $\mathcal{F}$, the nearest being $(13/2,7/2)$, and from the region
  $y \leq 2$, by exact comparison in $K$. Hence this vertex is neither a
  square incident to $\mathcal{F}$ nor a square of the lower large
  component, so it belongs to the upper component, and by
  Proposition~\ref{prop:finite}(2) it is a critical square. This proves
  the lemma.
\end{proof}

\subsection{The period-\texorpdfstring{$2$}{2} orbits are extremal} \label{sec:minimum}
The next lemma gives a lower bound for the unstable coordinates and an
upper bound for the stable coordinates of all points of the orbits that
pass through the upper component. Its mechanism is that the affine map
$x \mapsto \varepsilon^{2}x - \varepsilon(176 + 44\sqrt{14})
- (216 + 56\sqrt{14})$, which multiplication by $\varepsilon^{2}$ induces
on the unstable coordinates of the points near the period-$2$ orbit of
$(5/2,5/2) + \Lambda$, moves every point other than its fixed point
$(5 + 5\sqrt{14})/2$ away from the fixed point. Cavallar and
Lemmermeyer use the same mechanism~\cite[Proposition~4.2,
  p.~340]{CavallarLemmermeyer2000}. In their proposition a single map
$\xi \mapsto \varepsilon\xi - \ell$ sends every point of a set $S$ either
into $S$ or to a point $\rho$ with $|N(\rho - q)| < c$ for some
$q \in R$, where $c > 0$ is fixed. For a real quadratic field their
conclusion is that every point of $S$ whose orbit contains no such point
$\rho$ has the same unstable coordinate as the fixed point. Here two maps
alternate, and we obtain a one-sided bound instead.

\begin{lemma} \label{lem:minimum}
  Let $\rho + \Lambda$ be a periodic point of $\mathcal{X}$ whose orbit
  has an itinerary at the resolution $2^{10}$ that runs along a closed
  walk in the upper component, and let $\rho_{k}$ be the representatives
  in $D$ of the points of its orbit. Then
  \[
    (\rho_{k})_{u} \geq \frac{5 + 5\sqrt{14}}{2}\quad\text{and}\quad
    (\rho_{k})_{s} \leq \frac{11 - 5\sqrt{14}}{2}\qquad(k \in \Z).
  \]
\end{lemma}
\begin{proof}
  The point is periodic, so the sequence $(\rho_{k})$ takes finitely many
  values. Choose $k$ such that $(\rho_{k})_{u}$ is the least of the
  unstable coordinates of these values, and suppose, by way of
  contradiction, that $(\rho_{k})_{u} < (5 + 5\sqrt{14})/2$. The square of
  $\rho_{k}$ is a vertex of the upper component, and it contains the point
  $\rho_{k}$. By Proposition~\ref{prop:finite}(5), the labels $\ell_{k}$
  and $\ell_{k + 1}$ of the orbit are $176 + 44\sqrt{14}$ and
  $216 + 56\sqrt{14}$. Taking unstable coordinates in
  $\rho_{k + 2} = \varepsilon\rho_{k + 1} - \lambda_{k + 1}
  = \varepsilon^{2}\rho_{k} - \varepsilon\lambda_{k} - \lambda_{k + 1}$ gives
  \[
    (\rho_{k + 2})_{u} = \varepsilon^{2}(\rho_{k})_{u}
    - \varepsilon(176 + 44\sqrt{14}) - (216 + 56\sqrt{14}).
  \]
  The first formula in~\eqref{eq:twocycle} says that
  \[
    \frac{5 + 5\sqrt{14}}{2} = \varepsilon^{2}\,\frac{5 + 5\sqrt{14}}{2}
    - \varepsilon(176 + 44\sqrt{14}) - (216 + 56\sqrt{14}).
  \]
  Subtracting,
  \[
    (\rho_{k + 2})_{u} - \frac{5 + 5\sqrt{14}}{2}
    = \varepsilon^{2}\left((\rho_{k})_{u} - \frac{5 + 5\sqrt{14}}{2}\right).
  \]
  The right side is negative and $\varepsilon^{2} > 1$, so
  $(\rho_{k + 2})_{u} < (\rho_{k})_{u}$. This contradicts the choice of
  $k$ and proves the first inequality.

  For the second inequality, choose $k$ such that $(\rho_{k})_{s}$ is the
  largest stable coordinate of a point of the orbit, and suppose that
  $(\rho_{k})_{s} > (11 - 5\sqrt{14})/2$. By
  Proposition~\ref{prop:finite}(5), the labels $\ell_{k - 1}$ and
  $\ell_{k - 2}$ of the orbit are $288 + 76\sqrt{14}$ and
  $216 + 56\sqrt{14}$. Taking stable coordinates in
  $\rho_{k} = \varepsilon^{2}\rho_{k - 2} - \varepsilon\lambda_{k - 2}
  - \lambda_{k - 1}$ gives
  \[
    (\rho_{k})_{s} = \varepsilon^{-2}(\rho_{k - 2})_{s}
    - \varepsilon^{-1}(216 - 56\sqrt{14}) - (288 - 76\sqrt{14}),
  \]
  and the second formula in~\eqref{eq:twocycle} says that
  \[
    \frac{11 - 5\sqrt{14}}{2} = \varepsilon^{-2}\,\frac{11 - 5\sqrt{14}}{2}
    - \varepsilon^{-1}(216 - 56\sqrt{14}) - (288 - 76\sqrt{14}).
  \]
  Subtracting,
  \[
    (\rho_{k - 2})_{s} - \frac{11 - 5\sqrt{14}}{2}
    = \varepsilon^{2}\left((\rho_{k})_{s} - \frac{11 - 5\sqrt{14}}{2}\right)
    > (\rho_{k})_{s} - \frac{11 - 5\sqrt{14}}{2},
  \]
  so $(\rho_{k - 2})_{s} > (\rho_{k})_{s}$. This contradicts the choice of
  $k$ and proves the lemma.
\end{proof}

\subsection{The critical bound and the number \texorpdfstring{$T$}{T}} \label{sec:Tvalue}
\begin{lemma} \label{lem:cornerbound}
  Let $\rho + \Lambda$ be a periodic point of $\mathcal{X}$ whose
  representative $\rho$ in $D$ is assigned to a critical square, and let
  $q = 4 + 4\sqrt{14}$. Then
  \begin{equation} \label{eq:cornerbound}
    -U \leq \rho_{u} - q < 0,\qquad 0 < \rho_{s} - \bar{q} \leq V,
  \end{equation}
  and consequently $|N(\rho - q)| \leq UV$.
\end{lemma}
\begin{proof}
  Let $\rho_{k}$ be the representatives in $D$ of the points of the orbit
  of $\rho + \Lambda$, with $\rho_{0} = \rho$, and let $\ell_{k}$ be the
  labels of the orbit. By Lemma~\ref{lem:walks}, the itinerary of the
  orbit runs along a closed walk in the final graph, and since this walk
  passes through a critical square, it lies in the upper component. By
  Proposition~\ref{prop:finite}(3), the labels $\ell_{0},\ldots,\ell_{3}$
  and $\ell_{-3},\ell_{-2},\ell_{-1}$ of the orbit are the labels
  in~\eqref{eq:labels} of the critical orbit. Four applications of
  $(\rho_{k})_{u} = \varepsilon^{-1}((\rho_{k + 1})_{u} + \ell_{k})$ give
  \[
    (\rho_{0})_{u} = \varepsilon^{-1}\ell_{0} + \varepsilon^{-2}\ell_{1}
    + \varepsilon^{-3}\ell_{2} + \varepsilon^{-4}\ell_{3}
    + \varepsilon^{-4}(\rho_{4})_{u},
  \]
  and by Lemma~\ref{lem:minimum},
  $(\rho_{4})_{u} \geq (5 + 5\sqrt{14})/2$. Comparing
  with~\eqref{eq:etacoords} gives
  $(\rho_{0})_{u} \geq (\eta_{0})_{u} = q - U$. Similarly, three
  applications of $(\rho_{k})_{s} = \varepsilon^{-1}(\rho_{k - 1})_{s}
  - \bar{\ell}_{k - 1}$ give
  \[
    (\rho_{0})_{s} = -\bar{\ell}_{-1} - \varepsilon^{-1}\bar{\ell}_{-2}
    - \varepsilon^{-2}\bar{\ell}_{-3} + \varepsilon^{-3}(\rho_{-3})_{s},
  \]
  and $(\rho_{-3})_{s} \leq (11 - 5\sqrt{14})/2$ by
  Lemma~\ref{lem:minimum}, so
  $(\rho_{0})_{s} \leq (\eta_{0})_{s} = \bar{q} + V$
  by~\eqref{eq:etacoords}. The inequalities $\rho_{u} < q$ and
  $\rho_{s} > \bar{q}$ are part of Proposition~\ref{prop:finite}(2).
  Finally $|N(\rho - q)| = |\rho_{u} - q|\,|\rho_{s} - \bar{q}| \leq UV$.
  This proves the lemma.
\end{proof}

Thus the translated points $\rho - q$ of the representatives $\rho$ at the
critical squares lie in the rectangle~\eqref{eq:cornerbound} in the
coordinates $(s,u)$, whose corner $(V,-U)$ is the translated point
$\eta_{0} - q$ of the critical orbit. We now define
\begin{equation} \label{eq:nu}
  \nu = UV = \varepsilon U^{2}
  = \frac{43305976025 - 11574008934\sqrt{14}}{4} = 0.6278\ldots,
  \qquad T = 2\nu^{-1/4}.
\end{equation}
Thus $T$ is obtained from the series~\eqref{eq:etacoords}, whose terms are
the labels of the critical orbit and the coordinates of the two period-$2$
orbits. The closed form of $T$ in Theorem~\ref{thm:main} follows from the
relation $\pi^{2} = 2\varepsilon$:
\[
  \nu = \varepsilon U^{2} = \frac{(U\pi)^{2}}{2},\qquad
  U\pi = \frac{(805523 - 215285\sqrt{14})(4 + \sqrt{14})}{2}
  = 104051 - \frac{55617}{2}\sqrt{14} > 0.
\]
Since $\sqrt{14}/\sqrt{2} = \sqrt{7}$, it follows that
\[
  \sqrt{\nu} = \frac{U\pi}{\sqrt{2}}
  = \frac{104051\sqrt{2} - 55617\sqrt{7}}{2},
\]
and therefore
\begin{equation} \label{eq:T}
  T = 2\nu^{-1/4} = \frac{2}{\sqrt{\sqrt{\nu}}}
  = \sqrt{\frac{8}{104051\sqrt{2} - 55617\sqrt{7}}}.
\end{equation}
In particular $(T/2)^{4} = 1/\nu$. Exact comparison in $K$ gives
\begin{equation} \label{eq:nubounds}
  \left(\frac{8}{9}\right)^{4} < \frac{627}{1000} < \nu
  < \frac{157}{250} < \frac{16}{25}.
\end{equation}
Since $T^{4} = 16/\nu$, it follows that $25 < T^{4} < (9/4)^{4}$, and
taking positive fourth roots gives
\begin{equation} \label{eq:Tbounds}
  \sqrt{5} < T < \frac{9}{4} < \sqrt{7}.
\end{equation}
Consequently $(t/2)^{h} \leq (9/8)^{h}$ for $0 < t \leq T$ and $h \geq 0$,
so by Lemma~\ref{lem:covered} every covered square gives a reducing
normalized remainder at the weight $T$, and Lemma~\ref{lem:X}(d) applies
with $t = T$. Finally, $U\pi + \overline{U\pi} = 208102$ is twice the
integer $104051$ that appears in $T$. We will see in \S\ref{sec:corner}
that this integer is also the coefficient of the displacement of the
homoclinic point $\eta_{0}$ of the critical orbit from the class of $p$.

The first inequality in~\eqref{eq:nubounds} says that the test at the
trial weight $9/4$ does not cover the critical squares with the quotient
$q$, because $|N(\eta_{0} - q)| = \nu$ exceeds $(8/9)^{4}$. The difference
is small, $\nu - (8/9)^{4} = 0.0035\ldots$, and among the thirty-two
tested quotients~\eqref{eq:qbox} of a critical square, $q$ is the one for
which the test~\eqref{eq:rationaltest} fails by the smallest margin: for
each tested quotient of each critical square we computed the ratio of the
two sides of~\eqref{eq:rationaltest}, and the least ratio is between
$1.014$ and $1.035$, attained by $q$, while the next least, attained by
$5 + 3\sqrt{14}$, exceeds $1.05$. This is the sense in which the finite
calculation finds $T$: the critical squares are the squares through which
every obstructing orbit must pass, the quotient $q$ is the quotient that
comes closest to covering them, and $T$ is the weight at which the failing
test becomes an exact equality on the critical orbit.

\begin{remark} \label{rem:symmetry}
  The identities~\eqref{eq:symmetry} come from a symmetry of the critical
  orbit. The reflection $(x,y) \mapsto (8 - x,y)$ of \S\ref{sec:squares}
  reverses the direction of time and maps $\Lambda$ onto itself. A direct
  calculation with the matrix of multiplication by $\varepsilon^{-1}$
  shows that it carries an orbit with labels $(\ell_{k})$ to an orbit with
  labels $(\ell'_{k})$, where the coefficient point of $\ell'_{k}$ is
  $(112 + a,\,32 - b)$ when the coefficient point of
  $\varepsilon^{-1}\ell_{-k}$ is $(a,b)$. Applied to the
  labels~\eqref{eq:labels}, this rule gives $\ell'_{k} = \ell_{k}$ for
  every $k$. For example, $\varepsilon^{-1}\ell_{-2}$ has the coefficient
  point $(80,-16)$, and $(112 + 80,\,32 + 16) = (192,48)$ is the
  coefficient point of $\ell_{2}$. By Lemma~\ref{lem:coding}(a), the
  reflection therefore maps the critical orbit onto itself, and it sends
  $\eta_{k}$ to $\eta_{1 - k}$. Since the reflection sends a point with
  coordinates $(\rho_{s},\rho_{u})$ to a point with coordinates
  $(8 - \rho_{u},\,8 - \rho_{s})$, this gives the first two identities
  in~\eqref{eq:symmetry}. In the coordinates $(s,u)$, the translated point
  $\eta_{0} - q = (V,-U)$ and the translated point
  $\eta_{1} - q
  = (U,-V) = \varepsilon(V,-U)$ are reflections of each other in the line
  $s = -u$, and $V = \varepsilon U$ says that this reflection agrees with
  multiplication by $\varepsilon$ at the point $(V,-U)$. This is also why
  $\nu = \varepsilon U^{2}$ is $(U\pi)^{2}/2$, a square up to the factor
  $2$.
\end{remark}

\subsection{Division at the weight \texorpdfstring{$T$}{T}} \label{sec:endpoint}
We now prove statement (B) of \S\ref{sec:strategy}, that every class can
be divided at the weight $T$, and Proposition~\ref{prop:ordinary}.

\begin{proposition} \label{prop:division}
  Every $z \in \mathcal{Z}$ has a normalized remainder $w$ with
  $0 \leq \ve(w) \leq 4$ and $\phi_{T}(w) < 1$.
\end{proposition}
\begin{proof}
  Suppose, by way of contradiction, that some $z \in \mathcal{Z}$ has no
  normalized remainder $w$ with $\ve(w) \leq 4$ and $\phi_{T}(w) < 1$. The
  elements $-z$ and $\varepsilon^{k}z$, for $k \in \Z$, have the same
  property, because the maps $x \mapsto -x$ and
  $x \mapsto \varepsilon^{k}x$ permute $R$ and preserve absolute norms and
  valuations, so that they map the normalized remainders of $z$ onto the
  normalized remainders of $-z$ and of $\varepsilon^{k}z$ without changing
  their weighted norms. By Lemma~\ref{lem:periodic}(b), the class $z + R$
  contains an element $\xi$ with $\ve(\xi) \geq 5$, and subtracting an
  element of $\pi^{5}R$ we may take its coefficient point $\rho$ in $D$.
  By Lemma~\ref{lem:X}(d) with $t = T < 9/4$, the point $\rho + \Lambda$
  lies in $\mathcal{X}$, and it is periodic by
  Lemma~\ref{lem:periodic}(a).

  The orbit of $\rho + \Lambda$ is not contained in $\mathcal{F}$.
  Otherwise $\rho + \Lambda$ would be a point of $\mathcal{F}$, so $\xi$
  would lie in $\pi^{5}R$ or in $p + R$. The first is impossible because
  $z \notin R$, and the second because $\ve(\xi) \geq 5$ while every
  element of $p + R$ has valuation $-2$ by Lemma~\ref{lem:pclass}. By
  Lemma~\ref{lem:passage}, after replacing $z$ by $-z$ and $\xi$ by
  $8 + 4\sqrt{14} - \xi$, whose coefficient point $(8 - x,4 - y)$
  represents $-\rho + \Lambda$ and whose valuation is still at least $5$,
  we may assume that some point $\varepsilon^{k}\rho + \Lambda$ of the
  orbit has its representative $\rho_{k}$ in $D$ assigned to a critical
  square. Let $\xi_{k}$ be the element with coefficient point $\rho_{k}$.
  Then $\xi_{k} = \varepsilon^{k}\xi - \gamma_{k}$ with
  $\gamma_{k} \in \pi^{5}R$, so $\ve(\xi_{k}) \geq 5$ by
  Lemma~\ref{lem:periodic}(c), and $\xi_{k} \in \varepsilon^{k}z + R$. Let
  $q = 4 + 4\sqrt{14}$ and $w = \xi_{k} - q$. By~\eqref{eq:fixedval},
  $\ve(w) = \ve(q) = 4$, and by Lemma~\ref{lem:cornerbound},
  $|N(w)| = |N(\rho_{k} - q)| \leq \nu$. But $N(w) \in \Q$ and
  $\nu \notin \Q$, by~\eqref{eq:nu}, so $|N(w)| < \nu$. Since
  $(T/2)^{4} = 1/\nu$,
  \[
    \phi_{T}(w) = |N(w)|\left(\frac{T}{2}\right)^{4}
    = \frac{|N(w)|}{\nu} < 1.
  \]
  Thus $w$ is a normalized remainder of $\varepsilon^{k}z$ with
  $\ve(w) \leq 4$ and $\phi_{T}(w) < 1$, which contradicts the first
  paragraph. Hence every $z \in \mathcal{Z}$ has a normalized remainder
  $w$ with $\ve(w) \leq 4$ and $\phi_{T}(w) < 1$. Finally $\ve(w) \geq 0$,
  because the class $z + R$ has nonnegative valuation. This proves the
  proposition.
\end{proof}

\begin{proof}[Proof of Proposition~\ref{prop:ordinary}]
  Suppose, by way of contradiction, that $z \notin p + R$ has no
  normalized remainder of absolute norm less than $1$. Then $z \notin R$,
  because $0$ is a normalized remainder of every element of $R$. Let
  $\rho$ be the coefficient point of $z$, taken in $D$. By
  Lemma~\ref{lem:X}(c), $\rho + \Lambda$ lies in $\mathcal{X}$, and it is
  periodic by Lemma~\ref{lem:periodic}(a). Its orbit is not contained in
  $\mathcal{F}$, because the points of $\mathcal{F}$ are the coefficient
  points of the elements of $\pi^{5}R$ and of $p + R$ modulo $\Lambda$,
  and $z$ lies in neither class. The elements $-z$ and $\varepsilon^{k}z$
  also have no normalized remainder of absolute norm less than $1$. By
  Lemma~\ref{lem:passage}, after replacing $z$ by $-z$ we may assume that
  some point of the orbit has its representative $\rho_{k}$ assigned to a
  critical square. The element $\xi_{k}$ with coefficient point $\rho_{k}$
  lies in $\varepsilon^{k}z + R$, and by Lemma~\ref{lem:cornerbound} its
  normalized remainder $w = \xi_{k} - q$ with $q = 4 + 4\sqrt{14}$
  satisfies $|N(w)| \leq \nu < 1$. Since $\varepsilon^{-k}w$ is a
  normalized remainder of $z$ with the same absolute norm, this is a
  contradiction. This proves Proposition~\ref{prop:ordinary}.
\end{proof}

The bound~\eqref{eq:cornerbound} is sharp only at one point, the corner
$(\rho_{s},\rho_{u}) = (\bar{q} + V,\,q - U)$ of the
rectangle~\eqref{eq:cornerbound}. This point is $\eta_{0}$, and the
translated point $\eta_{0} - q$ has the coordinates $(V,-U)$. We call
$(V,-U)$ the \emph{critical corner}. The bound in
Proposition~\ref{prop:division} is strict for normalized remainders in
$K$, because their absolute norms are rational and $\nu$ is irrational.
This explains division at the weight $T$, although the bound on the
rectangle is not strict. To show that no greater weight lies in the window
we must approach the critical corner by normalized remainders of classes
of nonnegative valuation and exclude all other reducing normalized
remainders. A closed walk in the graph need not be the itinerary of a
periodic orbit, so the argument of this section does not by itself supply
those classes.
\section{Classes that obstruct division above \texorpdfstring{$T$}{T}} \label{sec:family}
In this section we prove statement (A) of \S\ref{sec:strategy}
(Proposition~\ref{prop:upper}). We construct periodic classes
$z_{m} + R$ of nonnegative valuation using sequences of labels from
the final graph of \S\ref{sec:weighted}. These sequences follow the labels
of the period-$2$ orbit of $(11/2,5/2) + \Lambda$ for about $m$ steps,
then those of the critical orbit through the critical squares, then
those of the period-$2$ orbit of $(5/2,5/2) + \Lambda$ for about $m$
steps, and they return using the labels of a second orbit of the final
graph. The periodic label sequences determine the elements $z_{m}$ by
the series of
Lemma~\ref{lem:coding}, and the orbit passes arbitrarily close to the
critical corner (Figure~\ref{fig:corner}). We choose the periods
arithmetically in \S\ref{sec:periods}. We then show that only eight
positions in the orbit can carry a small normalized remainder, that a reflection
pairs these positions, and that the quotients at the remaining four
positions lie in four finite boxes. Comparing the limits of the absolute
norms of the corresponding normalized remainders with $1$ excludes every reducing
normalized remainder at the weights above $T$.

\subsection{The critical corner} \label{sec:corner}
In \S\S\ref{sec:graph}--\ref{sec:weighted} we worked modulo $\pi^{5}R$, so
that we could choose representatives of valuation at least $5$. For such a
representative the valuation of a normalized remainder is the valuation of
its quotient, by~\eqref{eq:fixedval}. Here we work modulo $R$. We will
compute the valuations of the explicit normalized remainders directly, so
we no longer need the lattice $\pi^{5}R$. Modulo $R$ the class of $p$ is
fixed by $\varepsilon$, which makes the construction of the orbits
simpler. The lattice of $R$ contains the lattice $\pi^{5}R$, so congruence
modulo $\Lambda$ implies congruence modulo $R$, and an orbit in
$\R^{2}/\Lambda$ projects to an orbit modulo $R$. The two period-$2$
orbits of \S\ref{sec:etadef} consist of points of elements of $p + R$, so
both project to the fixed point of $p$, and the critical orbit, which is
heteroclinic in $\R^{2}/\Lambda$, projects to a homoclinic orbit of that
fixed point.

In this section we describe a point of $\R^{2}$ by its stable and unstable
coordinates, and we write it as $(s,u)$, as Ramsey
does~\cite[Section~3]{Ramsey2025}. Thus an element $x$ of $K$ corresponds
to $(\bar{x},x)$, which we call the point of $x$ in the coordinates
$(s,u)$. In this section congruence of points is always modulo $R$, and
the lattice of $R$ in these coordinates is
\[
  \{(\bar{r},r) \colon r \in R\} \subseteq \R^{2}.
\]
This is the lattice used for congruence modulo $R$ in this section, rather
than the lattice $\Lambda$ of \S\ref{sec:squares}. The notation
$(s,u) \equiv (s',u') \pmod{R}$ means that $(s - s',u - u') = (\bar{r},r)$
for some $r \in R$. Multiplication by $\varepsilon$ acts by
$\varepsilon \cdot (s,u) = (\varepsilon^{-1}s,\varepsilon u)$. We say that
points $(s_{k},u_{k})$ \emph{converge to $(s,u)$ modulo $R$} if there are
elements $r_{k} \in R$ with $s_{k} - s - \bar{r}_{k} \to 0$ and
$u_{k} - u - r_{k} \to 0$. Since $(\varepsilon^{k} - 1)p \in R$ by
Lemma~\ref{lem:pclass}, subtracting the point of $(\varepsilon^{k} - 1)p$
gives
\begin{equation} \label{eq:Ek}
  \varepsilon^{k} \cdot (\bar{p} + a,\,p + b) \equiv
  (\bar{p} + \varepsilon^{-k}a,\ p + \varepsilon^{k}b) \pmod{R}
\end{equation}
for all real numbers $a$ and $b$ and all $k \in \Z$. In particular the
point $(\bar{p},p)$ is fixed modulo $R$. The next lemma describes the
points whose orbits approach it.

\begin{lemma} \label{lem:nearp}
  Let $s,u \in \R$. The points $\varepsilon^{k} \cdot (s,u)$ converge to
  $(\bar{p},p)$ modulo $R$ as $k \to \infty$ if and only if $u \in p + R$.
  They converge to $(\bar{p},p)$ modulo $R$ as $k \to -\infty$ if and only
  if $s \in \bar{p} + R$.
\end{lemma}
\begin{proof}
  Write $s = \bar{p} + a$ and $u = p + b$. If $b \in R$, subtracting the
  point $(\varepsilon^{-k}\bar{b},\varepsilon^{k}b)$ of $\varepsilon^{k}b$
  from the right side of~\eqref{eq:Ek} gives the representative
  $(\bar{p} + \varepsilon^{-k}(a - \bar{b}),\ p)$, which tends to
  $(\bar{p},p)$ as $k \to \infty$.

  Conversely, suppose that the points converge as $k \to \infty$.
  By~\eqref{eq:Ek} there are elements $r_{k} \in R$ such that
  \[
    \alpha_{k} = \varepsilon^{k}b - r_{k} \to 0,\qquad
    \beta_{k} = \varepsilon^{-k}a - \bar{r}_{k} \to 0.
  \]
  The element $g_{k} = r_{k + 1} - \varepsilon r_{k}$ of $R$ satisfies
  $g_{k} = \varepsilon\alpha_{k} - \alpha_{k + 1}$ and
  $\bar{g}_{k} = \varepsilon^{-1}\beta_{k} - \beta_{k + 1}$, so
  $g_{k} \to 0$ and $\bar{g}_{k} \to 0$, and $N(g_{k}) \to 0$. A nonzero
  element of $R$ has norm of absolute value at least $1$. Hence there is
  $k_{0}$ such that $g_{k} = 0$ for $k \geq k_{0}$, and then
  $\alpha_{k} = \varepsilon^{k - k_{0}}\alpha_{k_{0}}$ for $k \geq k_{0}$.
  Since $\alpha_{k} \to 0$ and $\varepsilon^{k - k_{0}} \to \infty$, we
  have $\alpha_{k_{0}} = 0$, and so
  $b = \varepsilon^{-k_{0}}r_{k_{0}} \in R$.

  For $k \to -\infty$, consider the map $(s,u) \mapsto (u,s)$. It maps the
  point of $r \in R$ to the point of $\bar{r} \in R$, so it preserves
  congruence modulo $R$. It maps $\varepsilon^{-k} \cdot (s,u)$ to
  $\varepsilon^{k} \cdot (u,s)$, and it maps $(\bar{p},p)$ to
  $(p,\bar{p})$, which is congruent to $(\bar{p},p)$ modulo $R$ because
  $p - \bar{p} = \sqrt{14} \in R$. Hence the points
  $\varepsilon^{-k} \cdot (s,u)$ converge to $(\bar{p},p)$ modulo $R$ as
  $k \to \infty$ if and only if the points $\varepsilon^{k} \cdot (u,s)$
  do, that is, by the first part, if and only if
  $s \in p + R = \bar{p} + R$. This proves the lemma.
\end{proof}

A point that is not congruent to $(\bar{p},p)$ but whose orbit converges
to $(\bar{p},p)$ modulo $R$ in both directions is called a
\emph{homoclinic point} of $(\bar{p},p)$ in the theory of dynamical
systems. Its orbit leaves the neighborhood of $(\bar{p},p)$ and returns to
it. By Lemma~\ref{lem:nearp}, the homoclinic points are the points $(s,u)$
with $s \in \bar{p} + R$ and $u \in p + R$ that are not congruent to
$(\bar{p},p)$. Subtracting the point $(\bar{r},r)$ with
$\bar{r} = s - \bar{p}$ shows that every such point is congruent modulo
$R$ to a point $(\bar{p},p + \delta)$ with $0 \neq \delta \in R$, and
$\delta$ is unique, because the only point $(\bar{r},r)$ with
$\bar{r} = 0$ is $(0,0)$. We call $\delta$ the \emph{displacement} of the
homoclinic point. By~\eqref{eq:Ek}, for $k \geq 0$,
\[
  \varepsilon^{-k} \cdot (\bar{p},p + \delta)
  \equiv (\bar{p},\ p + \varepsilon^{-k}\delta),\qquad
  \varepsilon^{k} \cdot (\bar{p},p + \delta)
  \equiv (\bar{p} - \varepsilon^{-k}\bar{\delta},\ p) \pmod{R},
\]
and both points tend to $(\bar{p},p)$.

We now return to the critical orbit of \S\ref{sec:etadef} and describe it
modulo $R$. Let $q = 4 + 4\sqrt{14}$. By Lemma~\ref{lem:cornerbound}
and~\eqref{eq:cornerbound}, the translated point $\rho - q$ of the
representative $\rho$ of a periodic point of $\mathcal{X}$ in a critical
square satisfies $|N(\rho - q)| \leq UV = \nu$, with equality only when
$\rho = \eta_{0}$, where the translated point $\eta_{0} - q$ is $(V,-U)$.
At the negatives of the critical squares the corresponding translated
point is $(-V,U)$. Negation preserves the absolute norms and the
valuations of normalized remainders and carries the class $p + R$ to
itself, because $2p \in R$. So we consider only the point $(V,-U)$.

Both coordinates of $\eta_{0}$ come from the class of $p$. By
Lemma~\ref{lem:pclass}, for $r \in R$,
\[
  \varepsilon^{-1}(p + r) = p + (\varepsilon^{-1} - 1)p + \varepsilon^{-1}r
  \in p + R,
\]
and $(5 + 5\sqrt{14})/2 = p + 1 + 2\sqrt{14} \in p + R$, while the labels
$\ell_{0},\ldots,\ell_{3}$ lie in $R$. Hence the first formula
in~\eqref{eq:etacoords} shows that $(\eta_{0})_{u}$ lies in $p + R$.
Similarly, $\varepsilon(p + r) \in p + R$ for $r \in R$, so conjugating
gives $\varepsilon^{-1}(\bar{p} + \bar{r}) \in \bar{p} + R$, and
$(11 - 5\sqrt{14})/2 = \overline{p + 4 + 2\sqrt{14}} \in \bar{p} + R$.
Hence the second formula in~\eqref{eq:etacoords} shows that
$(\eta_{0})_{s}$ lies in $\bar{p} + R$. Explicitly,
by~\eqref{eq:etacoords} and $p = (3 + \sqrt{14})/2$,
\[
  (\eta_{0})_{s} = \bar{q} + V = \bar{p} + 13445 - 3595\sqrt{14},\qquad
  (\eta_{0})_{u} = q - U = p - 402759 + 107646\sqrt{14}.
\]
Thus the stable coordinate of $\eta_{0}$ lies in $\bar{p} + R$ and its
unstable coordinate lies in $p + R$, and $\eta_{0}$ is a homoclinic point
of $(\bar{p},p)$. Subtracting the point of $r = 13445 + 3595\sqrt{14}$
shows that it is congruent to $(\bar{p},p + \delta)$ with
\[
  \delta = (\eta_{0})_{u} - \overline{(\eta_{0})_{s}}
  = -416204 + 104051\sqrt{14} = -104051\bar{\pi}.
\]
We write $\delta = -n\bar{\pi}$ with $n = 104051$. By~\eqref{eq:Ek}, the
point $\eta_{k}$ is congruent to $(\bar{p},p + \varepsilon^{k}\delta)$
modulo $R$ for every $k \in \Z$, so modulo $R$ the critical orbit is the
homoclinic orbit through the point of displacement $\delta$.

This congruence determines the coordinates of $\eta_{k}$ that were not
computed in \S\ref{sec:etadef}. Let $k \leq -4$. By \S\ref{sec:etadef},
the stable coordinate of $\eta_{k}$ is that of $(11/2,5/2)$ if $k$ is odd
and that of $(13/2,7/2)$ if $k$ is even. These two points are the
coefficient points of $(11 + 5\sqrt{14})/2 = p + 4 + 2\sqrt{14}$ and
$(13 + 7\sqrt{14})/2 = p + 5 + 3\sqrt{14}$, which lie in $p + R$, so both
points are congruent to $(\bar{p},p)$ modulo $R$. Hence the difference
between $\eta_{k}$ and the point of the same parity is congruent to
$(0,\varepsilon^{k}\delta)$ modulo $R$, and its stable coordinate is $0$.
Since the only point $(\bar{r},r)$ with $\bar{r} = 0$ is $(0,0)$, the
difference is $(0,\varepsilon^{k}\delta)$. In the same way, for $k \geq 4$
the point $\eta_{k}$ has the unstable coordinate of $(5/2,5/2)$ or of
$(3/2,7/2)$, these points are congruent to $(\bar{p},p)$, and the
congruence $\eta_{k} \equiv (\bar{p} - \varepsilon^{-k}\bar{\delta},\,p)$,
which follows from~\eqref{eq:Ek}, gives the stable coordinate. Thus
\begin{equation} \label{eq:etatail}
  \begin{aligned}
    (\eta_{k})_{u}
                                   & = \frac{11 + 5\sqrt{14}}{2} + \varepsilon^{k}\delta
    \quad\text{($k \leq -5$ odd)}, &
    (\eta_{k})_{u}                 & = \frac{13 + 7\sqrt{14}}{2} + \varepsilon^{k}\delta
    \quad\text{($k \leq -4$ even)}, \\
    (\eta_{k})_{s}                 & = \frac{5 - 5\sqrt{14}}{2} - \varepsilon^{-k}\bar{\delta}
    \quad\text{($k \geq 4$ even)}, &
    (\eta_{k})_{s}                 & = \frac{3 - 7\sqrt{14}}{2} - \varepsilon^{-k}\bar{\delta}
    \quad\text{($k \geq 5$ odd)}.
  \end{aligned}
\end{equation}
The quotient $q_{*} = q - r = -13441 - 3591\sqrt{14}$ satisfies
\begin{equation} \label{eq:targetpair}
  \bar{p} - \bar{q}_{*} = V,\qquad p + \delta - q_{*} = -U.
\end{equation}
Equivalently, $q_{*} = p - \bar{V}$ and $\delta = -U - \bar{V}$. The
stable coordinate of $\eta_{0}$ determines the quotient, and the unstable
coordinate then determines the displacement $\delta$. Since
$\bar{V} = \varepsilon^{-1}\bar{U}$ and $\varepsilon^{-1}\pi = \bar{\pi}$,
\[
  n = \frac{U + \bar{V}}{\bar{\pi}} = \frac{\pi(U + \bar{V})}{2}
  = \frac{U\pi + \overline{U\pi}}{2},
\]
so the integer $n$ is the integer $104051$ that appears in the formula for
$T$ in \S\ref{sec:Tvalue}. Since $n$ is odd and $\ve(\bar{\pi}) = 1$, the
displacement $\delta = -n\bar{\pi}$ has valuation $1$. Also
$\bar{\delta} = -n\pi = \varepsilon\delta$, because
$\pi = \varepsilon\bar{\pi}$.

\begin{figure}[ht]
  \centering
  \begin{tikzpicture}
    \begin{scope}
      \filldraw[thin, fill = gray!35] (0,-1.6) rectangle (3.5,0);
      \draw[thin, -stealth] (-0.3,0) -- (4.6,0) node[right] {\small $s$};
      \draw[thin, -stealth] (0,-3.9) -- (0,0.8) node[above] {\small $u$};
      \node[above] at (3.5,0) {\small $V$}; \node[left] at (0,-1.6)
      {\small $-U$};
      % Schematic: the corner (V,-U) is drawn at (3.5,-1.6), so the
      % hyperbola |su| = nu through it is su = -3.5*1.6 = -5.6, and the
      % line u = -s/epsilon through the origin and the corner has the
      % slope -1.6/3.5. The normalized remainders w_m lie on this line because
      % conj(w_m) = -epsilon w_m, and they approach the corner from the
      % side of the origin.
      \draw[semithick] plot[domain = 1.5:4.5, samples = 80, smooth]
      (\x,{-5.6/\x}); \node[right] at (1.62,-3.62) {\small $|su| = \nu$};
      \draw[densely dotted] (0,0) -- (4.2,-1.92); \node[below left] at
      (2.25,-1.12) {\small $u = -s/\varepsilon$}; \fill (2.65,-1.2114)
      circle (1.1pt); \fill (3.10,-1.4171) circle (1.1pt); \fill
      (3.33,-1.5223) circle (1.1pt); \node[above right] at (2.60,-1.17)
      {\small $w_{m}$}; \filldraw[fill = white] (3.5,-1.6) circle (1.8pt);
      \node[below right] at (3.5,-1.6) {\small $(V,-U)$}; \node at
      (2.0,-4.45) {\small (a)};
    \end{scope}
    \begin{scope}[shift = {(7.6,0)}]
      % Schematic: the hyperbolas su = +-0.05 with 0.22361 = sqrt(0.05),
      % parametrized by (0.22361 e^t, 0.22361 e^-t). The parameter limits
      % 2.7926 = log(3.65/0.22361) and 1.6376 = log(1.15/0.22361) make the
      % curves end at the outermost drawn orbit points.
      \draw[gray!70] plot[variable = \t, domain = -2.7926:1.6376, samples
      = 150, smooth] ({-0.22361*exp(\t)},{-0.22361*exp(-\t)});
      \draw[gray!70] plot[variable = \t, domain = -1.6376:2.7926, samples
      = 150, smooth] ({0.22361*exp(\t)},{0.22361*exp(-\t)}); \draw[thin]
      (-2.0,0) -- (4.4,0); \node[above left] at (4.4,0) {\footnotesize
      stable line}; \draw[thin, -stealth] (-2.0,0) -- (-1.55,0);
      \draw[thin, -stealth] (4.4,0) -- (3.25,0); \draw[thin, -stealth]
      (0,0) -- (0,1.9) node[above] {\footnotesize unstable line};
      \draw[thin, -stealth] (0,0) -- (0,-4.0); \draw[thin, dashed,
      -stealth] (0.07,-3.62) to[bend right = 25] (3.62,0.07); \draw[thin,
      dashed, -stealth] (-0.09,1.12) to[bend right = 30] (-1.12,-0.07);
      \foreach \x/\y in {-1.15/-0.0435, -0.45/-0.1111, -0.3114/-0.1606,
      -0.2155/-0.232, -0.1491/-0.3352, -0.1032/-0.4844, -0.0714/-0.7,
      -0.0323/-1.55, -0.0227/-2.2, -0.0172/-2.9, -0.0137/-3.65,
      0.0435/1.15, 0.1111/0.45, 0.1606/0.3114, 0.232/0.2155,
      0.3352/0.1491, 0.4844/0.1032, 0.7/0.0714, 1.55/0.0323, 2.2/0.0227,
      2.9/0.0172, 3.65/0.0137} \fill (\x,\y) circle (1.1pt); \draw (0,0)
      circle (1.8pt); \draw (0,1.15) circle (2.6pt); \draw (0,-3.65)
      circle (2.6pt); \node[below right] at (0.05,-0.05) {\small
      $(\bar{p},p)$}; \node[right] at (0.14,1.15) {\small
      $(\bar{p},p+\bar{\pi})$}; \node[left] at (-0.14,-3.65) {\small
      $(\bar{p},p+\delta)$}; \node at (1.2,-4.45) {\small (b)};
    \end{scope}
  \end{tikzpicture}
  \caption{Schematic pictures, not to scale, with the stable coordinate
  $s$ horizontal and the unstable coordinate $u$ vertical. (a) The
  rectangle~\eqref{eq:cornerbound} in the coordinates of $\rho - q$, which
  has the origin and the critical corner $(V,-U)$ (open circle) as
  opposite corners, and the normalized remainders $w_{m}$ of~\eqref{eq:wm}
  for three values of $m$, which lie on the line $u = -s/\varepsilon$
  because $\bar{w}_{m} = -\varepsilon w_{m}$. (b) Representatives of the
  classes $\varepsilon^{k}z_{m} + R$ near the stable and unstable lines
  through $(\bar{p},p)$.} \label{fig:corner}
\end{figure}
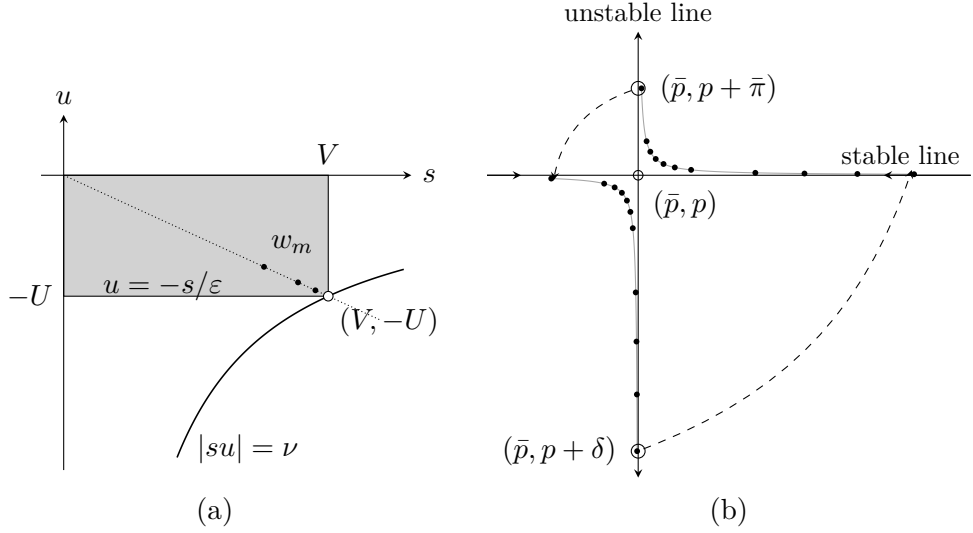

\subsection{Closing an orbit} \label{sec:closing}
The critical orbit leaves the period-$2$ orbit of $(11/2,5/2) + \Lambda$
and arrives at the period-$2$ orbit of $(5/2,5/2) + \Lambda$. A periodic
orbit that follows the critical orbit must return from the second of these
orbits to the first, and the final graph contains the simplest return.
Consider the sequence of labels that consists of the labels
$176 + 44\sqrt{14}$, $216 + 56\sqrt{14}$ of the period-$2$ orbit of
$(5/2,5/2) + \Lambda$ up to some point and of the labels
$216 + 56\sqrt{14}$, $288 + 76\sqrt{14}$ of the period-$2$ orbit of
$(11/2,5/2) + \Lambda$ from that point on, with $176 + 44\sqrt{14}$ as the
last label of the first kind and $216 + 56\sqrt{14}$ as the first label of
the second kind. By Lemma~\ref{lem:coding}(b), this sequence codes an
orbit. By Lemma~\ref{lem:coding}(a) and~\eqref{eq:twocycle}, the point of
this orbit at which the labels change has the unstable coordinate
$(11 + 5\sqrt{14})/2$ of $(11/2,5/2)$, because the labels that follow it
are those that follow $(11/2,5/2)$, and it has the stable coordinate
$(3 - 7\sqrt{14})/2$ of $(3/2,7/2)$, because the labels that precede it
are those that precede $(3/2,7/2)$. Its coefficient point is
\[
  \Bigl(\frac{7 - \sqrt{14}}{2},\ 3 + \frac{\sqrt{14}}{7}\Bigr)
  = (1.629\ldots,\,3.534\ldots),
\]
which lies inside one of the two rectangles of Figure~\ref{fig:clusters}.
This is the second orbit of that figure. Since
\[
  \frac{11 + 5\sqrt{14}}{2} - \overline{\Bigl(\frac{3 - 7\sqrt{14}}{2}\Bigr)}
  = 4 - \sqrt{14} = \bar{\pi},
\]
the second orbit is, modulo $R$, the homoclinic orbit of $(\bar{p},p)$
through the point of displacement $\bar{\pi}$. We call $\bar{\pi}$ the
\emph{auxiliary displacement} and $(\bar{p},p + \bar{\pi})$ the
\emph{auxiliary homoclinic point}. The displacement is small, which helps
to bound the normalized remainders near this point.

Now let $m \geq 7$ be odd. We define a sequence of labels
$(\ell_{m,k})_{k \in \Z}$ of period $2m$ by
\begin{equation} \label{eq:periodiclabels}
  \begin{aligned}
    \ell_{m,k}      & = \ell_{k - m - 1}\quad(3 \leq k \leq 2m + 1),\qquad
    \ell_{m,2} = 288 + 76\sqrt{14}, \\
    \ell_{m,k + 2m} & = \ell_{m,k}\quad(k \in \Z).
  \end{aligned}
\end{equation}
One period, read from $k = 3$, consists of $m - 5$ labels of the
period-$2$ orbit of $(11/2,5/2) + \Lambda$, beginning with
$216 + 56\sqrt{14}$, the seven labels~\eqref{eq:labels}, $m - 3$ labels of
the period-$2$ orbit of $(5/2,5/2) + \Lambda$, beginning with
$176 + 44\sqrt{14}$, and the label $288 + 76\sqrt{14}$. The last two
labels of the period are $216 + 56\sqrt{14}$, $288 + 76\sqrt{14}$, so the
sequence returns along the second orbit from the period-$2$ orbit of
$(5/2,5/2) + \Lambda$ to the period-$2$ orbit of $(11/2,5/2) + \Lambda$,
and the period closes up (Figure~\ref{fig:closing}). The passage through
the critical squares is at $k = m + 1$, where $\ell_{m,m + 1} = \ell_{0}$,
and the change of labels of the second orbit is at $k = 1$, where
$\ell_{m,0} = 176 + 44\sqrt{14}$ and $\ell_{m,1} = 216 + 56\sqrt{14}$. All
the labels lie in $\pi^{5}R$.

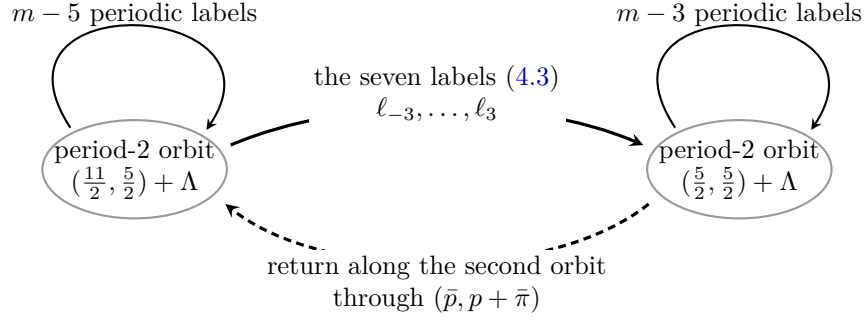
\begin{figure}[!htbp]
  \centering
  \begin{tikzpicture}[x=1cm,y=1cm,
  every node/.style={font=\footnotesize}]
    \draw[gray!80,thick] (-4,0) ellipse (1.22 and 0.65);
    \draw[gray!80,thick] (4,0) ellipse (1.22 and 0.65);
    \node[align=center] at (-4,0) {period-$2$
    orbit\\$(\frac{11}{2},\frac{5}{2})+\Lambda$};
    \node[align=center] at (4,0)
      {period-$2$ orbit\\$(\frac{5}{2},\frac{5}{2})+\Lambda$};
    \draw[-stealth,thick] (-4.85,0.55) .. controls (-6.1,2.45) and
    (-1.9,2.25) .. (-3.05,0.52); \node[align=center] at (-4.0,2.05) {$m-5$
    periodic labels}; \draw[-stealth,thick] (3.15,0.55) .. controls
    (1.9,2.45) and (6.1,2.25) .. (4.95,0.52); \node[align=center] at
    (4.0,2.05) {$m-3$ periodic labels}; \draw[-stealth,very thick]
    (-2.73,0.32) .. controls (-1.18,0.98) and (1.18,0.98) .. (2.73,0.32);
    \node[fill=white,inner sep=2pt,align=center] at (0,1.0) {the seven
    labels~\eqref{eq:labels}\\$\ell_{-3},\ldots,\ell_{3}$};
    \draw[-stealth,very thick,densely dashed] (2.8,-0.46) .. controls
    (1.5,-1.5) and (-1.5,-1.5) .. (-2.8,-0.46); \node[fill=white,inner
    sep=2pt,align=center] at (0,-1.52) {return along the second
    orbit\\through $(\bar{p},p + \bar{\pi})$};
  \end{tikzpicture}
  \caption{How the obstructing periodic classes are obtained by closing
  the critical orbit (\S\ref{sec:closing}). One period of the label
  sequence~\eqref{eq:periodiclabels} consists of $m - 5$ labels of the
  period-$2$ orbit of $(11/2,5/2) + \Lambda$, the seven
  labels~\eqref{eq:labels}, $m - 3$ labels of the period-$2$ orbit of
  $(5/2,5/2) + \Lambda$, and the label $288 + 76\sqrt{14}$. The dashed
  arrow represents the return along the second orbit of
  \S\ref{sec:closing}, where the labels of the orbit of
  $(5/2,5/2) + \Lambda$ give way to those of the orbit of
  $(11/2,5/2) + \Lambda$. The arrows indicate successive portions of the
  label sequence, not individual edges of the graph on the uncovered
  squares. As $m$ grows, the representatives $y_{m,m + 1 + d}$ at the
  passage tend to the points $\eta_{d}$ of the critical orbit
  (Lemma~\ref{lem:positions}(b)).} \label{fig:closing}
\end{figure}

For $k \in \Z$ let
\begin{equation} \label{eq:ymk}
  y_{m,k} = \frac{1}{\varepsilon^{2m} - 1}
  \sum_{j = 0}^{2m - 1}\varepsilon^{2m - 1 - j}\ell_{m,k + j}.
\end{equation}
This is an element of $K$, and $y_{m,k + 2m} = y_{m,k}$ because the labels
have period $2m$. Multiplying~\eqref{eq:ymk} by $\varepsilon$, separating
the term with $j = 0$, and using $\ell_{m,k + 2m} = \ell_{m,k}$ gives
\begin{equation} \label{eq:ymkrecursion}
  y_{m,k + 1} = \varepsilon y_{m,k} - \ell_{m,k}\qquad(k \in \Z).
\end{equation}
Taking conjugates in~\eqref{eq:ymkrecursion} gives
$\bar{y}_{m,k + 1} = \varepsilon^{-1}\bar{y}_{m,k} - \bar{\ell}_{m,k}$, so
the coefficient points of the $y_{m,k}$, whose stable and unstable
coordinates are $(\bar{y}_{m,k},y_{m,k})$, form a sequence of points
$(\rho_{k})_{k \in \Z}$ with
$\rho_{k + 1} = \varepsilon\rho_{k} - \lambda_{k}$, where $\lambda_{k}$ is
the coefficient point of $\ell_{m,k}$. This sequence is bounded, because
it is periodic. By Lemma~\ref{lem:coding}(a), it is the orbit coded by the
labels $\ell_{m,k}$, and its coordinates are given by the
series~\eqref{eq:series}.

\begin{lemma} \label{lem:reversal}
  For every odd $m \geq 7$ and every $k \in \Z$,
  \begin{equation} \label{eq:reversal}
    y_{m,3 - k} = 8 - \bar{y}_{m,k}.
  \end{equation}
  Consequently, for every $q \in R$, the normalized remainder
  $y_{m,3 - k} - (8 - \bar{q})$ at the position $3 - k$ is
  $-\overline{(y_{m,k} - q)}$, and it has the same absolute norm and the
  same valuation as the normalized remainder $y_{m,k} - q$ at the position
  $k$.
\end{lemma}
\begin{proof}
  By Remark~\ref{rem:symmetry}, the labels of the critical orbit satisfy
  \begin{equation} \label{eq:labelrule}
    8(\varepsilon - 1) + \varepsilon\bar{\ell}_{j} = \ell_{-j}\qquad(j \in \Z),
  \end{equation}
  because $8(\varepsilon - 1) = 112 + 32\sqrt{14}$ and
  $\varepsilon\bar{\ell}_{j} = \overline{\varepsilon^{-1}\ell_{j}}$ has
  the coefficient point $(a,-b)$ when $\varepsilon^{-1}\ell_{j}$ has the
  coefficient point $(a,b)$. We first show that
  \begin{equation} \label{eq:periodicrule}
    8(\varepsilon - 1) + \varepsilon\bar{\ell}_{m,2 - k} = \ell_{m,k}
    \qquad(k \in \Z).
  \end{equation}
  Both sides have period $2m$ in $k$, so let $0 \leq k < 2m$. If
  $3 \leq k \leq 2m - 1$, then $2 - k \equiv 2m + 2 - k \pmod{2m}$ with
  $3 \leq 2m + 2 - k \leq 2m - 1$, so $\ell_{m,2 - k} = \ell_{m + 1 - k}$
  by~\eqref{eq:periodiclabels}, and~\eqref{eq:labelrule} gives
  $8(\varepsilon - 1) + \varepsilon\bar{\ell}_{m,2 - k} = \ell_{k - m - 1}
  = \ell_{m,k}$. For $k = 0$, $1$, $2$ we have
  $\ell_{m,2} = 288 + 76\sqrt{14} = \ell_{-4}$, $\ell_{m,1} = \ell_{m}$,
  and $\ell_{m,0} = \ell_{m - 1}$, and~\eqref{eq:labelrule} gives
  $\ell_{4}$, $\ell_{-m}$, and $\ell_{1 - m}$. Since $4$, $m - 1$, and
  $1 - m$ are even and $m$ and $-m$ are odd, the definition of $\ell_{k}$
  for $|k| \geq 4$ in \S\ref{sec:etadef} gives
  $\ell_{4} = \ell_{m - 1} = \ell_{m,0}$,
  $\ell_{-m} = 216 + 56\sqrt{14} = \ell_{m} = \ell_{m,1}$, and
  $\ell_{1 - m} = 288 + 76\sqrt{14} = \ell_{m,2}$. This
  proves~\eqref{eq:periodicrule}.

  Now let $y'_{k} = 8 - \bar{y}_{m,3 - k}$ for $k \in \Z$. Conjugating
  $y_{m,3 - k} = \varepsilon y_{m,2 - k} - \ell_{m,2 - k}$, which
  is~\eqref{eq:ymkrecursion}, and multiplying by $\varepsilon$ gives
  $\bar{y}_{m,2 - k} = \varepsilon\bar{y}_{m,3 - k} + \varepsilon\bar{\ell}_{m,2 - k}$,
  and therefore
  \[
    y'_{k + 1} = 8 - \bar{y}_{m,2 - k}
    = \varepsilon(8 - \bar{y}_{m,3 - k})
    - \bigl(8(\varepsilon - 1) + \varepsilon\bar{\ell}_{m,2 - k}\bigr)
    = \varepsilon y'_{k} - \ell_{m,k}
  \]
  by~\eqref{eq:periodicrule}. Thus the sequences $(y'_{k})$ and
  $(y_{m,k})$ satisfy the same recursion~\eqref{eq:ymkrecursion}, and both
  have period $2m$. Their difference $D_{k} = y'_{k} - y_{m,k}$ satisfies
  $D_{k + 1} = \varepsilon D_{k}$, so
  $D_{k} = D_{k + 2m} = \varepsilon^{2m}D_{k}$, and since
  $\varepsilon^{2m} \neq 1$ we have $D_{k} = 0$. This
  proves~\eqref{eq:reversal}. Subtracting $8 - \bar{q}$ from both sides
  of~\eqref{eq:reversal} gives
  $y_{m,3 - k} - (8 - \bar{q}) = -\overline{(y_{m,k} - q)}$, and
  conjugation and negation preserve absolute norms and valuations. This
  proves the lemma.
\end{proof}

The reflection $(x,y) \mapsto (8 - x,y)$ of \S\ref{sec:squares}, which
sends the point with coordinates $(s,u)$ to the point with coordinates
$(8 - u,8 - s)$, therefore maps the orbit of the representatives onto
itself, exchanging the positions $k$ and $3 - k$, just as it maps the
critical orbit onto itself, exchanging $\eta_{k}$ and $\eta_{1 - k}$.

The next lemma gives the representatives in closed form. It compares the
orbit of $z_{m}$ with the critical orbit, which has the same labels for
$2m - 1$ consecutive steps.

\begin{lemma} \label{lem:closing}
  Let $m \geq 7$ be odd.
  \begin{enumerate}
    \item[(a)] For $3 \leq k \leq 2m + 2$,
      \begin{equation} \label{eq:representative}
        y_{m,k} = (\eta_{k - m - 1})_{u}
        + \frac{\bar{\pi}\varepsilon^{k - m - 1}(\varepsilon^{m} - n)}
        {\varepsilon^{2m} - 1},\qquad
        \bar{y}_{m,k} = (\eta_{k - m - 1})_{s}
        - \frac{\pi\varepsilon^{m + 1 - k}(\varepsilon^{m} - n)}
        {\varepsilon^{2m} - 1}.
      \end{equation}
    \item[(b)] Let $z_{m} = y_{m,0} - 1 - 2\sqrt{14}$. Then
      \begin{equation} \label{eq:family}
        z_{m} = p + \frac{\bar{\pi}\varepsilon^{m - 1}(\varepsilon^{m} - n)}
        {\varepsilon^{2m} - 1},
      \end{equation}
      and $y_{m,k} \in \varepsilon^{k}z_{m} + R$ for every $k \in \Z$. In
      particular the class $z_{m} + R$ has period dividing $2m$.
    \item[(c)] The normalized remainders $r_{m} = y_{m,1} - 4\sqrt{14}$
      and $w_{m} = y_{m,m + 1} - q$ of the classes $\varepsilon z_{m} + R$
      and $\varepsilon^{m + 1}z_{m} + R$, where $q = 4 + 4\sqrt{14}$, are
      \begin{equation} \label{eq:rm}
        r_{m} = \bar{p} + \frac{\bar{\pi}\varepsilon^{m}(\varepsilon^{m} - n)}
        {\varepsilon^{2m} - 1}
      \end{equation}
      and
      \begin{equation} \label{eq:wm}
        w_{m} = -U + \frac{\bar{\pi}(\varepsilon^{m} - n)}{\varepsilon^{2m} - 1},
        \qquad
        \bar{w}_{m} = V - \frac{\pi(\varepsilon^{m} - n)}{\varepsilon^{2m} - 1}.
      \end{equation}
  \end{enumerate}
\end{lemma}
\begin{proof}
  (a) For $3 \leq k \leq 2m + 1$ we have $\ell_{m,k} = \ell_{k - m - 1}$,
  so the representatives and the critical orbit satisfy the same
  recursions, $y_{m,k + 1} = \varepsilon y_{m,k} - \ell_{m,k}$ and
  $(\eta_{k - m})_{u} = \varepsilon(\eta_{k - m - 1})_{u} - \ell_{k - m - 1}$,
  for these $k$. Hence the differences
  $D_{k} = y_{m,k} - (\eta_{k - m - 1})_{u}$ satisfy
  $D_{k + 1} = \varepsilon D_{k}$ for $3 \leq k \leq 2m + 1$, and
  $D_{k} = \varepsilon^{k - 3}D_{3}$ for $3 \leq k \leq 2m + 2$. To find
  $D_{3}$ we close the orbit. Since $m + 1$ is even,
  $(\eta_{m + 1})_{u} = (5 + 5\sqrt{14})/2$ by \S\ref{sec:etadef}, so
  $y_{m,2} = y_{m,2m + 2} = (5 + 5\sqrt{14})/2 + \varepsilon^{2m - 1}D_{3}$.
  The recursion~\eqref{eq:ymkrecursion} at $k = 2$ and the direct
  calculation
  \[
    \varepsilon\,\frac{5 + 5\sqrt{14}}{2} - (288 + 76\sqrt{14})
    = \frac{11 + 5\sqrt{14}}{2} - (116 + 31\sqrt{14})
    = \frac{11 + 5\sqrt{14}}{2} - \varepsilon^{2}\bar{\pi}
  \]
  give
  $y_{m,3} = (11 + 5\sqrt{14})/2 - \varepsilon^{2}\bar{\pi} + \varepsilon^{2m}D_{3}$.
  Since $2 - m$ is odd and at most $-5$, \eqref{eq:etatail} gives
  $(\eta_{2 - m})_{u} = (11 + 5\sqrt{14})/2 + \varepsilon^{2 - m}\delta$.
  Subtracting, and using $\delta = -n\bar{\pi}$,
  \[
    D_{3} = y_{m,3} - (\eta_{2 - m})_{u}
    = -\varepsilon^{2}\bar{\pi} + n\bar{\pi}\varepsilon^{2 - m} + \varepsilon^{2m}D_{3},
    \qquad\text{so}\qquad
    D_{3} = \frac{\bar{\pi}\varepsilon^{2 - m}(\varepsilon^{m} - n)}{\varepsilon^{2m} - 1}.
  \]
  Multiplying by $\varepsilon^{k - 3}$ gives the first formula
  in~\eqref{eq:representative}.

  The stable coordinates are treated in the same way. Conjugating the
  recursions gives
  $\bar{y}_{m,k + 1} = \varepsilon^{-1}\bar{y}_{m,k} - \bar{\ell}_{m,k}$
  and
  $(\eta_{k - m})_{s} = \varepsilon^{-1}(\eta_{k - m - 1})_{s} - \bar{\ell}_{k - m - 1}$,
  so the differences $E_{k} = \bar{y}_{m,k} - (\eta_{k - m - 1})_{s}$
  satisfy $E_{k + 1} = \varepsilon^{-1}E_{k}$ for $3 \leq k \leq 2m + 1$,
  and $E_{k} = \varepsilon^{3 - k}E_{3}$ for $3 \leq k \leq 2m + 2$. By
  \eqref{eq:etatail}, $(\eta_{m + 1})_{s} = (5 - 5\sqrt{14})/2
  - \varepsilon^{-m - 1}\bar{\delta}$, so
  $\bar{y}_{m,2} = \bar{y}_{m,2m + 2} = (5 - 5\sqrt{14})/2
  - \varepsilon^{-m - 1}\bar{\delta} + \varepsilon^{1 - 2m}E_{3}$. The conjugate of the direct calculation above
  gives $\bar{y}_{m,3} = \varepsilon^{-1}\bar{y}_{m,2} - (288 - 76\sqrt{14})
  = (11 - 5\sqrt{14})/2 - \varepsilon^{-2}\pi - \varepsilon^{-m - 2}\bar{\delta}
  + \varepsilon^{-2m}E_{3}$, and
  $(\eta_{2 - m})_{s} = (11 - 5\sqrt{14})/2$ by \S\ref{sec:etadef},
  because $2 - m$ is odd. Subtracting, and using $\bar{\delta} = -n\pi$,
  \[
    E_{3} = -\varepsilon^{-2}\pi + n\pi\varepsilon^{-m - 2} + \varepsilon^{-2m}E_{3},
    \qquad\text{so}\qquad
    E_{3} = -\frac{\pi\varepsilon^{m - 2}(\varepsilon^{m} - n)}{\varepsilon^{2m} - 1}.
  \]
  Multiplying by $\varepsilon^{3 - k}$ gives the second formula
  in~\eqref{eq:representative}.

  (b) Apply~\eqref{eq:representative} with $k = 2m$. Since $m - 1$ is even
  and at least $6$, $(\eta_{m - 1})_{u} = (5 + 5\sqrt{14})/2$, and
  $(5 + 5\sqrt{14})/2 - 1 - 2\sqrt{14} = p$. This
  proves~\eqref{eq:family}. For $k \geq 0$, repeated use
  of~\eqref{eq:ymkrecursion} gives
  $y_{m,k} = \varepsilon^{k}y_{m,0} - (\varepsilon^{k - 1}\ell_{m,0} + \cdots + \ell_{m,k - 1})$,
  and the labels and $y_{m,0} - z_{m} = 1 + 2\sqrt{14}$ lie in $R$, so
  $y_{m,k} \in \varepsilon^{k}z_{m} + R$. For $k < 0$ the same follows
  from $y_{m,k} = y_{m,k + 2mj}$ with $j$ large, because
  $(\varepsilon^{2mj} - 1)z_{m} \in R$ by the case $k = 2mj$. In
  particular $(\varepsilon^{2m} - 1)z_{m} \in R$.

  (c) Apply~\eqref{eq:representative} with $k = 2m + 1$ and with
  $k = m + 1$. Since $m$ is odd, $(\eta_{m})_{u} = (3 + 7\sqrt{14})/2$,
  and $(3 + 7\sqrt{14})/2 - 4\sqrt{14} = \bar{p}$, which
  proves~\eqref{eq:rm}. At $k = m + 1$ we have $(\eta_{0})_{u} - q = -U$
  and $(\eta_{0})_{s} - \bar{q} = V$ by~\eqref{eq:UV}, which
  proves~\eqref{eq:wm}. This proves the lemma.
\end{proof}

The formula~\eqref{eq:representative} shows what the orbit of $z_{m}$
looks like. At the positions $k = m + 1 + d$ with $d$ fixed, the
representatives tend to the points $\eta_{d}$ of the critical orbit as
$m \to \infty$, because the two fractions tend to $0$. At the positions
$k = 1$ and $k = 2$, which are the positions $2m + 1$ and $2m + 2$, the
fractions tend to $\bar{\pi}$ and $\varepsilon\bar{\pi} = \pi$ in the
unstable coordinate and to $0$ in the stable coordinate, and the
representatives tend to the points
\begin{equation} \label{eq:secondorbitpoints}
  \Bigl(\frac{3 - 7\sqrt{14}}{2},\ \frac{11 + 5\sqrt{14}}{2}\Bigr)
  \qquad\text{and}\qquad
  \Bigl(\frac{5 - 5\sqrt{14}}{2},\ \frac{13 + 7\sqrt{14}}{2}\Bigr)
\end{equation}
of the second orbit, in the coordinates $(s,u)$. Here
$(13 + 7\sqrt{14})/2 = (5 + 5\sqrt{14})/2 + \pi$. The orbit therefore
stays near the class of $p$ except for two short excursions, one near the
critical corner and one near the homoclinic point
$(\bar{p},p + \bar{\pi})$. Figure~\ref{fig:corner}(b) shows the two
excursions modulo $R$. Between them the representatives have the form
$(\bar{p} + a,p + b)$ modulo $R$, and by~\eqref{eq:Ek} multiplication by
$\varepsilon$ preserves the product $ab$, so they lie on the hyperbolas
drawn in gray. The normalized remainders $w_{m}$ of~\eqref{eq:wm} satisfy
$\bar{w}_{m} = -\varepsilon w_{m}$, because $V = \varepsilon U$ and
$\pi = \varepsilon\bar{\pi}$, so in the coordinates $(s,u)$ they lie on
the line $u = -s/\varepsilon$ through the origin and the critical corner,
and they converge to the point $(V,-U)$ (Figure~\ref{fig:corner}(a)).

\subsection{Choosing the periods} \label{sec:periods}
We now choose $m$ so that $z_{m}$ has nonnegative valuation and the
excursion near the auxiliary homoclinic point $(\bar{p},p + \bar{\pi})$
supplies no normalized remainder that is reducing at weights near $T$.
Proposition~\ref{prop:limits} below shows that two normalized remainders
near the auxiliary homoclinic point matter, the normalized remainder
$r_{m}$ of~\eqref{eq:rm} and its reflection. Their absolute norms tend to
$(9 - 2\sqrt{14})/4 = 0.379\ldots$, and their valuations are even, by the
next lemma. Since $(T/2)^{h} = \nu^{-h/4}$, we have
$((9 - 2\sqrt{14})/4)(T/2)^{h} > 1$ if and only if
$((9 - 2\sqrt{14})/4)^{4} > \nu^{h}$. Exact comparison in $K$ gives
$((9 - 2\sqrt{14})/4)^{4} < \nu^{8}$ and
$((9 - 2\sqrt{14})/4)^{4} > \nu^{10}$. Since $\nu < 1$, an even valuation
$h \leq 8$ is therefore too small, and $h = 10$ suffices. We therefore
need an even valuation of at least $10$, and this leads to a congruence
modulo $2^{8} = 256$.

\begin{lemma} \label{lem:zm}
  Let $m \geq 7$ be odd, write $\varepsilon^{m} = A_{m} + B_{m}\sqrt{14}$,
  and let $r_{m}$ be the normalized remainder~\eqref{eq:rm}. Then
  \begin{equation} \label{eq:rmexplicit}
    r_{m} = \frac{\bar{\pi}(A_{m} - 7B_{m} - n)}{2B_{m}\sqrt{14}},\qquad
    \bar{r}_{m} = -\varepsilon r_{m},\qquad
    \ve(r_{m}) = 2\operatorname{ord}_{2}(A_{m} - 7B_{m} - n) - 6.
  \end{equation}
  Furthermore, $\ve(r_{m}) \geq 10$ if and only if
  \begin{equation} \label{eq:residue}
    A_{m} - 7B_{m} \equiv n \pmod{256},
  \end{equation}
  and this holds if and only if $m \equiv 33 \pmod{64}$. For such $m$,
  every representative satisfies $\ve(y_{m,k}) \geq 5$, and
  $\ve(z_{m}) = 0$.
\end{lemma}
\begin{proof}
  For the explicit formula, note that $\varepsilon^{2m} - 1 = \varepsilon^{m}(\varepsilon^{m} - \varepsilon^{-m})
  = 2B_{m}\sqrt{14}\,\varepsilon^{m}$, because
  $\varepsilon^{-m} = \bar{\varepsilon}^{m} = A_{m} - B_{m}\sqrt{14}$.
  Hence $2B_{m}\sqrt{14}\,r_{m}
  = (\varepsilon^{m} - \varepsilon^{-m})\bar{p} + \bar{\pi}(\varepsilon^{m} - n)$. A direct calculation gives
  \[
    2B_{m}\sqrt{14}\,\bar{p} = 3B_{m}\sqrt{14} - 14B_{m},\qquad
    \bar{\pi}\varepsilon^{m} = 4A_{m} - 14B_{m} + (4B_{m} - A_{m})\sqrt{14},
  \]
  and the sum of these two numbers is
  $4A_{m} - 28B_{m} + (7B_{m} - A_{m})\sqrt{14} = \bar{\pi}(A_{m} - 7B_{m})$.
  This proves the first formula in~\eqref{eq:rmexplicit}. Conjugating it
  gives $\bar{r}_{m} = -(\pi/\bar{\pi})r_{m} = -\varepsilon r_{m}$,
  because $\pi = \varepsilon\bar{\pi}$. The integer $A_{m} - 7B_{m} - n$
  is not zero, because $A_{m}^{2} = 14B_{m}^{2} + 1 < 49B_{m}^{2}$ gives
  $A_{m} < 7B_{m}$. By Lemma~\ref{lem:valuation}(a) and (b),
  \begin{align*}
    \ve(r_{m}) & = \ve(\bar{\pi}) + 2\operatorname{ord}_{2}(A_{m} - 7B_{m} - n)
    - \ve(2) - \ve(B_{m}) - \ve(\sqrt{14}) \\
               & = 2\operatorname{ord}_{2}(A_{m} - 7B_{m} - n) - 2 - 2\operatorname{ord}_{2}(B_{m}),
  \end{align*}
  and we will see below that $\operatorname{ord}_{2}(B_{m}) = 2$ for every
  odd $m$. This gives the formula for $\ve(r_{m})$. In particular
  $\ve(r_{m})$ is even, and $\ve(r_{m}) \geq 10$ if and only if
  $\operatorname{ord}_{2}(A_{m} - 7B_{m} - n) \geq 8$, which
  is~\eqref{eq:residue}.

  Now write $m = 2k + 1$, and let
  $\omega = \varepsilon^{2} - 1 = 448 + 120\sqrt{14}$. A direct
  calculation gives
  \[
    \omega^{3} \equiv 0,\qquad
    \varepsilon\omega \equiv 128 + 8\sqrt{14},\qquad
    \varepsilon\omega^{2} \equiv 128 \pmod{256R}.
  \]
  Therefore
  \[
    \varepsilon^{m} = \varepsilon(1 + \omega)^{k}
    \equiv \varepsilon + k\varepsilon\omega + \binom{k}{2}\varepsilon\omega^{2}
    \equiv 15 + 4\sqrt{14} + k(128 + 8\sqrt{14}) + 128\binom{k}{2}
    \pmod{256R}.
  \]
  Thus $B_{m} \equiv 4 + 8k \pmod{256}$, which shows that
  $\operatorname{ord}_{2}(B_{m}) = 2$, and
  \[
    A_{m} - 7B_{m} \equiv -13 + 72k + 64k(k - 1) = -13 + 8k + 64k^{2} \pmod{256}.
  \]
  Since $n = 104051 \equiv 115 \pmod{256}$, condition~\eqref{eq:residue}
  is equivalent to $8k + 64k^{2} \equiv 128 \pmod{256}$, that is, to
  \begin{equation} \label{eq:scongruence}
    k + 8k^{2} \equiv 16 \pmod{32}.
  \end{equation}
  The map $k \mapsto k + 8k^{2}$ is injective modulo $32$, because
  $(k_{1} + 8k_{1}^{2}) - (k_{2} + 8k_{2}^{2}) = (k_{1} - k_{2})(1 + 8(k_{1} + k_{2}))$
  and the second factor is odd. So it permutes the residue classes modulo
  $32$, and since $k = 16$ satisfies~\eqref{eq:scongruence}, the
  congruence holds if and only if $k \equiv 16 \pmod{32}$, that is, if and
  only if $m \equiv 33 \pmod{64}$.

  Finally, let $m \equiv 33 \pmod{64}$. By~\eqref{eq:valuation},
  $\ve(4\sqrt{14}) = 5$, and $\ve(r_{m}) \geq 10$, so
  $y_{m,1} = r_{m} + 4\sqrt{14}$ has valuation $5$ by
  Lemma~\ref{lem:valuation}(b). The labels $\ell_{m,k}$ lie in $\pi^{5}R$,
  so Lemma~\ref{lem:periodic}(c) and the recursion~\eqref{eq:ymkrecursion}
  give $\ve(y_{m,k}) \geq 5$ for all $k \geq 1$, and
  $y_{m,k} = \varepsilon^{-1}(y_{m,k + 1} + \ell_{m,k})$ gives the same
  for all $k \leq 0$, because multiplication by the unit
  $\varepsilon^{-1}$ preserves valuations. Since $\ve(1 + 2\sqrt{14}) = 0$
  by~\eqref{eq:valuation}, the element $z_{m} = y_{m,0} - 1 - 2\sqrt{14}$
  has valuation $0$ by Lemma~\ref{lem:valuation}(b). This proves the
  lemma.
\end{proof}

From now on we let $m$ run over the progression
\begin{equation} \label{eq:progression}
  m = 33 + 64j\qquad(j \geq 0).
\end{equation}
For $m = 33$ the valuation of $r_{m}$ is $12$, and for $m = 97$ it is
$10$.

\subsection{Small normalized remainders of \texorpdfstring{$z_{m}$}{z\_m}}\label{sec:smallrem}
We now reduce the possible small normalized remainders of $z_{m}$ to a finite list.
The reduction has three steps. We first multiply a normalized remainder by a power
of $\varepsilon$ so that its two conjugates have comparable size. We then
exclude the positions close to the class of $p$, using the gap between
$1$ and $5/4$ in Lemma~\ref{lem:pclass}. Along the two period-$2$ orbits
the distance from the class of $p$ grows by the factor $\varepsilon$ at
each step, so a bound at the end of each stretch of the orbit bounds the
whole stretch, and the number of positions excluded in this way grows
with $m$. At the remaining positions the bounds on the two conjugates
restrict the possible quotients to finite boxes. These boxes do not
depend on $m$ once $m$ is sufficiently large.

The next lemma is a basic device in the study of inhomogeneous minima, and
we will use it to reduce questions about all normalized remainders of a
periodic class to finitely many normalized remainders. We call a nonzero
element $w$ of $K$ \emph{balanced} if
$\varepsilon^{-1} \leq |w/\bar{w}| < \varepsilon$.

\begin{lemma} \label{lem:balance}
  \mbox{}
  \begin{enumerate}
    \item[(a)] For every nonzero $w \in K$ there is a unique integer $g$
      such that $\varepsilon^{g}w$ is balanced.
    \item[(b)] Every balanced element $w$ of $K$ with $|N(w)| < 1$
      satisfies
      \begin{equation} \label{eq:balance}
        |w|,|\bar{w}| < \sqrt{\varepsilon},
        \qquad |w| + |\bar{w}| < \varepsilon^{1/2} + \varepsilon^{-1/2} = \sqrt{32}.
      \end{equation}
    \item[(c)] Let $w$ be a normalized remainder of a class $z + R$ with
      $(\varepsilon^{P} - 1)z \in R$ for some $P \geq 1$, and let
      $g \in \Z$. Then $\varepsilon^{g}w$ is a normalized remainder of one
      of the classes $\varepsilon^{k}z + R$ with $0 \leq k < P$, and
      $N(\varepsilon^{g}w) = N(w)$ and $\ve(\varepsilon^{g}w) = \ve(w)$.
  \end{enumerate}
\end{lemma}
\begin{proof}
  (a) Since $\bar{\varepsilon} = \varepsilon^{-1}$, we have
  $|\varepsilon^{g}w/\overline{\varepsilon^{g}w}| = \varepsilon^{2g}|w/\bar{w}|$,
  and the intervals $[\varepsilon^{-1 - 2g},\varepsilon^{1 - 2g})$ with
  $g \in \Z$ partition the positive real numbers. Hence there is a unique
  integer $g$ for which $\varepsilon^{g}w$ is balanced.

  (b) Write $r = |w/\bar{w}|$, so that
  $\varepsilon^{-1} \leq r < \varepsilon$. Since
  $|w|\,|\bar{w}| = |N(w)|$, we have $|w|^{2} = |N(w)|\,r$ and
  $|\bar{w}|^{2} = |N(w)|/r$. Both products are less than $\varepsilon$,
  because $|N(w)| < 1$, $r < \varepsilon$, and $1/r \leq \varepsilon$.
  Also $|w| + |\bar{w}| = |N(w)|^{1/2}(r^{1/2} + r^{-1/2})
  < \varepsilon^{1/2} + \varepsilon^{-1/2}$, because $x + 1/x$ is at most
  $\varepsilon^{1/2} + \varepsilon^{-1/2}$ for
  $\varepsilon^{-1/2} \leq x \leq \varepsilon^{1/2}$. Since
  $(\varepsilon^{1/2} + \varepsilon^{-1/2})^{2} = \varepsilon + 2 + \varepsilon^{-1}
  = 32$, this proves~\eqref{eq:balance}.

  (c) Write $g = k + Pj$ with $0 \leq k < P$ and $j \in \Z$. Since
  $(\varepsilon^{P} - 1)z \in R$ and $\varepsilon$ is a unit,
  $(\varepsilon^{Pj} - 1)z \in R$, and so
  $\varepsilon^{g}w \in \varepsilon^{g}z + R = \varepsilon^{k}z + R$.
  Multiplication by a unit preserves norms and valuations. This proves the
  lemma.
\end{proof}

Ramsey~\cite[Lemma~1]{Ramsey2025} states a form of part~(b), which he
attributes to Berend and Moran~\cite[Lemma~4.2]{BerendMoran1992}. In the
notation of this paper, his lemma says that if a class has a normalized
remainder of absolute norm less than $c$, then some class in its orbit has
a normalized remainder $w$ with $|w|,|\bar{w}| < \sqrt{\varepsilon c}$ and
$|N(w)| < c$. Cavallar and Lemmermeyer use a similar
device~\cite[Proposition~4.3, p.~341]{CavallarLemmermeyer2000}, with a
bound on $|w|$ in place of the condition on $|w/\bar{w}|$.

By Lemma~\ref{lem:balance}, it is enough to consider balanced normalized
remainders of the classes $\varepsilon^{k}z_{m} + R$ with $0 \leq k < 2m$.
For $0 \leq k < 2m$ we call $k$ the \emph{position} of the representative
$y_{m,k}$ of this class, and we speak of a normalized remainder
$y_{m,k} - q$ at the position $k$. Two positions may represent the same
class, which does no harm, because the lemmas below are statements about
the representatives $y_{m,k}$. The next lemma shows that no balanced
normalized remainder of absolute norm less than $1$ occurs near the class
of $p$.

\begin{lemma} \label{lem:gap}
  Let $y \in K$ satisfy $|y - p| < 1/23$ and $|\bar{y} - \bar{p}| < 1/23$.
  Then the class $y + R$ contains no balanced element of absolute norm
  less than $1$.
\end{lemma}
\begin{proof}
  Suppose, by way of contradiction, that $w = y - q$ is such an element,
  with $q \in R$. Let $a = y - p$ and $b = \bar{y} - \bar{p}$. Then
  $p - q = w - a$ and $\bar{p} - \bar{q} = \bar{w} - b$, so
  by~\eqref{eq:balance},
  \[
    |N(p - q)| \leq |w\bar{w}| + \frac{|w| + |\bar{w}|}{23} + \frac{1}{23^{2}}
    < 1 + \frac{\sqrt{32}}{23} + \frac{1}{23^{2}}.
  \]
  The right side is less than $5/4$, because
  $\sqrt{32}/23 + 1/23^{2} = 0.2478\ldots < 1/4$. This contradicts
  Lemma~\ref{lem:pclass}, and proves the lemma.
\end{proof}

The radius comes from the gap $5/4 - 1$ in Lemma~\ref{lem:pclass}. For a
radius $1/b$ the proof gives $|N(p - q)| < 1 + \sqrt{32}/b + 1/b^{2}$, and
$23$ is the least positive integer $b$ for which this bound is at most
$5/4$.

We apply Lemma~\ref{lem:gap} to representatives that are close to a point
$c$ of $p + R$ other than $p$. If $c \in p + R$ and $y \in K$ satisfy
$|y - c| < 1/23$ and $|\bar{y} - \bar{c}| < 1/23$, then the element
$y - (c - p)$ of the class $y + R$ satisfies the hypothesis of
Lemma~\ref{lem:gap}, so $y + R$ contains no balanced element of absolute
norm less than $1$.

\begin{lemma} \label{lem:positions}
  Let $m \geq 7$ be odd.
  \begin{enumerate}
    \item[(a)] Let $0 \leq k < 2m$ with
      $k \notin \{1,\,2,\,m - 2,\,m,\,m + 1,\,m + 2,\,m + 3,\,m + 5\}$.
      Then the class $\varepsilon^{k}z_{m} + R$ contains no balanced
      element of absolute norm less than $1$.
    \item[(b)] As $m \to \infty$ through odd integers, the points
      $(\bar{y}_{m,m + 1 + d},y_{m,m + 1 + d})$ converge to $\eta_{d}$,
      written in the coordinates $(s,u)$ as
      $((\eta_{d})_{s},(\eta_{d})_{u})$, for every fixed $d \in \Z$, and
      the points $(\bar{y}_{m,1},y_{m,1})$ and $(\bar{y}_{m,2},y_{m,2})$
      converge to the two points~\eqref{eq:secondorbitpoints}. For each of
      the eight positions $k$ listed in (a), let $(s,u)$ be the limit of
      $(\bar{y}_{m,k},y_{m,k})$. Then, for all sufficiently large $m$,
      every balanced element of absolute norm less than $1$ in the class
      $\varepsilon^{k}z_{m} + R$ has the form $y_{m,k} - q$ with $q \in R$
      and
      \begin{equation} \label{eq:box}
        |u - q| \leq 6,\qquad |s - \bar{q}| \leq 6.
      \end{equation}
  \end{enumerate}
\end{lemma}
\begin{proof}
  Let
  \[
    \theta_{m} = \frac{\varepsilon^{m} - n}{\varepsilon^{2m} - 1}.
  \]
  Since $\varepsilon > 29$ and $29^{4} = 707281 > n$, we have
  $\varepsilon^{m} > n$, and $\varepsilon^{2m} - 1 > \varepsilon^{2m}/2$,
  so $0 < \theta_{m} < 2\varepsilon^{-m} \leq 2\varepsilon^{-7}$. In
  particular $0 < \theta_{m} < 1$. For $3 \leq k \leq 2m + 2$ and
  $j = k - m - 1$, the formula~\eqref{eq:representative} says that
  \begin{equation} \label{eq:deviation}
    y_{m,k} = (\eta_{j})_{u} + \bar{\pi}\varepsilon^{j}\theta_{m},\qquad
    \bar{y}_{m,k} = (\eta_{j})_{s} - \pi\varepsilon^{-j}\theta_{m}.
  \end{equation}
  Every position $0 \leq k < 2m$ is congruent modulo $2m$ to exactly one
  $k$ in this range, and $y_{m,k}$ depends only on $k$ modulo $2m$.

  (a) Suppose first that $3 \leq k \leq m - 3$, so that
  $2 - m \leq j \leq -4$. By~\eqref{eq:etatail}, $\eta_{j}$ is the point
  $(\bar{c},c)$ of an element $c \in p + R$, namely
  $c = (11 + 5\sqrt{14})/2$ or $c = (13 + 7\sqrt{14})/2$, plus the point
  $(0,\varepsilon^{j}\delta)$. Hence, by~\eqref{eq:deviation} and
  $\delta = -n\bar{\pi}$,
  \[
    y_{m,k} - c = \varepsilon^{j}\bar{\pi}(\theta_{m} - n),\qquad
    \bar{y}_{m,k} - \bar{c} = -\pi\varepsilon^{-j}\theta_{m}.
  \]
  Since $0 < \theta_{m} < 1 < n$ and $j \leq -4$, the first difference has
  absolute value less than $n\bar{\pi}\varepsilon^{-4}$. Since
  $j \geq 2 - m$ and $\varepsilon^{m - 2}\theta_{m} < \varepsilon^{-2}$,
  which is the inequality
  $\varepsilon^{2m} - n\varepsilon^{m} < \varepsilon^{2m} - 1$, the second
  difference has absolute value less than $\pi\varepsilon^{-2}$. Exact
  comparison in $K$ gives
  \begin{equation} \label{eq:tail}
    \frac{n\bar{\pi}}{\varepsilon^{4}} < \frac{1}{23},\qquad
    \frac{\pi}{\varepsilon^{2}} < \frac{1}{23},\qquad
    \frac{n\pi}{\varepsilon^{5}} < \frac{1}{23},\qquad
    \frac{\bar{\pi}}{\varepsilon} < \frac{1}{23}.
  \end{equation}
  By the first two inequalities and the remark before the lemma, the class
  $\varepsilon^{k}z_{m} + R$ contains no balanced element of absolute norm
  less than $1$.

  Next suppose that $m + 6 \leq k \leq 2m$, so that $5 \leq j \leq m - 1$.
  By~\eqref{eq:etatail}, $\eta_{j}$ is the point $(\bar{c},c)$ of
  $c = (5 + 5\sqrt{14})/2$ or $c = (3 + 7\sqrt{14})/2$, both of which lie
  in $p + R$, plus the point $(-\varepsilon^{-j}\bar{\delta},0)$. Hence,
  by~\eqref{eq:deviation} and $\bar{\delta} = -n\pi$,
  \[
    y_{m,k} - c = \bar{\pi}\varepsilon^{j}\theta_{m},\qquad
    \bar{y}_{m,k} - \bar{c} = \varepsilon^{-j}\pi(n - \theta_{m}).
  \]
  Since $j \leq m - 1$ and
  $\varepsilon^{m - 1}\theta_{m} < \varepsilon^{-1}$, by the same
  inequality as before, the first difference is less than
  $\bar{\pi}\varepsilon^{-1}$, and since $j \geq 5$ the second is less
  than $n\pi\varepsilon^{-5}$. By the last two inequalities
  in~\eqref{eq:tail}, the class contains no balanced element of absolute
  norm less than $1$.

  It remains to treat $k = m - 1$ and $k = m + 4$, that is, $j = -2$ and
  $j = 3$. Let $c = (15 + 7\sqrt{14})/2 = p + 6 + 3\sqrt{14}$ if $j = -2$
  and $c = (1 + 7\sqrt{14})/2 = p - 1 + 3\sqrt{14}$ if $j = 3$. Exact
  comparison in $K$ gives
  \begin{equation} \label{eq:twomore}
    |(\eta_{j})_{u} - c| < \frac{1}{23} - \frac{1}{\varepsilon^{3}},\qquad
    |(\eta_{j})_{s} - \bar{c}| < \frac{1}{23} - \frac{1}{\varepsilon^{3}}
    \qquad(j \in \{-2,3\}).
  \end{equation}
  By~\eqref{eq:deviation}, the point $(\bar{y}_{m,k},y_{m,k})$ differs
  from $\eta_{j}$ by less than
  $\pi\varepsilon^{3}\theta_{m} < 2\pi\varepsilon^{-4}$ in the stable
  coordinate and by less than
  $\bar{\pi}\varepsilon^{3}\theta_{m} < 2\bar{\pi}\varepsilon^{-4}$ in the
  unstable coordinate, and both bounds are less than $\varepsilon^{-3}$,
  because $2\bar{\pi} < 2\pi < \varepsilon$. Hence $|y_{m,k} - c| < 1/23$
  and $|\bar{y}_{m,k} - \bar{c}| < 1/23$, and the class contains no
  balanced element of absolute norm less than $1$.

  The positions treated so far are $3 \leq k \leq m - 3$, $k = m - 1$,
  $k = m + 4$, and $m + 6 \leq k \leq 2m$, where $k = 2m$ is the position
  $0$. The remaining positions are the eight positions listed in the
  statement. This proves (a).

  (b) For fixed $d$ and $k = m + 1 + d$ we have $j = d$
  in~\eqref{eq:deviation}, and $\varepsilon^{d}\theta_{m}$ and
  $\varepsilon^{-d}\theta_{m}$ tend to $0$, which proves the first
  assertion. For $k = 1$ and $k = 2$ we use $k = 2m + 1$ and $k = 2m + 2$,
  so that $j = m$ and $j = m + 1$. Since $m$ is odd, \S\ref{sec:etadef}
  and~\eqref{eq:etatail} give $(\eta_{m})_{u} = (3 + 7\sqrt{14})/2$,
  $(\eta_{m})_{s} = (3 - 7\sqrt{14})/2 - \varepsilon^{-m}\bar{\delta}$,
  $(\eta_{m + 1})_{u} = (5 + 5\sqrt{14})/2$, and
  $(\eta_{m + 1})_{s} = (5 - 5\sqrt{14})/2 - \varepsilon^{-m - 1}\bar{\delta}$.
  Also $\varepsilon^{m}\theta_{m} \to 1$, so the fractions
  in~\eqref{eq:deviation} tend to $\bar{\pi}$ and to
  $\varepsilon\bar{\pi} = \pi$ for the unstable coordinates and to $0$ for
  the stable coordinates. Since
  $(3 + 7\sqrt{14})/2 + \bar{\pi} = (11 + 5\sqrt{14})/2$ and
  $(5 + 5\sqrt{14})/2 + \pi = (13 + 7\sqrt{14})/2$, the limits are the
  points~\eqref{eq:secondorbitpoints}.

  Now let $w$ be a balanced element of absolute norm less than $1$ in
  $\varepsilon^{k}z_{m} + R = y_{m,k} + R$, and write $w = y_{m,k} - q$
  with $q \in R$. By~\eqref{eq:balance}, $|w| < \sqrt{\varepsilon}$ and
  $|\bar{w}| < \sqrt{\varepsilon}$. The radius $6$ in~\eqref{eq:box} is
  the least integer greater than $\sqrt{\varepsilon}$, since
  $\sqrt{14} < 4$ gives $\varepsilon = 15 + 4\sqrt{14} < 31 < 36$, and
  $\sqrt{14} > 3$ gives $\varepsilon > 27 > 25$. Since
  \[
    |u - q| \leq |w| + |y_{m,k} - u|,\qquad
    |s - \bar{q}| \leq |\bar{w}| + |\bar{y}_{m,k} - s|,
  \]
  and the last terms tend to $0$, the margin between $\sqrt{\varepsilon}$
  and $6$ gives~\eqref{eq:box} for all sufficiently large $m$. There are
  eight positions, so one bound on $m$ suffices for all of them. This
  proves the lemma.
\end{proof}

By Lemma~\ref{lem:reversal}, the eight positions of
Lemma~\ref{lem:positions} form four pairs $\{k,3 - k\}$, namely
\[
  \{1,2\},\qquad \{m - 2,m + 5\},\qquad \{m,m + 3\},\qquad \{m + 1,m + 2\},
\]
where positions are read modulo $2m$. The map $q \mapsto 8 - \bar{q}$
sends the box~\eqref{eq:box} at the position $k$ onto the
box~\eqref{eq:box} at the position $3 - k$, because by~\eqref{eq:reversal}
the limit $(s,u)$ at the position $k$ is sent by the reflection to the
limit $(8 - u,8 - s)$ at the position $3 - k$, and
$|(8 - s) - (8 - \bar{q})| = |s - \bar{q}|$ and
$|(8 - u) - \overline{(8 - \bar{q})}| = |u - q|$. Since the normalized
remainder at the position $3 - k$ with the quotient $8 - \bar{q}$ has the
same absolute norm and valuation as the normalized remainder at the
position $k$ with the quotient $q$, it will suffice to examine the four
positions
\begin{equation} \label{eq:fourpositions}
  k = m + 1 + d\quad(d \in \{-3,-1,0\}),\qquad k = 1.
\end{equation}
The first three are the positions near the critical corner, and the last
is the position near the auxiliary homoclinic point. We say that the
position $k = m + 1 + d$ belongs to the displacement $\gamma = \delta$ and
the index $d$, and that the position $k = 1$ belongs to
$\gamma = \bar{\pi}$ and $d = 0$.

The boxes~\eqref{eq:box} are finite. To enumerate them, write
$u = A + B\sqrt{14}$ and $s = A' + B'\sqrt{14}$ with rational numbers $A$,
$B$, $A'$, $B'$, which is possible because the limit points have
coordinates in $K$, and $q = a + b\sqrt{14}$. Adding and subtracting the
two inequalities in~\eqref{eq:box} gives
\[
  \left|a - \frac{A + A' + (B + B')\sqrt{14}}{2}\right| \leq 6,
  \qquad
  \left|b - \frac{B - B'}{2} - \frac{(A - A')\sqrt{14}}{28}\right|
  \leq \frac{6}{\sqrt{14}} < 2.
\]
We therefore test~\eqref{eq:box} for the integers $a$ and $b$ in these
bounded ranges.

\subsection{The limit} \label{sec:limit}
For a position $k$ in~\eqref{eq:fourpositions}, with the displacement
$\gamma$ and the index $d$ to which it belongs, and a quotient $q$ in the
box~\eqref{eq:box}, Lemma~\ref{lem:positions}(b) shows that the stable and
unstable coordinates of the normalized remainder $y_{m,k} - q$ tend to
$s - \bar{q}$ and $u - q$. Let
\[
  \lambda_{\gamma,d,q} = |(s - \bar{q})(u - q)|
\]
be the absolute value of their product. The two coordinates need not be
conjugate, so this need not be an absolute norm, but it is the limit of
the absolute norms $|N(y_{m,k} - q)|$. For $\gamma = \delta$ the limit
point is $(s,u) = \eta_{d}$, and for $\gamma = \bar{\pi}$ it is the first
point in~\eqref{eq:secondorbitpoints}.

The valuations of the normalized remainders are easy to find. For $m$
in~\eqref{eq:progression}, every representative $y_{m,k}$ has valuation at
least $5$ by Lemma~\ref{lem:zm}, so by~\eqref{eq:fixedval} every
normalized remainder $y_{m,k} - q$ with $\ve(q) \leq 4$ has valuation
$\ve(q)$, for every $m$ in~\eqref{eq:progression}. The finite calculation
gives the following proposition.

\begin{proposition} \label{prop:limits}
  Each of the four boxes~\eqref{eq:box} at the
  positions~\eqref{eq:fourpositions} contains $17$ quotients. No number
  $\lambda_{\gamma,d,q}$ equals $1$, and exactly six of them are less than
  $1$. These six pairs are listed in the following table. For each row
  with $\gamma = \delta$, the entry $h$ is the valuation of the
  corresponding normalized remainder $y_{m,k} - q$ for every $m$
  in~\eqref{eq:progression}, and for the row with $\gamma = \bar{\pi}$ it
  is a lower bound for that valuation.
  \begin{center}
    \begin{tabular}{crlcc}
      \toprule
      $\gamma$    & $d$  & $q$              & $4\lambda_{\gamma,d,q}$               & $h$ \\
      \midrule
      $\delta$    & $-3$ & $4 + 2\sqrt{14}$ & $146137132281 - 39056791462\sqrt{14}$ & $3$ \\
      $\delta$    & $-3$ & $6 + 2\sqrt{14}$ & $59430185377 - 15883385150\sqrt{14}$  & $2$ \\
      $\delta$    & $-1$ & $4 + 3\sqrt{14}$ & $43287288513 - 11569014486\sqrt{14}$  & $1$ \\
      $\delta$    & $-1$ & $6 + 3\sqrt{14}$ & $43190736361 - 11543209838\sqrt{14}$  & $1$ \\
      $\delta$    & $0$  & $4 + 4\sqrt{14}$ & $43305976025 - 11574008934\sqrt{14}$  & $4$ \\
      $\bar{\pi}$ & $0$  & $4\sqrt{14}$     & $9 - 2\sqrt{14}$                      & $10$ \\
      \bottomrule
    \end{tabular}
  \end{center}
  Every listed value $\lambda = \lambda_{\gamma,d,q}$ satisfies
  \begin{equation} \label{eq:limitineq}
    \lambda^{4} \geq \nu^{h},
  \end{equation}
  with equality exactly in the row with $\gamma = \delta$ and $d = 0$, for
  which $\lambda = \nu$.
\end{proposition}
\begin{proof}
  The boxes~\eqref{eq:box} are finite, and the enumeration described at
  the end of \S\ref{sec:smallrem} lists their elements. This is a finite
  calculation, which we carried out on a computer. It enumerates every
  quotient in the four boxes, compares every $\lambda_{\gamma,d,q}$ with
  $1$ by exact comparison in $K$, lists the pairs with
  $\lambda_{\gamma,d,q} < 1$, and verifies~\eqref{eq:limitineq}.

  The five values of $h$ for $\gamma = \delta$ are the valuations $3$,
  $2$, $1$, $1$, and $4$ of the quotients in the table,
  by~\eqref{eq:valuation}, and by the remark before the proposition they
  are the valuations of the normalized remainders for every $m$
  in~\eqref{eq:progression}. The normalized remainder in the row with
  $\gamma = \bar{\pi}$ is $y_{m,1} - 4\sqrt{14} = r_{m}$, whose valuation
  is at least $10$ for every $m$ in~\eqref{eq:progression} by
  Lemma~\ref{lem:zm}, and its limit is
  \[
    \lambda_{\bar{\pi},0,4\sqrt{14}}
    = \left|\Bigl(\frac{3 - 7\sqrt{14}}{2} + 4\sqrt{14}\Bigr)
    \Bigl(\frac{11 + 5\sqrt{14}}{2} - 4\sqrt{14}\Bigr)\right|
    = |p(\bar{p} + \bar{\pi})| = \frac{p^{2}}{\varepsilon}
    = \frac{9 - 2\sqrt{14}}{4},
  \]
  by the identity $\varepsilon(\bar{p} + \bar{\pi}) = -p$, which is a
  direct calculation. The equality $\lambda_{\delta,0,q} = \nu$ for
  $q = 4 + 4\sqrt{14}$ follows from $(\eta_{0})_{s} - \bar{q} = V$ and
  $(\eta_{0})_{u} - q = -U$, which is~\eqref{eq:UV}. This proves the
  proposition.
\end{proof}

By Lemma~\ref{lem:reversal}, the six pairs of the table have six partners
at the positions $m + 5$, $m + 3$, and $m + 2$ near the critical corner
and at the position $2$, with the quotients $8 - \bar{q}$, which are
\[
  4 + 2\sqrt{14},\quad 2 + 2\sqrt{14},\quad 4 + 3\sqrt{14},\quad
  2 + 3\sqrt{14},\quad 4 + 4\sqrt{14},\quad 8 + 4\sqrt{14},
\]
in the order of the rows. The partner normalized remainders have the same
absolute norms and valuations. Since $(T/2)^{h} = \nu^{-h/4}$, the
inequality~\eqref{eq:limitineq} says that $\lambda(T/2)^{h} \geq 1$. The
row with $\gamma = \delta$ and $d = 0$ is the only row with equality. Its
normalized remainder is the normalized remainder $w_{m} = y_{m,m + 1} - q$
of~\eqref{eq:wm}, whose point converges to the corner $(V,-U)$ of
Figure~\ref{fig:corner}(a), and its partner is the normalized remainder
$y_{m,m + 2} - q = -\bar{w}_{m}$ at the position $m + 2$, with the same
quotient $q = 8 - \bar{q}$.

We can now prove statement (A) of \S\ref{sec:strategy}.

\begin{proposition} \label{prop:upper}
  For every $t > T$ there is an $m_{0}$ such that, for every
  $m \geq m_{0}$ in~\eqref{eq:progression}, $z_{m} \in \mathcal{Z}$ and no
  normalized remainder of $z_{m}$ is reducing at weight $t$.
\end{proposition}
\begin{proof}
  Let $t > T$. By Lemma~\ref{lem:zm}, $\ve(z_{m}) = 0$ for every $m$
  in~\eqref{eq:progression}, so the class $z_{m} + R$ has nonnegative
  valuation, and so does every class $\varepsilon^{k}z_{m} + R$, because
  $\ve(\varepsilon) = 0$. We first show that $z_{m} \notin R$ for all
  sufficiently large $m$. By Lemma~\ref{lem:closing}(c), the normalized
  remainder $r_{m}$ lies in $\varepsilon z_{m} + R$, and by the proof of
  Proposition~\ref{prop:limits}, $|N(r_{m})|$ tends to
  $\lambda = (9 - 2\sqrt{14})/4$, which satisfies $0 < \lambda < 1$. So
  $0 < |N(r_{m})| < 1$ for all sufficiently large $m$. A nonzero element
  of $R$ has absolute norm at least $1$, so $r_{m} \notin R$, and
  therefore $\varepsilon z_{m} \notin R$ and $z_{m} \notin R$. Thus
  $z_{m} \in \mathcal{Z}$ for all sufficiently large $m$.

  We next show that there is an $m_{0}$ such that, for every
  $m \geq m_{0}$ in~\eqref{eq:progression}, every position $k$
  in~\eqref{eq:fourpositions}, and every quotient $q$ in the
  box~\eqref{eq:box} at that position,
  \begin{equation} \label{eq:allabove}
    \phi_{t}(y_{m,k} - q) > 1.
  \end{equation}
  Since there are only finitely many pairs $(k,q)$, it will suffice to
  prove this for each pair separately. Note that every element of the
  class $\varepsilon^{k}z_{m} + R$ has nonnegative valuation, because the
  class has nonnegative valuation, and $t > 2$, so
  $\phi_{t}(y_{m,k} - q) \geq |N(y_{m,k} - q)|$.

  For each pair, $|N(y_{m,k} - q)|$ tends to
  $\lambda = \lambda_{\gamma,d,q}$ as $j \to \infty$
  in~\eqref{eq:progression}. If $\lambda > 1$, then $|N(y_{m,k} - q)| > 1$
  for all sufficiently large $m$, and so
  $\phi_{t}(y_{m,k} - q) \geq |N(y_{m,k} - q)| > 1$. By
  Proposition~\ref{prop:limits}, every other pair is one of the six listed
  there. For a pair with $\gamma = \delta$, the valuation of $y_{m,k} - q$
  is the integer $h \geq 1$ of the proposition for every $m$
  in~\eqref{eq:progression}, so
  \[
    \phi_{t}(y_{m,k} - q) = |N(y_{m,k} - q)|\left(\frac{t}{2}\right)^{h}
    \longrightarrow \lambda\left(\frac{t}{2}\right)^{h}
    > \lambda\left(\frac{T}{2}\right)^{h} = \lambda\nu^{-h/4} \geq 1,
  \]
  because $t > T$, $h \geq 1$, and $\lambda^{4} \geq \nu^{h}$. For the
  pair with $\gamma = \bar{\pi}$, the normalized remainder is $r_{m}$,
  whose valuation $h_{m}$ may depend on $m$, but $h_{m} \geq 10$ and
  $t/2 > 1$, so
  \[
    \phi_{t}(r_{m}) \geq |N(r_{m})|\left(\frac{t}{2}\right)^{10}
    \longrightarrow \lambda\left(\frac{t}{2}\right)^{10}
    > \lambda\nu^{-10/4} \geq 1,
  \]
  because $\lambda^{4} \geq \nu^{10}$. In each case
  $\phi_{t}(y_{m,k} - q) > 1$ for all sufficiently large $m$,
  and~\eqref{eq:allabove} follows.

  Finally, let $m \geq m_{0}$ in~\eqref{eq:progression} be so large that
  $z_{m} \notin R$ and that the conclusion of Lemma~\ref{lem:positions}(b)
  holds. Suppose, by way of contradiction, that a normalized remainder $w$
  of $z_{m}$ is reducing at weight $t$. Since $\ve(w) \geq 0$ and $t > 2$,
  we have $|N(w)| \leq \phi_{t}(w) < 1$, and $w \neq 0$ because
  $z_{m} \notin R$. By Lemma~\ref{lem:balance}(a) and~(c) with $P = 2m$,
  there is an integer $g$ such that $\varepsilon^{g}w$ is a balanced
  normalized remainder of a class $\varepsilon^{k}z_{m} + R$ with
  $0 \leq k < 2m$. It has the same norm and valuation as $w$, so
  $\phi_{t}(\varepsilon^{g}w) = \phi_{t}(w) < 1$ and
  $|N(\varepsilon^{g}w)| < 1$. By Lemma~\ref{lem:positions}(a), $k$ is one
  of the eight positions listed there, and by
  Lemma~\ref{lem:positions}(b), $\varepsilon^{g}w = y_{m,k} - q$ with $q$
  in the box~\eqref{eq:box} at the position $k$. If $k$ is not one of the
  four positions~\eqref{eq:fourpositions}, then $3 - k$ is, and by
  Lemma~\ref{lem:reversal} and the remark after Lemma~\ref{lem:positions},
  the normalized remainder $y_{m,3 - k} - (8 - \bar{q})$, with
  $8 - \bar{q}$ in the box at the position $3 - k$, has the same weighted
  norm as $\varepsilon^{g}w$. In either case we obtain a normalized
  remainder at one of the four positions, with a quotient in its box,
  whose weighted norm at $t$ is less than $1$. This
  contradicts~\eqref{eq:allabove}. Hence no normalized remainder of
  $z_{m}$ is reducing at weight $t$. This proves the proposition.
\end{proof}

\begin{remark} \label{rem:familywalk}
  The classes $z_{m} + R$ were constructed using periodic sequences of
  labels modeled on walks in the final graph. We now examine when their
  orbits follow the prescribed labels. Let $t = 9/4$, which exceeds $T$
  by~\eqref{eq:Tbounds}, and let $m \geq m_{0}$ be
  in~\eqref{eq:progression}, with $m_{0}$ as in
  Proposition~\ref{prop:upper} for this $t$. By Lemma~\ref{lem:zm} and
  Lemma~\ref{lem:closing}(b), the element $y_{m,k}$ has valuation at least
  $5$ and lies in the class $\varepsilon^{k}z_{m} + R$, so every
  normalized remainder of $y_{m,k}$ is $\varepsilon^{k}$ times a
  normalized remainder of $z_{m}$ and has the same weighted norm, because
  $N(\varepsilon) = 1$ and $\ve(\varepsilon) = 0$. Hence no normalized
  remainder of $y_{m,k}$ is reducing at weight $9/4$, and by
  Lemma~\ref{lem:X}(d) the point of $y_{m,k}$ in $\R^{2}/\Lambda$ lies in
  $\mathcal{X}$. These points form the orbit of the point of $z_{m}$ in
  $\R^{2}/\Lambda$, which is periodic by Lemma~\ref{lem:closing}(b), so by
  Lemma~\ref{lem:walks} its itinerary at the resolution $2^{10}$ runs
  along a closed walk in the final graph. If every $y_{m,k}$ lies in $D$,
  then $y_{m,k}$ is the representative in $D$ of its point, and
  by~\eqref{eq:ymkrecursion} the labels of this closed walk are the labels
  $\ell_{m,k}$ of~\eqref{eq:periodiclabels}. We verified by computer, for
  $m = 33$ and $m = 97$, that every $y_{m,k}$ lies in $D$, that the
  itinerary runs along a closed walk in the upper component with the
  labels $\ell_{m,k}$, and that it passes through a critical square at
  $k = m + 1$. The proof of Theorem~\ref{thm:main} does not use this
  remark.
\end{remark}

\section{Proofs of the main results} \label{sec:division}
In this section we combine Propositions~\ref{prop:division}
and~\ref{prop:upper} to prove Theorem~\ref{thm:main}. We then show that
the bound $4$ on the valuations in Theorem~\ref{thm:main} cannot be
lowered, we prove Corollary~\ref{cor:integer}, and we show that the
exponent $6$ in Corollary~\ref{cor:integer} is the least possible.
Finally, we show that the Euclidean minimum of $R$ with respect to
$\phi_{T}$ is $1$.

\begin{proof}[Proof of Theorem~\ref{thm:main}]
  By Proposition~\ref{prop:shape}(a), no weight $t$ with
  $0 < t \leq \sqrt{5}$ is in the window. By~\eqref{eq:Tbounds},
  $T > \sqrt{5}$, and by Proposition~\ref{prop:division}, every
  $z \in \mathcal{Z}$ has a normalized remainder that is reducing at
  weight $T$. Hence $\phi_{T}$ is a Euclidean function by
  Proposition~\ref{prop:shape}(b), and
  $(\sqrt{5},T] \subseteq w(\mathfrak{p})$ by
  Proposition~\ref{prop:shape}(d). Now let $t > T$. By
  Proposition~\ref{prop:upper}, some $z_{m} \in \mathcal{Z}$ has no
  normalized remainder that is reducing at weight $t$, and since
  $t > \sqrt{5}$, Proposition~\ref{prop:shape}(b) shows that $\phi_{t}$ is
  not a Euclidean function. Hence $w(\mathfrak{p}) = (\sqrt{5},T]$. That
  is, $\phi_{t}$ is a Euclidean function on $R$ if and only if
  $\sqrt{5} < t \leq T$. The closed form of $T$ was obtained in
  \S\ref{sec:Tvalue}, and the last assertion of the theorem is
  Proposition~\ref{prop:division}. This completes the proof of
  Theorem~\ref{thm:main}.
\end{proof}

\begin{proposition} \label{prop:valuationfour}
  For all sufficiently large $m$ in~\eqref{eq:progression}, the class
  $z_{m} + R$ has a normalized remainder that is reducing at weight $T$,
  and every such normalized remainder has valuation $4$. Consequently, the
  number $4$ in the last assertion of Theorem~\ref{thm:main} cannot be
  replaced by $3$.
\end{proposition}
\begin{proof}
  We first show that $z_{m} + R$ has a normalized remainder that is
  reducing at weight $T$, for every $m$ in~\eqref{eq:progression}. Recall
  the normalized remainder $w_{m} = y_{m,m + 1} - q$ of the class
  $\varepsilon^{m + 1}z_{m} + R$ in~\eqref{eq:wm}, where
  $q = 4 + 4\sqrt{14}$. By Proposition~\ref{prop:limits} (the row with
  $\gamma = \delta$ and $d = 0$), it has valuation $4$. To bound its norm,
  let
  \[
    a_{m} = \frac{\bar{\pi}(\varepsilon^{m} - n)}{\varepsilon^{2m} - 1}.
  \]
  Then~\eqref{eq:wm}, $\pi = \varepsilon\bar{\pi}$, and
  $V = \varepsilon U$ give
  \[
    w_{m} = -U + a_{m},\qquad \bar{w}_{m} = V - \varepsilon a_{m}.
  \]
  We have $29 < \varepsilon < 31$, $0 < \bar{\pi} < 1$, and
  $n = 104051 < 29^{4}$. Since $m \geq 33$, it follows that
  \[
    0 < a_{m}
    < \frac{\varepsilon^{-m}}{1 - \varepsilon^{-2m}}
    < 2\varepsilon^{-m} < \frac{1}{10}.
  \]
  Also, by~\eqref{eq:nu} and~\eqref{eq:nubounds},
  \[
    U^{2} = \frac{\nu}{\varepsilon}
    > \frac{627}{31000} > \frac{1}{100}.
  \]
  Since $U > 0$, we have $0 < a_{m} < U$. Consequently
  \[
    |N(w_{m})| = (U - a_{m})(V - \varepsilon a_{m})
    = \varepsilon(U - a_{m})^{2} < \varepsilon U^{2} = \nu,
  \]
  and so $\phi_{T}(w_{m}) = |N(w_{m})|(T/2)^{4} = |N(w_{m})|/\nu < 1$.
  Multiplication by the unit $\varepsilon^{-m - 1}$ preserves norms and
  valuations, so $\varepsilon^{-m - 1}w_{m}$ is a normalized remainder of
  $z_{m}$ that is reducing at weight $T$.

  Next we apply the proof of Proposition~\ref{prop:upper} with $t = T$.
  For a pair with $\lambda_{\gamma,d,q} > 1$, the normalized remainder
  $y_{m,k} - q$ has absolute norm greater than $1$ for all sufficiently
  large $m$. For the rows of Proposition~\ref{prop:limits} other than the
  row with equality in~\eqref{eq:limitineq}, the inequality
  $\lambda^{4} > \nu^{h}$ is strict, so $\lambda(T/2)^{h} > 1$. The
  weighted norms of the corresponding normalized remainders therefore
  exceed $1$ for all sufficiently large $m$. For the normalized remainder
  $r_{m}$ in the row with $\gamma = \bar{\pi}$ we use
  $\ve(r_{m}) \geq 10$, exactly as in that proof. By
  Lemmas~\ref{lem:balance}, \ref{lem:positions}, and~\ref{lem:reversal},
  every reducing normalized remainder of $z_{m}$ is a unit multiple of one
  of the finitely many normalized remainders $y_{m,k} - q$ with $k$
  in~\eqref{eq:fourpositions} and $q$ in the box~\eqref{eq:box}, or of the
  conjugate of one of them. Thus, for all sufficiently large $m$
  in~\eqref{eq:progression}, every normalized remainder of $z_{m}$ that is
  reducing at weight $T$ is a unit multiple of $w_{m}$ or of
  $\bar{w}_{m}$, and both have valuation $4$. This proves the proposition.
\end{proof}

\begin{proof}[Proof of Corollary~\ref{cor:integer}]
  Let $t_{0} = 2^{7/6}$. Since
  $t_{0}^{6} = 2^{7} = 128 > 125 = (\sqrt{5})^{6}$, we have
  $\sqrt{5} < t_{0}$. Exact comparison in $K$ gives $\nu^{3} < 1/4$. Since
  $T^{4} = 16/\nu$, it follows that
  \[
    T^{12} = \frac{4096}{\nu^{3}} > 16384 = 2^{14} = t_{0}^{12},
  \]
  so $t_{0} < T$. Hence $\phi_{t_{0}}$ is a Euclidean function by
  Theorem~\ref{thm:main}. Furthermore,
  \[
    \phi_{t_{0}}(x)^{6}
    = |N(x)|^{6}\left(\frac{t_{0}}{2}\right)^{6\ve(x)}
    = 2^{\ve(x)}|N(x)|^{6} = \Phi(x)\qquad(x \neq 0).
  \]
  Taking sixth powers preserves every strict inequality between
  nonnegative numbers, so $\Phi$ satisfies condition (iii) with
  $\kappa = 1$. The values of $\Phi$ on $R$ are nonnegative integers, with
  $\Phi(x) = 0$ only for $x = 0$, so conditions (i) and (ii) hold.
  Complete multiplicativity follows from that of the norm and the
  additivity of the valuation. This proves the corollary.
\end{proof}

\begin{remark} \label{rem:leastexponent}
  Among the functions with value $0$ at $0$ and value
  $2^{j\ve(x)}|N(x)|^{k}$ at $x \neq 0$, where $j$ and $k$ are positive
  integers, $6$ is the least exponent $k$ for which the function is a
  Euclidean function on $R$. These functions satisfy conditions (i) and
  (ii) of \S\ref{sec:classes}, because they take nonnegative integer
  values and vanish only at $0$. Let $t = 2^{1 + j/k}$. Then
  \[
    2^{j\ve(x)}|N(x)|^{k} = \phi_{t}(x)^{k}\qquad(x \neq 0).
  \]
  Taking a positive integral power preserves strict inequalities between
  nonnegative real numbers in both directions. Since $t > 1$,
  Lemma~\ref{lem:classes} shows that this function is Euclidean if and
  only if $t$ belongs to the window $w(\mathfrak{p})$.

  If $k \leq 5$, then $j/k \geq 1/5$. Also $2^{6/5} > 9/4$, because
  $2^{16} > 9^{5}$. By~\eqref{eq:Tbounds},
  \[
    2^{1 + j/k} \geq 2^{6/5} > \frac{9}{4} > T,
  \]
  so this weight is not in the window. Thus $k \geq 6$ is necessary, and
  Corollary~\ref{cor:integer} gives an example with $j = 1$ and $k = 6$.
\end{remark}

We now return to the Euclidean minimum of \S\ref{sec:classes}. The classes
$z_{m} + R$ of \S\ref{sec:family} show that it equals $1$ at the upper
endpoint of the window.

\begin{proposition} \label{prop:minimumT}
  The Euclidean minimum of $R$ with respect to $\phi_{T}$ is $1$. Every
  class contains a normalized remainder $w$ with $\phi_{T}(w) < 1$, but
  for every $\kappa < 1$ some class contains no normalized remainder $w$
  with $\phi_{T}(w) < \kappa$.
\end{proposition}
\begin{proof}
  By Theorem~\ref{thm:main}, $\phi_{T}$ is a Euclidean function on $R$. So
  it satisfies condition (iii) with $\kappa = 1$, and by
  Lemma~\ref{lem:classes}(b) every class contains a normalized remainder
  $w$ with $\phi_{T}(w) < 1$.

  Now let $\kappa < 1$. In the proof of
  Proposition~\ref{prop:valuationfour} we showed that the normalized
  remainder $w_{m}$ of~\eqref{eq:wm} has valuation $4$ and that
  $|N(w_{m})| = \varepsilon(U - a_{m})^{2}$, where
  $a_{m} = \bar{\pi}(\varepsilon^{m} - n)/(\varepsilon^{2m} - 1)$
  satisfies $0 < a_{m} < 2\varepsilon^{-m}$. Since $(T/2)^{4} = 1/\nu$ and
  $\nu = \varepsilon U^{2}$, it follows that
  \[
    \phi_{T}(w_{m}) = \frac{|N(w_{m})|}{\nu}
    = \left(1 - \frac{a_{m}}{U}\right)^{2},
  \]
  and this number tends to $1$ as $m \to \infty$
  in~\eqref{eq:progression}. The proof of
  Proposition~\ref{prop:valuationfour} also shows that, for all
  sufficiently large $m$ in~\eqref{eq:progression}, every normalized
  remainder of $z_{m}$ that is reducing at weight $T$ is a unit multiple
  of $w_{m}$ or of $\bar{w}_{m}$, and
  $\phi_{T}(\bar{w}_{m}) = \phi_{T}(w_{m})$. Choose such an $m$ with
  $\phi_{T}(w_{m}) \geq \kappa$. Then every normalized remainder $w$ of
  $z_{m}$ satisfies $\phi_{T}(w) \geq \kappa$, because either $w$ is
  reducing at weight $T$ and $\phi_{T}(w) = \phi_{T}(w_{m})$, or
  $\phi_{T}(w) \geq 1 > \kappa$. Write $z_{m} = a/b$ with $a,b \in R$ and
  $b \neq 0$. For every $q \in R$,
  \[
    \phi_{T}(a - bq) = \phi_{T}(b)\,\phi_{T}(z_{m} - q)
    \geq \kappa\,\phi_{T}(b),
  \]
  so condition (iii) does not hold with this $\kappa$. Hence the Euclidean
  minimum of $R$ with respect to $\phi_{T}$ is $1$. This proves the
  proposition.
\end{proof}

\section{Open problems} \label{sec:questions}
We close with three questions suggested by the proof.

Cavallar and Lemmermeyer asked whether, for a number field with unit rank
at least $1$ and a prime ideal $\mathfrak{p}$ of its ring of integers, the
Euclidean minimum of the weighted norm that gives $\mathfrak{p}$ the
weight $c$ is a continuous function of $c$ on
$[N\mathfrak{p},\infty)$~\cite[Section~8,
p.~353]{CavallarLemmermeyer2000}. Here $N\mathfrak{p}$ is the norm of
$\mathfrak{p}$, and for the prime ideal $\mathfrak{p} = \pi R$ of this
paper the interval is $[2,\infty)$. For $t > 1$ the function $\phi_{t}$
satisfies conditions (i) and (ii) of \S\ref{sec:classes} by
Lemma~\ref{lem:classes}, so the Euclidean minimum of $R$ with respect to
$\phi_{t}$ is defined. By Proposition~\ref{prop:minimumT}, it is $1$ at
the upper endpoint $t = T$ of the window.

\begin{question} \label{q:minimum}
  What is the Euclidean minimum of $R$ with respect to $\phi_{t}$ for
  $\sqrt{5} < t < T$, and is it a continuous function of $t$?
\end{question}

The method of \S\S\ref{sec:graph}--\ref{sec:family} is not special to
$\Z[\sqrt{14}]$, although the details of the finite calculation are.

\begin{question} \label{q:otherfields}
  Can the method of this paper determine the Euclidean windows of the
  prime ideals in other real quadratic fields whose rings of integers are
  principal ideal domains but are not norm-Euclidean?
\end{question}

The number $\nu = UV$ is the absolute value of the product of the
coordinates of the translated point $\eta_{0} - q$, and the critical orbit
is a homoclinic orbit of $(\bar{p},p)$ modulo $R$, by \S\ref{sec:corner}.
Every homoclinic orbit of $(\bar{p},p)$ modulo $R$ has points represented
by $(\bar{p},p + \varepsilon^{k}\delta')$ with $0 \neq \delta' \in R$, and
the critical orbit contains the distinguished point $\eta_{0}$ with
displacement $\delta' = -104051\bar{\pi}$. In this paper, however, the
critical orbit is found through the finite calculation of
\S\ref{sec:weighted}, and we know no way to single out this displacement
among all elements of $R$ without such a calculation.

\begin{question} \label{q:homoclinic}
  Is there a characterization of the displacement $-104051\bar{\pi}$ of
  the homoclinic point $\eta_{0}$ of the critical orbit, and hence of
  $\nu$ and $T$, in terms of the homoclinic orbits of $(\bar{p},p)$ alone,
  without the finite calculation of \S\ref{sec:weighted}?
\end{question}
\section*{Acknowledgment of AI use}
In preparing this paper I used the AI tools ChatGPT (by OpenAI) and Claude
(by Anthropic). I used them to perform literature searches, to suggest
grammatical improvements, to check for typographical errors, to keep track
of the cross-references and labels in the \LaTeX{} source, to convert my
numerical algorithms into the Python programming language, to create the
graphical illustrations, to test many unsuccessful strategies to find the
exact Euclidean window of $\Z[\sqrt{14}]$, and to perform a very large
number of numerical experiments.

\end{document}